\documentclass{amsart}
\usepackage{amssymb}
\usepackage[dvips]{graphicx}
\usepackage{color}
\usepackage{verbatim}
\newcommand{\Z}{\mathbb{Z}}
\newcommand{\C}{\mathbb{C}}
\newcommand{\Q}{\mathbb{Q}}
\newcommand{\R}{\mathbb{R}}
\renewcommand{\H}{\mathbb{H}}
\newcommand{\Id}{\operatorname{Id}}
\newcommand{\End}{\operatorname{End}}
\newcommand{\Vol}{\operatorname{Vol}}
\newcommand{\CS}{\operatorname{CS}}
\newcommand{\Li}{\operatorname{Li}}
\newcommand{\Int}{\operatorname{Int}}
\newcommand{\Tr}{\operatorname{Tr}}
\newcommand{\length}{\operatorname{length}}
\newcommand{\torsion}{\operatorname{torsion}}
\renewcommand{\Re}{\operatorname{Re}}
\renewcommand{\Im}{\operatorname{Im}}
\definecolor{gray1}{rgb}{0.8,0.8,0.8}
\definecolor{gray2}{rgb}{0.95,0.95,0.95}

\newtheorem{theorem}{Theorem}[section]
\newtheorem{lemma}[theorem]{Lemma}
\newtheorem{conjecture}[theorem]{Conjecture}
\theoremstyle{definition}
\newtheorem{definition}[theorem]{Definition}
\newtheorem{example}[theorem]{Example}

\theoremstyle{remark}
\newtheorem{remark}[theorem]{Remark}
\newtheorem*{ack}{Acknowledgments}

\numberwithin{equation}{section}
\begin{document}

\title{An Introduction to the Volume Conjecture}

\author{Hitoshi Murakami}
\address{Department of Mathematics,
Tokyo Institute of Technology,
Oh-okayama, Meguro, Tokyo 152-8551, Japan}
\email{starshea@tky3.3web.ne.jp}

\date{\today}
\begin{abstract}
This is an introduction to the Volume Conjecture and its generalizations for nonexperts.
\par
The Volume Conjecture states that a certain limit of the colored Jones polynomial of a knot would give the volume of its complement.
If we deform the parameter of the colored Jones polynomial we also conjecture that it would also give the volume and the Chern--Simons invariant of a three-manifold obtained by Dehn surgery determined by the parameter.
\par
I start with a definition of the colored Jones polynomial and include elementary examples and short description of elementary hyperbolic geometry.
\end{abstract}
\keywords{volume conjecture, knot, hyperbolic knot, quantum invariant, colored Jones polynomial, Chern--Simons invariant}

\subjclass[2000]{Primary 57M27 57M25 57M50}
\thanks{The author is supported by Grant-in-Aid for Challenging Exploratory Research (21654053)}
\maketitle
\section{Introduction}
In 1995 Kashaev introduced a complex valued link invariant for an integer $N\ge 2$ by using the quantum dilogarithm \cite{Kashaev:MODPLA95} and then he observed that his invariant grows exponentially with growth rate proportional to the volume of the knot complement for several hyperbolic knots \cite{Kashaev:LETMP97}.
He also conjectured that this also holds for any hyperbolic knot, where a knot in the three-sphere is called hyperbolic if its complement possesses a complete hyperbolic structure with finite volume.
\par
In 2001 J.~Murakami and the author proved that Kashaev's invariant turns out to be a special case of the colored Jones polynomial.
More precisely Kashaev's invariant is equal to $J_N\bigl(K;\exp(2\pi\sqrt{-1}/N)\bigr)$, where $J_N(K;q)$ is the $N$-dimensional colored Jones polynomial associated with the $N$-dimensional irreducible representation of the Lie algebra $\mathfrak{sl}(2;\C)$ and $K$ is a knot (\S~\ref{sec:YB}).
We also generalized Kashaev's conjecture to any knot (Volume Conjecture) by using the Gromov norm, which can be regarded as a natural generalization of the hyperbolic volume (\S~\ref{sec:Volume_Conjecture}).
If it is true it would give interesting relations between quantum topology and hyperbolic geometry.
So far the conjecture is proved only for several knots and some links but we have supporting evidence which is described in \S~\ref{sec:VC_proof}.
\par
In the Volume Conjecture we study the colored Jones polynomial at the $N$-th root of unity $\exp(2\pi\sqrt{-1}/N)$.
What happens if we replace $2\pi\sqrt{-1}$ with another complex number?
Recalling that the complete hyperbolic structure of a hyperbolic knot complement can be deformed by using a complex parameter \cite{Thurston:GT3M}, we expect that we can also relate the colored Jones polynomial evaluated at $\exp\bigl((2\pi\sqrt{-1}+u)/N\bigr)$ to the volume of the deformed hyperbolic structure.
At least for the figure-eight knot this is true if $u$ is small \cite{Murakami/Yokota:JREIA2007}.
It is also true (for the figure-eight knot) that we can also get the Chern--Simons invariant, which can be regarded as the imaginary part of the volume, from the colored Jones polynomial (\S~\ref{sec:generalization}).
\par
In general we conjecture that this is also true, that is, for any knot the asymptotic behavior of the colored Jones polynomial would determine the volume of a three-manifold obtained as the deformation associated with the parameter $u$.
\par
The aim of this article is to give an elementary introduction to these conjectures including many examples so that nonexperts can easily understand.
I hope you will join us.
\begin{ack}
The author would like to thank the organizers of the workshop and conference ``Interactions Between Hyperbolic Geometry, Quantum Topology and Number Theory'' held at Columbia University, New York in June 2009.
\par
Thanks are also due to an immigration officer at J.~F.~Kennedy Airport, who knows me by papers, for interesting and exciting discussion about quantum topology.
\end{ack}
\section{Link invariant from a Yang--Baxter operator}
\label{sec:YB}
In this section I describe how we can define a link invariant by using a Yang--Baxter operator.
\subsection{Braid presentation of a link}
An $n$-braid is a collection of $n$ strands that go downwards monotonically from a set of fixed $n$ points to another set of fixed $n$ points as shown in Figure~\ref{fig:braid}.
\begin{figure}[h]
  \includegraphics[scale=0.3]{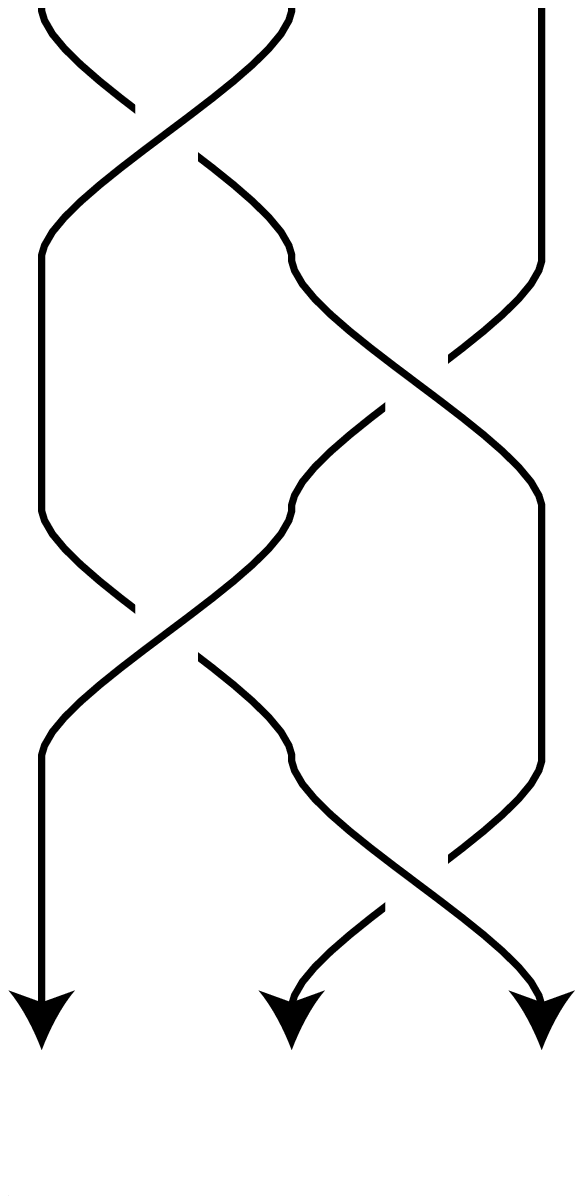}
  \caption{braid}
  \label{fig:braid}
\end{figure}
The set of all $n$-braids makes a group $B_n$ with product of braids $\beta_1$ and $\beta_2$ given by putting $\beta_2$ below $\beta_1$.
It is well known (see for example \cite{Birman:1974}) that $B_n$ is generated by $\sigma_1, \sigma_2,\dots,\sigma_{n-1}$ (Figure~\ref{fig:s_i}) with relations $\sigma_i\sigma_j=\sigma_j\sigma_i$ ($|i-j|>1$) and $\sigma_k\sigma_{k+1}\sigma_k=\sigma_{k+1}\sigma_k\sigma_{k+1}$.
See Figure~\ref{fig:braid_relation} for the latter relation, which is called the braid relation.
\begin{figure}[h]
  \includegraphics[scale=0.3]{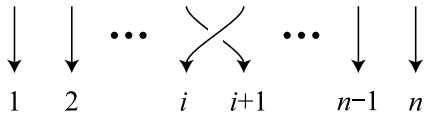}
  \caption{$i$th generator of the $n$-braid group $B_n$}
  \label{fig:s_i}
\end{figure}
\begin{figure}[h]
  \includegraphics[scale=0.3]{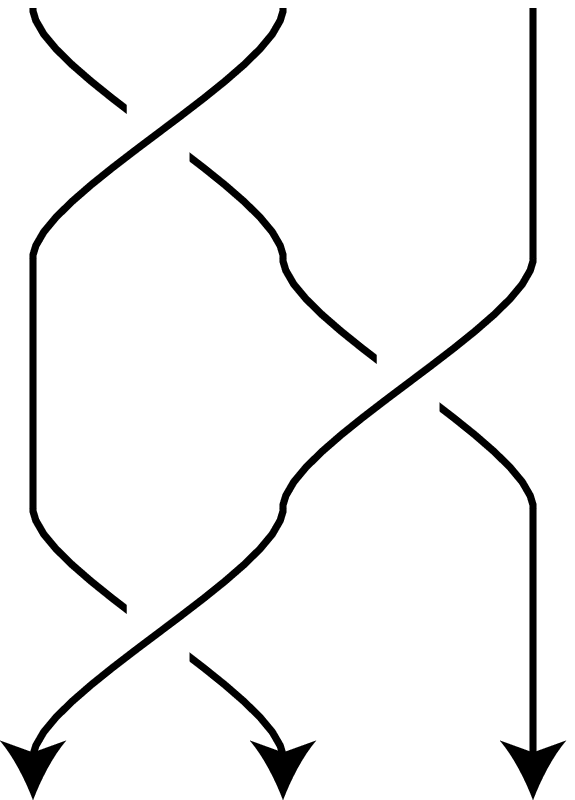}
  \quad\raisebox{12mm}{$=$}\quad
  \includegraphics[scale=0.3]{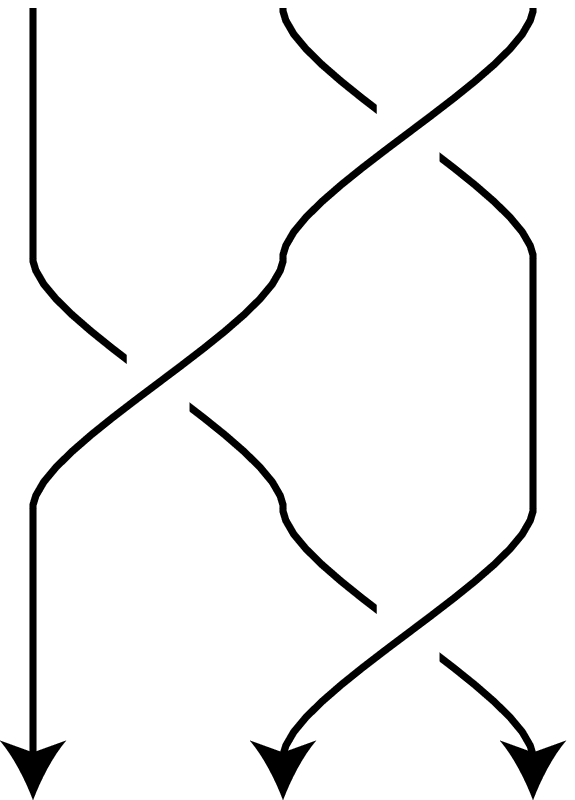}
  \caption{braid relation}
  \label{fig:braid_relation}
\end{figure}
So we have the following group presentation of $B_n$.
\begin{equation}\label{eq:braid_presentation}
  B_n
  =
  \langle
    \sigma_1,\sigma_2,\dots,\sigma_{n-1}
    \mid
    \sigma_i\sigma_j=\sigma_j\sigma_i\,\text{($|i-j|>1$)},
    \sigma_k\sigma_{k+1}\sigma_k=\sigma_{k+1}\sigma_k\sigma_{k+1}
  \rangle
\end{equation}
\par
It is known that any knot or link can be presented as the closure of a braid.
\begin{theorem}[Alexander \cite{Alexander:PRONA1923}]\label{thm:Alexander}
Any knot or link can be presented as the closure of a braid.
\end{theorem}
\begin{figure}[h]
  \includegraphics[scale=0.3]{braid_small.eps}
  \quad\raisebox{20mm}{$\xrightarrow{\text{closure}}$}\quad
  \includegraphics[scale=0.3]{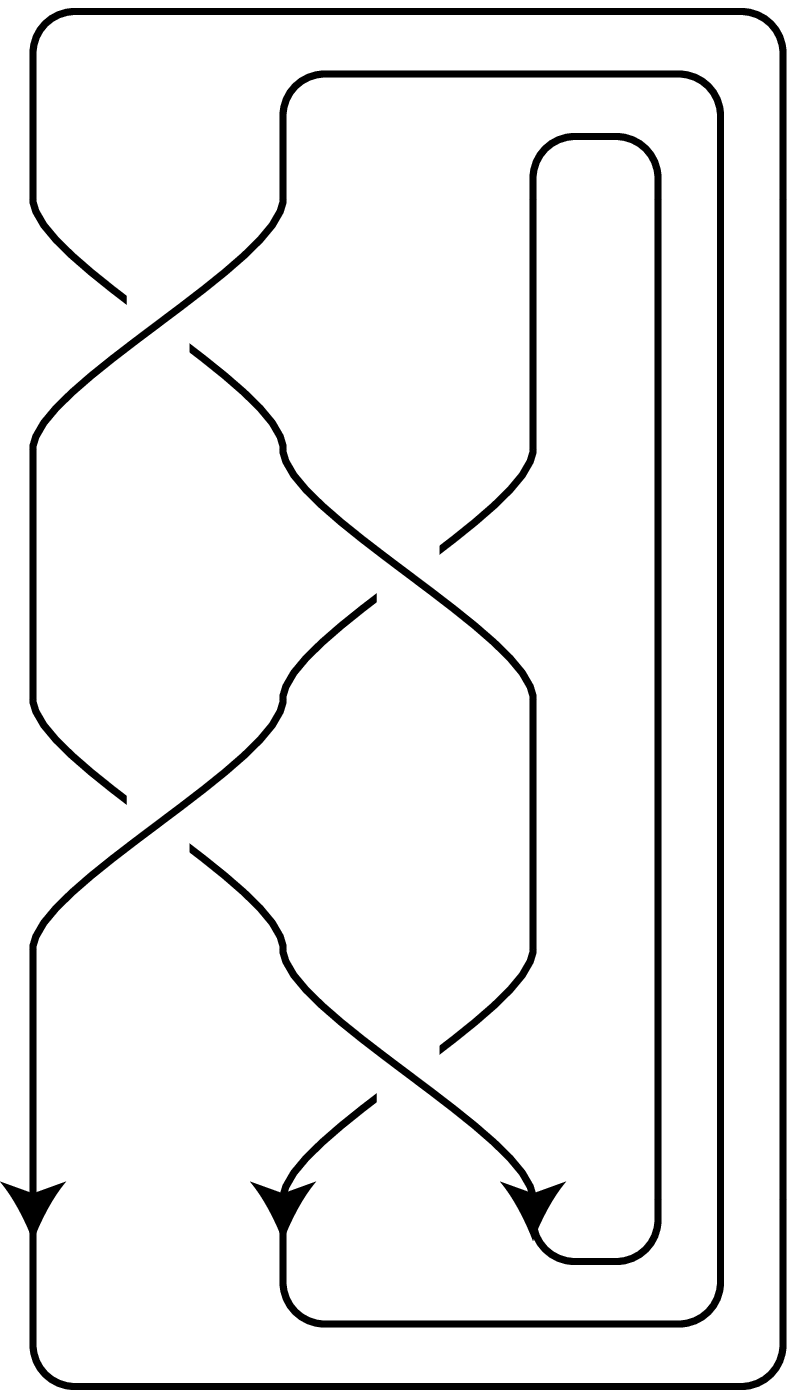}
  \quad\raisebox{20mm}{$=$}\quad
  \raisebox{5mm}{\includegraphics[scale=0.3]{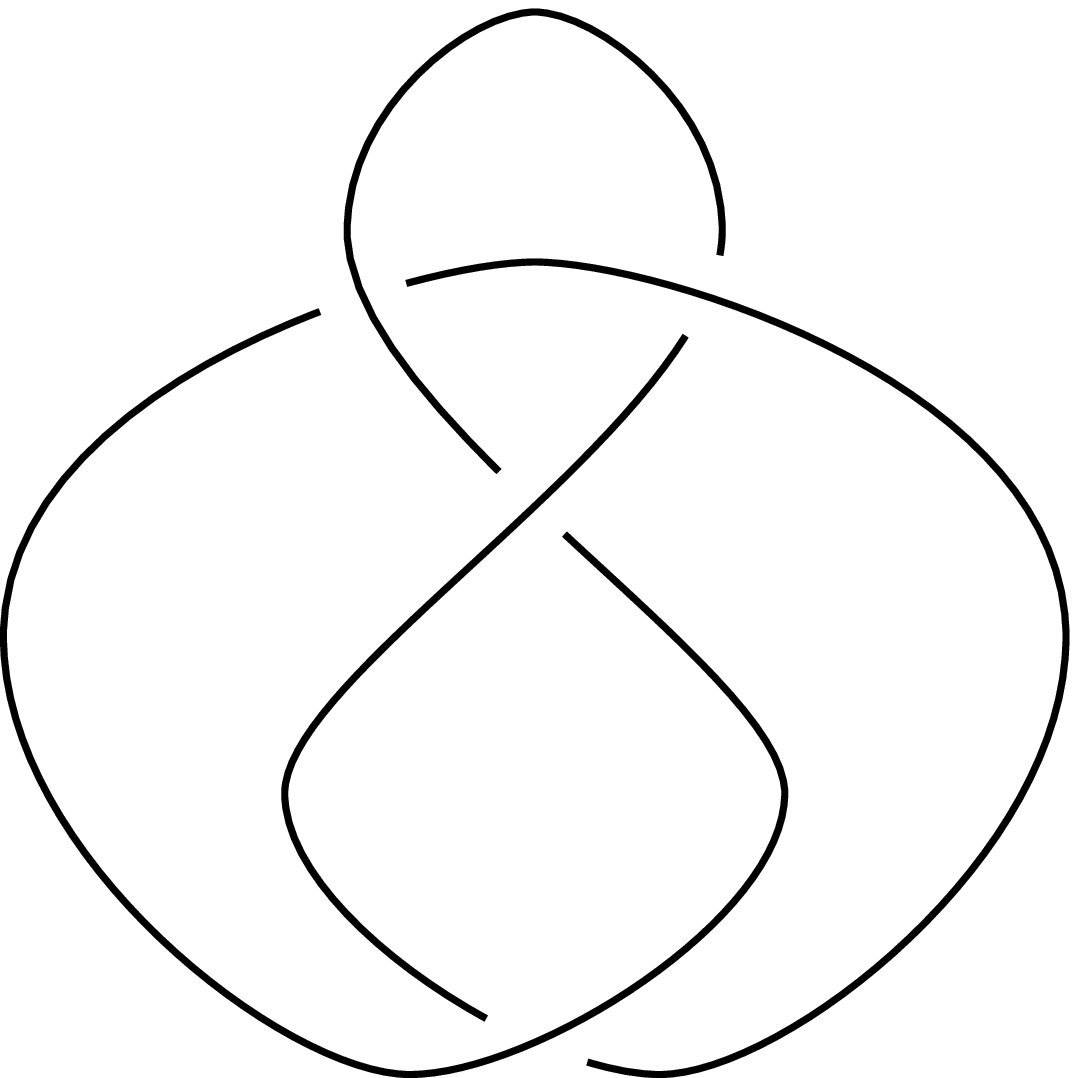}}
  \caption{the figure-eight knot is presented as the closure of a braid}
  \label{fig:braid_presentation}
\end{figure}
Here the closure of an $n$-braid is obtained by connecting the $n$ points on the top with the $n$ points on the bottom without entanglement as shown in the middle picture of Figure~\ref{fig:braid_presentation}.
\par
There are many braids that present a knot or link but if two braids present the same knot or link, they are related by a finite sequence of conjugations and (de-)stabilizations.
In fact we have the following theorem.
\begin{theorem}[Markov \cite{Markov:1936}]\label{thm:Markov}
If two braids $\beta$ and $\beta'$ give equivalent links, then they are related by a finite sequence of conjugations, stabilizations, and destabilizations.
Here a conjugation replaces $\alpha\beta$ with $\beta\alpha$, or equivalently $\beta$ with $\alpha^{-1}\beta\alpha$ {\rm(}Figure~$\ref{fig:conjugate}${\rm)}, a stabilization replaces $\beta\in B_n$ with $\beta\sigma_n^{\pm1}\in B_{n+1}$ {\rm(}Figure~$\ref{fig:stabilization}${\rm)}, a destabilization replaces $\beta\sigma_n^{\pm1}\in B_{n+1}$ with $\beta\in B_n$.
\begin{figure}[h]
  \includegraphics[scale=0.2]{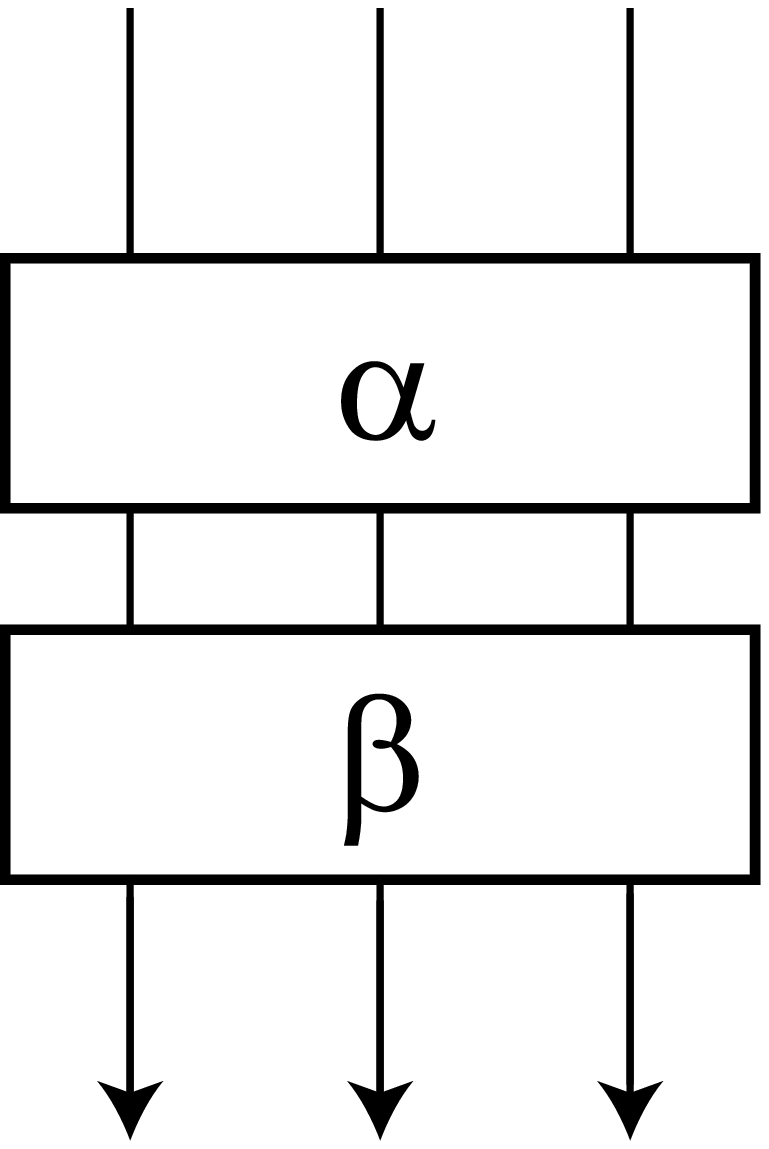}
  \raisebox{13mm}{\quad$\xrightarrow{\text{\rm conjugation}}$\quad}
  \includegraphics[scale=0.2]{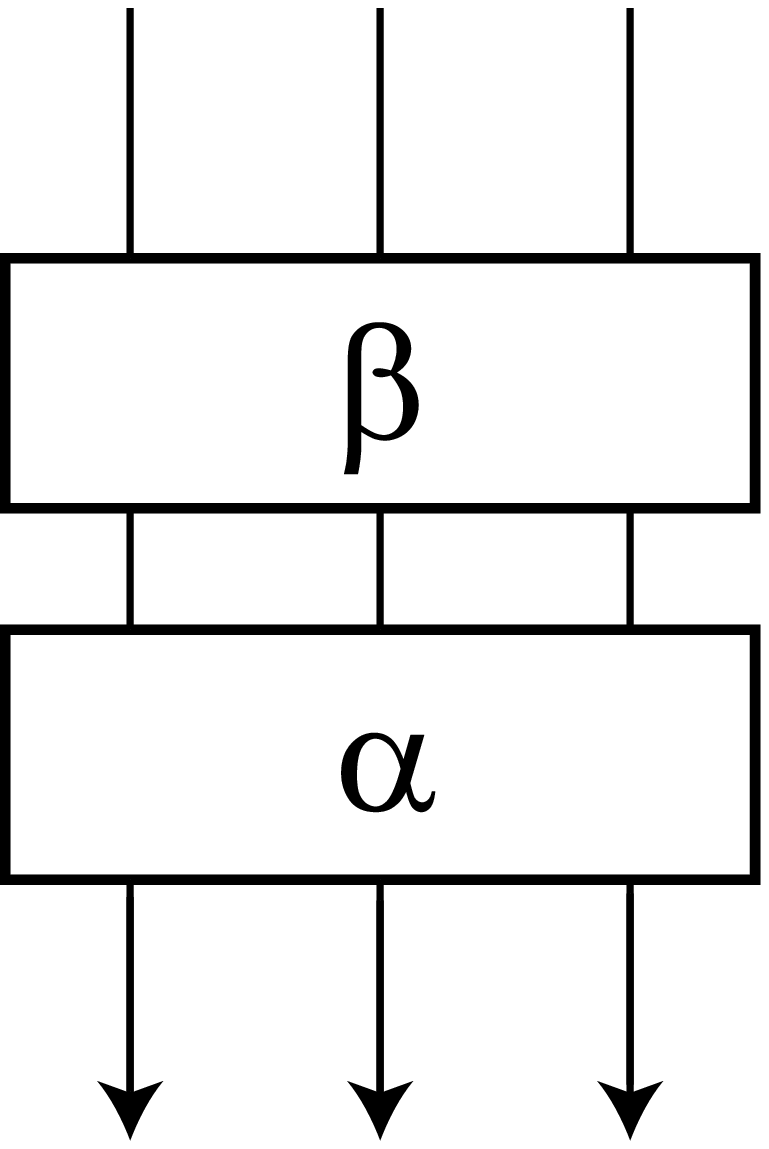}
  \caption{$\alpha\beta$ is conjugate to $\beta$.}
  \label{fig:conjugate}
\end{figure}
\begin{figure}[h]
  \includegraphics[scale=0.2]{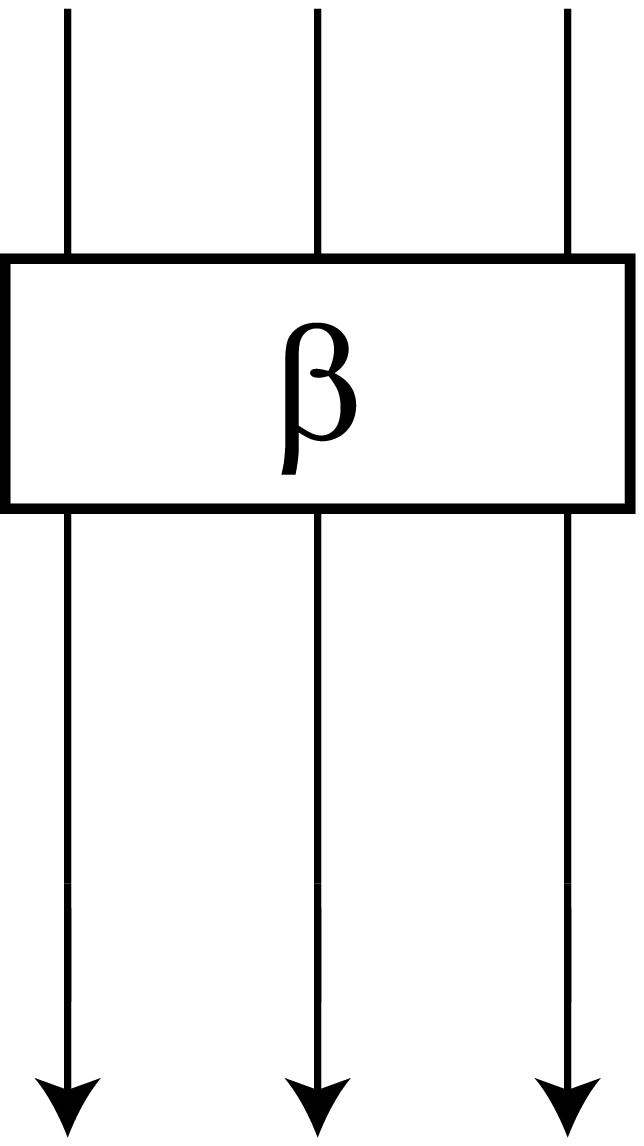}
  \raisebox{13mm}{
    $\begin{array}{c}
      \xrightarrow{\text{\rm stabilization}}\\[2mm]
      \xleftarrow{\text{\rm destabilization}}
     \end{array}$
    }
  \includegraphics[scale=0.2]{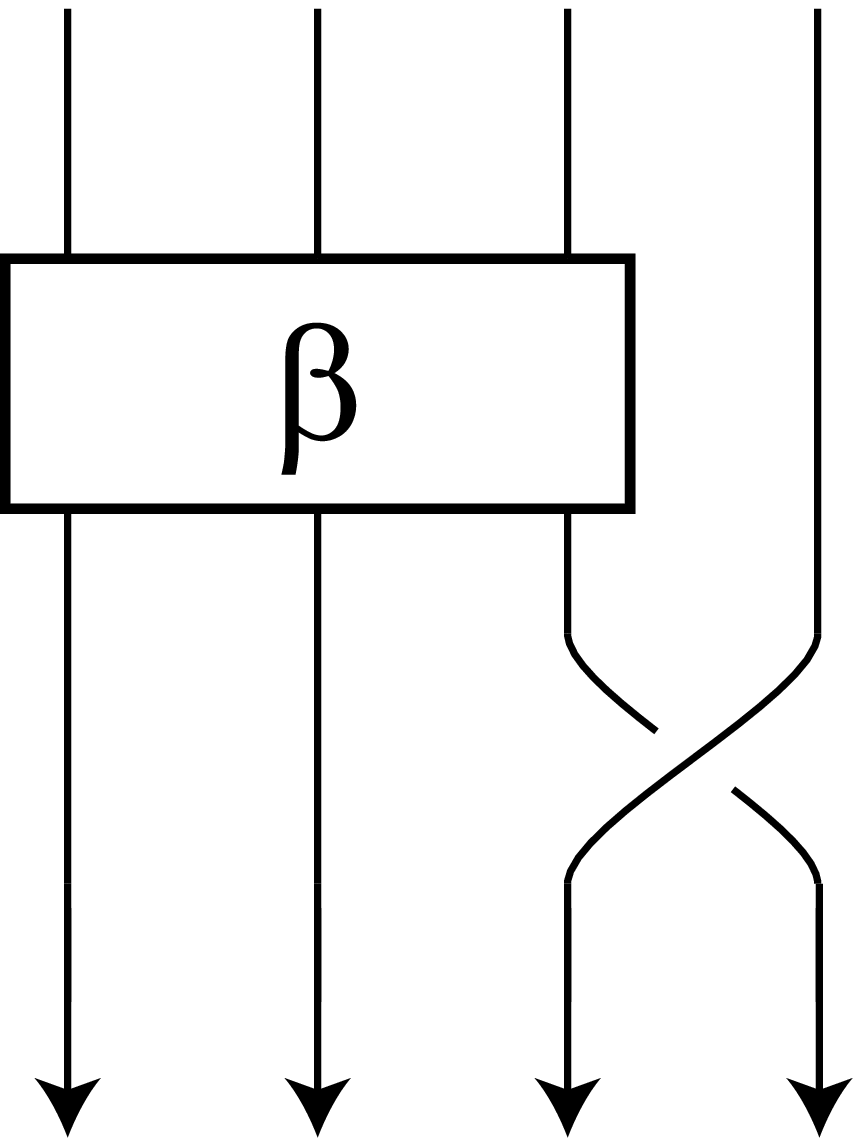}
  \caption{$\beta$ ($\beta\sigma_{n}^{\pm}$, respectively) is stabilized (destabilized, respectively) to $\beta\sigma_{n}^{\pm}$ ($\beta$, respectively).}
  \label{fig:stabilization}
\end{figure}
\end{theorem}
\subsection{Yang--Baxter operator}
Alexander's theorem (Theorem~\ref{thm:Alexander}) and Markov's theorem (Theorem~\ref{thm:Markov}) can be used to define link invariants.
I will follow Turaev \cite{Turaev:INVEM88} to introduce a link invariant derived from a Yang-Baxter operator.
\par
Let $V$ be an $N$-dimensional vector space over $\C$, $R$ an isomorphism from $V\otimes V$ to itself, $m$ an isomorphism from $V$ to itself, and $a$ and $b$ non-zero complex numbers.
\begin{definition}\label{def:YB}
A quadruple $(R,\mu,a,b)$ is called an enhanced Yang--Baxter operator if it satisfies the following:
\begin{enumerate}
\item $(R\otimes\Id_V)(\Id_V\otimes R)(R\otimes\Id_V)
=(\Id_V\otimes R)(R\otimes\Id_V)(\Id_V\otimes R)$,
\item $R(\mu\otimes\mu)=(\mu\otimes\mu)R$,
\item $\Tr_2\bigl(R^{\pm1}(\Id_V\otimes\mu)\bigr)=a^{\pm1}b\Id_V$.
\end{enumerate}
Here $\Tr_k\colon\End(V^{\otimes k})\to\End(V^{\otimes(k-1)})$ is defined by
\begin{equation*}
  \Tr_k(f)(e_{i_1}\otimes e_{i_2}\dots\otimes e_{i_{k-1}})
  :=
  \sum_{j_1,j_2,\dots,j_{k-1},j=0}^{N-1}
  f_{i_1,i_2,\dots,i_{k-1},j}^{j_1,j_2,\dots,j_{k-1},j}
  (e_{j_1}\otimes e_{j_2}\otimes\dots\otimes e_{j_{k-1}}\otimes e_{j}),
\end{equation*}
where $f\in\End(V^{\otimes k})$ is given by
\begin{equation*}
  f(e_{i_1}\otimes e_{i_2}\otimes\dots\otimes e_{i_k})
  =
  \sum_{j_1,j_2,\dots,j_k=0}^{N-1}
  f_{i_1,i_2,\dots,i_k}^{j_1,j_2,\dots,j_k}
  (e_{j_1}\otimes e_{j_2}\otimes\dots\otimes e_{j_k})
\end{equation*}
and $\{e_0,e_1,\dots,e_{N-1}\}$ is a basis of $V$.
\end{definition}
\begin{remark}
The isomorphism $R$ is often called  an $R$-matrix, and the equation {\rm(1)} is known as the Yang--Baxter equation.
\end{remark}
\par
Given an $n$-braid $\beta$, we can construct a homomorphism $\Phi(\beta)\colon V^{\otimes n}\to V^{\otimes n}$ by replacing a generator $\sigma_i$ with $\Id_V^{\otimes(i-1)}\otimes R\otimes\Id_V^{\otimes(n-i-1)}$, and its inverse $\sigma_i^{-1}$ with $\Id_V^{\otimes(i-1)}\otimes R^{-1}\otimes\Id_V^{\otimes(n-i-1)}$ (Figure~\ref{fig:braid_homomorphism}).
\begin{figure}[h]
  \raisebox{-4mm}{\includegraphics[scale=0.3]{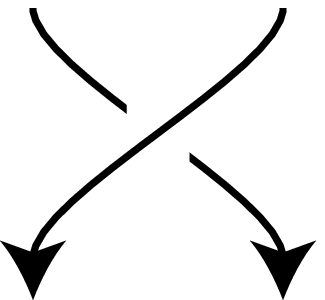}}
  $\Rightarrow$
  \raisebox{-8mm}{\includegraphics[scale=0.3]{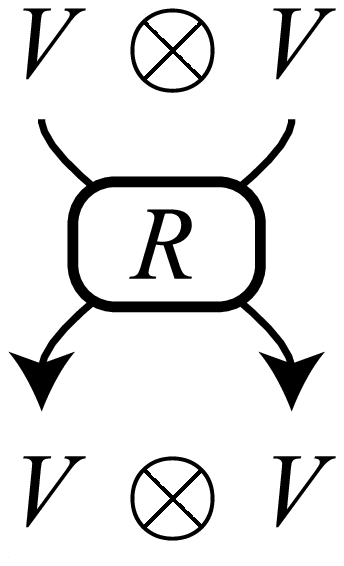}},
  \quad\quad
  \raisebox{-4mm}{\includegraphics[scale=0.3]{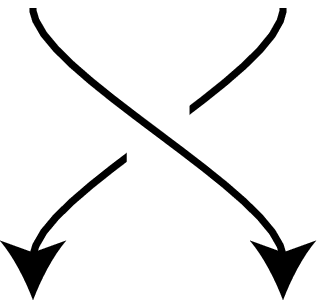}}
  $\Rightarrow$
  \raisebox{-8mm}{\includegraphics[scale=0.3]{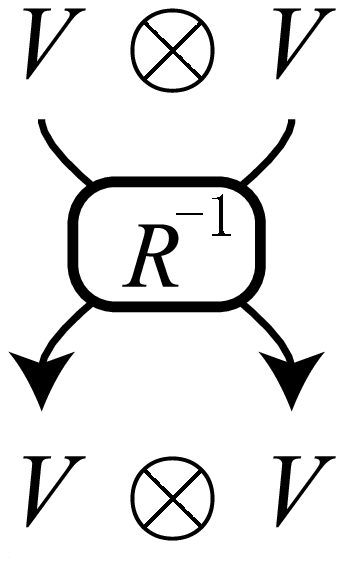}}
  \caption{Replace a generator with the $R$-matrix.}
  \label{fig:braid_homomorphism}
\end{figure}
\begin{example}
For the braid $\sigma_1\sigma_2^{-1}\sigma_1\sigma_2^{-1}$ the corresponding homomorphism is give as follows (Figure~\ref{fig:fig_8_homomorphism}):
\begin{equation*}
  \Phi(\sigma_1\sigma_2^{-1}\sigma_1\sigma_2^{-1})
  =
  (R\otimes\Id_V)(\Id_V\otimes R^{-1})(R\otimes\Id_V)(\Id_V\otimes R^{-1}).
\end{equation*}
\begin{figure}[h]
  \includegraphics[scale=0.3]{braid_small.eps}
  \quad\raisebox{20mm}{$\xrightarrow{\Phi}$}\quad
  \includegraphics[scale=0.3]{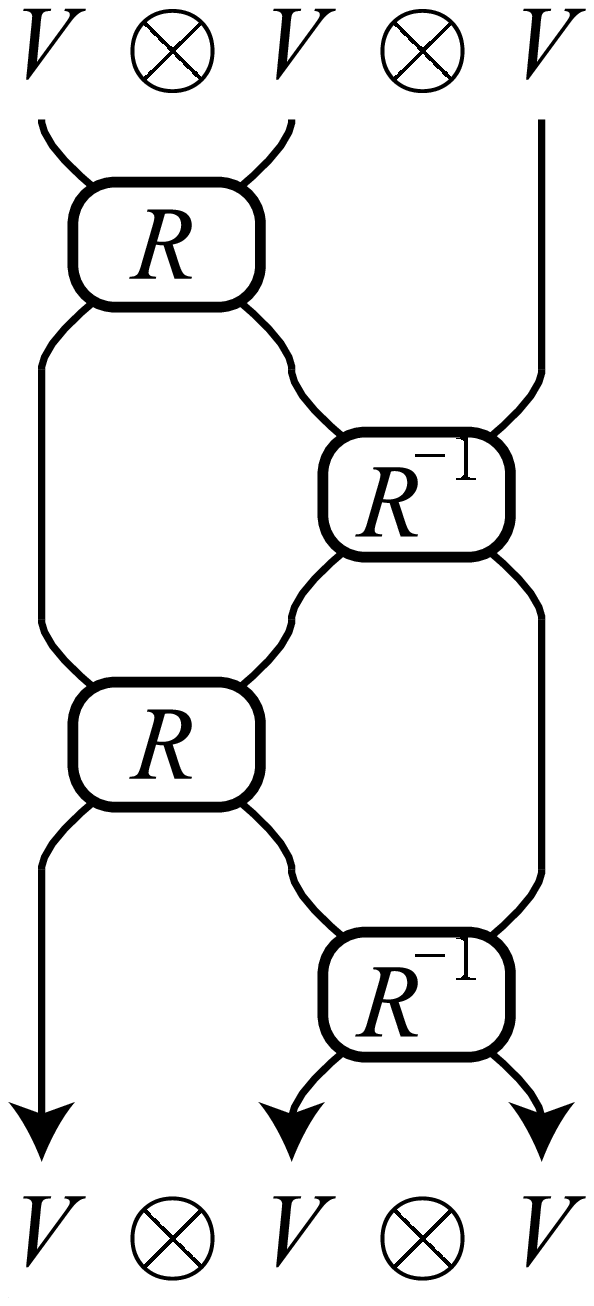}
  \caption{A braid and the corresponding homomorhism}
  \label{fig:fig_8_homomorphism}
\end{figure}
\end{example}
\subsection{Invariant}
Let $(R,\mu,a,b)$ be an enhanced Yang--Baxter operator on an $N$-dimensional vector space $V$.
\begin{definition}
For an $n$-braid $\beta$, we define $T_{(R,\mu,a,b)}(\beta)\in\C$ by the following formula.
\begin{equation*}
  T_{(R,\mu,a,b)}(\beta)
  :=
  a^{-w(\beta)}
  b^{-n}
  \Tr_1
  \Bigl(
    \Tr_2
    \bigl(
      \cdots
      \left(
        \Tr_n\left(\Phi(\beta)\mu^{\otimes n}\right)
      \right)
      \cdots
    \bigr)
  \Bigr),
\end{equation*}
where $w(\beta)$ is the sum of the exponents in $\beta$.
Note that $\Tr_1\colon\End(V)\to\C$ is the usual trace.
\end{definition}
\begin{example}
For the braid $\sigma_1\sigma_2^{-1}\sigma_1\sigma_2^{-1}$, we have
\begin{multline*}
  T_{(R,\mu,a,b)}(\sigma_1\sigma_2^{-1}\sigma_1\sigma_2^{-1})
  \\
  =
  b^{-2}
  \Tr_1\bigl(\Tr_2(\Tr_3(
  (R\otimes\Id_V)(\Id_V\otimes R^{-1})(R\otimes\Id_V)(\Id_V\otimes R^{-1})
  (\mu\otimes\mu\otimes\mu)))\bigr)
\end{multline*}
since $w(\sigma_1\sigma_2^{-1}\sigma_1\sigma_2^{-1})=+1-1+1-1=0$ (Figure~\ref{fig:fig_8_invariant}).
\begin{figure}[h]
  \raisebox{-18mm}{\includegraphics[scale=0.3]{closure_small.eps}}
  \quad$\Rightarrow$\quad$a^{-w(\beta)}b^{-n}\times$
  \raisebox{-21mm}{\includegraphics[scale=0.3]{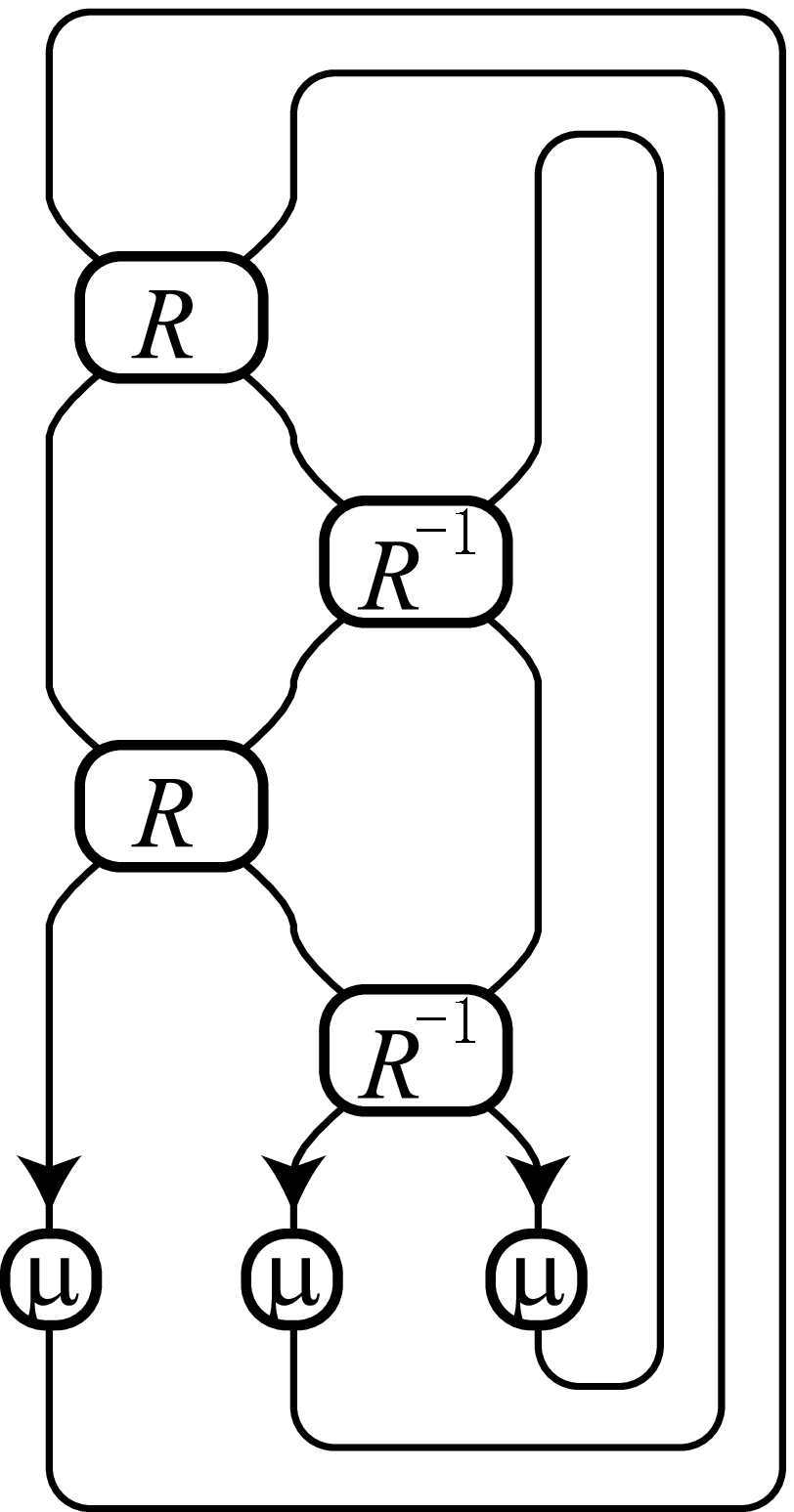}}
  \caption{A braid and its invariant}
  \label{fig:fig_8_invariant}
\end{figure}
\end{example}
\par
We can show that $T_{(R,\mu,a,b)}$ gives a link invariant.
\begin{theorem}[Turaev \cite{Turaev:INVEM88}]
If $\beta$ and $\beta'$ present the same link, then $T_{R,\mu,a,b}(\beta)=T_{R,\mu,a,b}(\beta')$.
\end{theorem}
\begin{proof}[Sketch of a proof]
By Markov's theorem (Theorem~\ref{thm:Markov}) it is sufficient to prove that $T_{R,\mu,a,b}$ is invariant under a braid relation, a conjugation and a stabilization.
\par
The invariance under a braid relation $\sigma_i\sigma_{i+1}\sigma_i=\sigma_{i+1}\sigma_i\sigma_{i+1}$ follows from Figure~\ref{fig:braid_relation_YB}.
Note that the left hand side depicts a braid relation \eqref{eq:braid_presentation} and the right hand side depicts the corresponding Yang--Baxter equation (Definition~\ref{def:YB} (1)).
\begin{figure}[h]
  \raisebox{4mm}{$\underset{\text{\raisebox{-5mm}{braid relation}}}
     {\includegraphics[scale=0.3]{braid_relation1_small.eps}
      \quad\raisebox{12mm}{=}\quad
      \includegraphics[scale=0.3]{braid_relation2_small.eps}}$}
  \raisebox{16mm}{\quad$\xrightarrow{\Phi}$\quad}
  $\underset{\text{\raisebox{-1mm}{Yang--Baxter equation}}}
    {\includegraphics[scale=0.3]{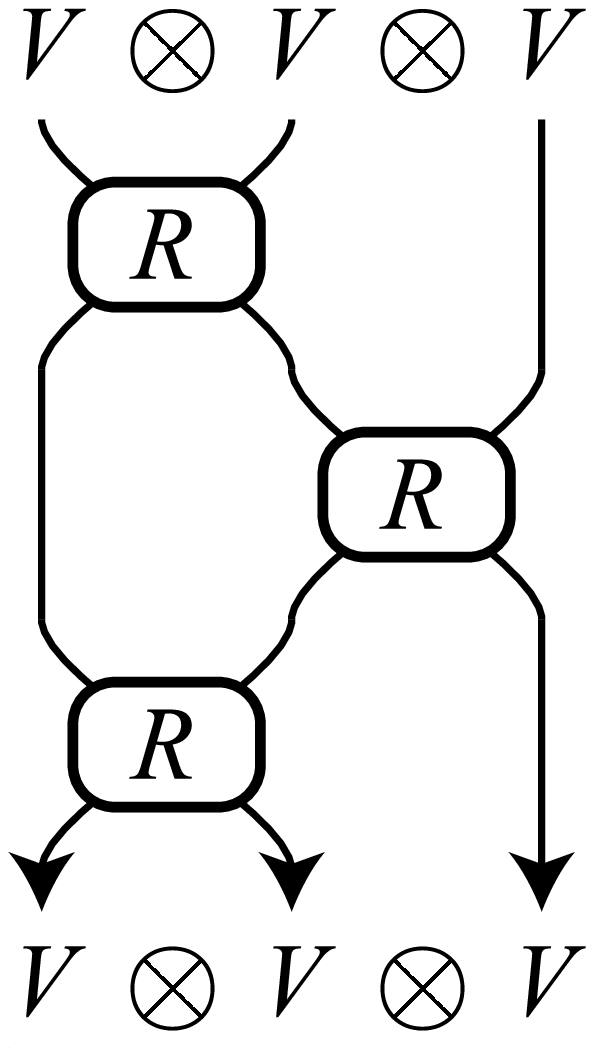}
     \quad\raisebox{16mm}{=}\quad
     \includegraphics[scale=0.3]{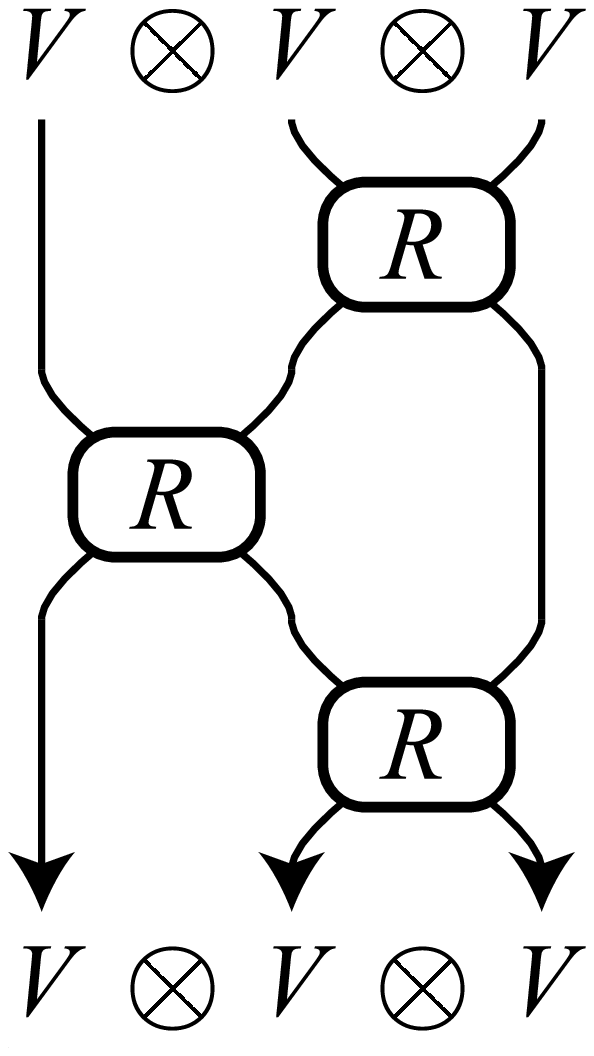}}$
  \caption{braid relation corresponds to the Yang--Baxter equation}
  \label{fig:braid_relation_YB}
\end{figure}
\par
The invariance under a conjugation follows from Figure~\ref{fig:invariance_conjugation}.
The first equality follows since $\Tr_k$ is invariant under a conjugation.
The second equality follows since $(\mu\otimes\mu)R=R(\mu\otimes\mu)$ (Definition~\ref{def:YB} (2)).
Note that the equality $(\mu\otimes\mu)R=R(\mu\otimes\mu)$ means that a pair $\mu\otimes\mu$ can pass through a crossing.
\begin{figure}[h]
  \includegraphics[scale=0.3]{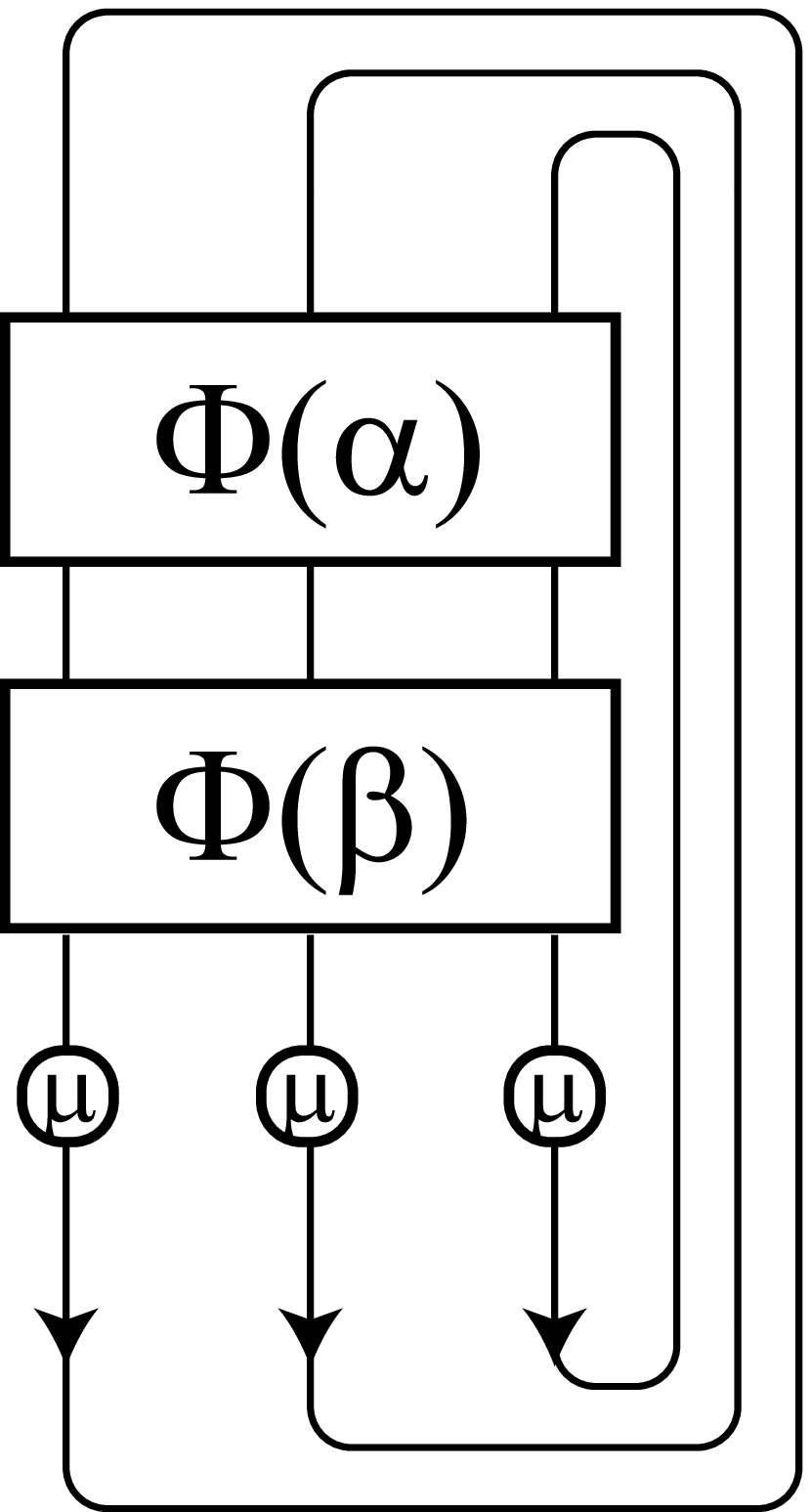}
  \quad\raisebox{24mm}{$=$}\quad
  \includegraphics[scale=0.3]{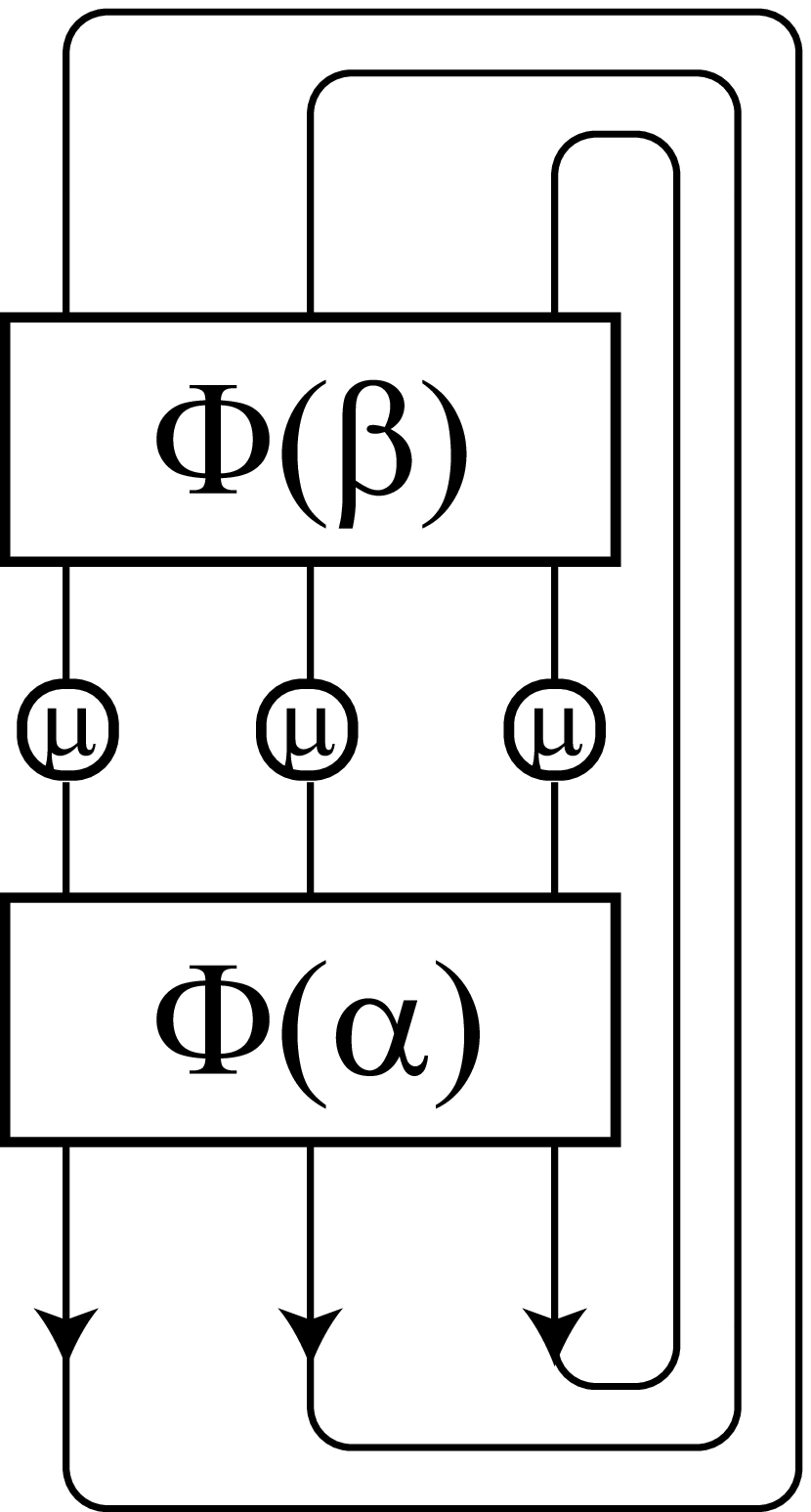}
  \quad\raisebox{24mm}{$=$}\quad
  \includegraphics[scale=0.3]{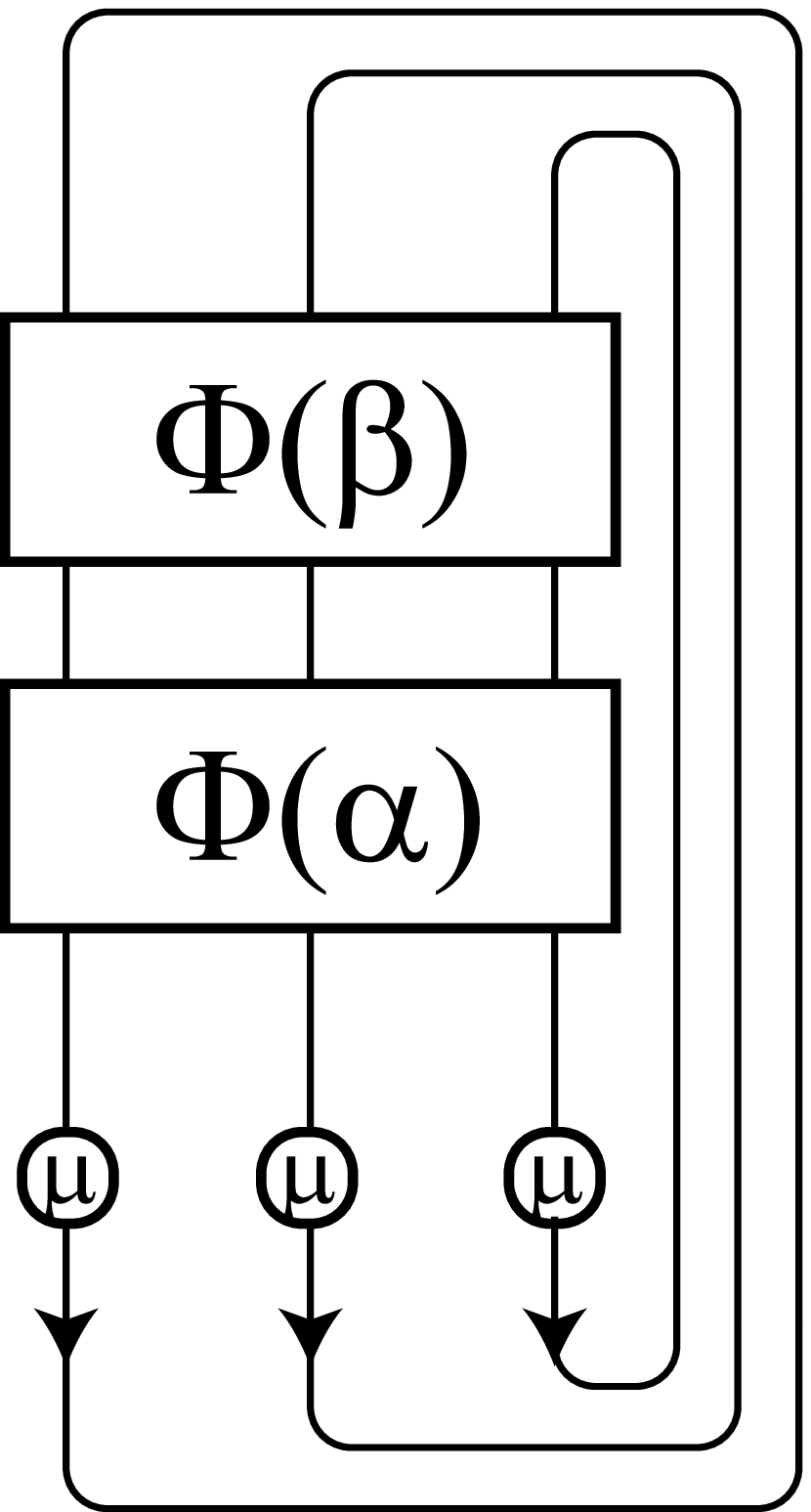}
  \caption{invariance under a conjugation}
  \label{fig:invariance_conjugation}
\end{figure}
\par
To prove the invariance under a stabilization, we first note that if a homomorphism $f\colon V^{\otimes n}\to V^{\otimes n}$ given by
\begin{equation*}
  f(e_{i_1}\otimes\dots\otimes e_{i_n})
  =
  \sum_{j_1,\dots,j_n=0}^{N-1}
  f_{i_1,\dots,i_n}^{j_1,\dots,j_n}(e_{j_1}\otimes\dots\otimes e_{j_n}),
\end{equation*}
then its $n$-fold trace is given by
\begin{equation*}
  \Tr_1(\cdots(\Tr_{n}(f))\cdots)
  =
  \sum_{j_1,\dots,j_n}f_{j_1,\dots,j_n}^{j_1,\dots,j_n}.
\end{equation*}
Therefore if $g$ is a homomorphism $g\colon V\otimes V\to V\otimes V$ given by $g_{l_1,l_2}^{k_1,k_2}$, then we have
\begin{multline*}
  \Tr_1
  \left(\cdots
    \left(\Tr_n
      \left(\Tr_{n+1}
        \bigl(
          (f\otimes\Id_V)(\Id_V^{\otimes(n-1)}\otimes g)
        \bigr)
      \right)
    \right)
  \right)
  \\
  =
  \sum_{j_1,\dots,j_n,k_n,k_{n+1}}
  f_{j_1,\dots,j_{n-1},j_n}^{j_1,\dots,j_{n-1},k_n}
  g_{k_n,k_{n+1}}^{j_n,k_{n+1}},
\end{multline*}
which coincides with the $n$-fold trace of the homomorphism $f\bigl(\Id_V^{\otimes(n-1)}\otimes\Tr_2(g)\bigr)\colon V^{\otimes n}\to V^{\otimes n}$.
\par
Therefore for $\beta\in B_n$ we have
\begin{equation*}
\begin{split}
  &\Tr_1
  \Biggl(
    \Tr_2
    \biggl(
      \cdots
      \Bigl(
        \Tr_n
        \left(
          \Tr_{n+1}\left(\Phi(\beta\sigma_n^{\pm1})\mu^{\otimes(n+1)}\right)
        \right)
      \Bigr)
      \cdots
    \biggr)
  \Biggr)
  \\
  =&
  \Tr_1
  \Biggl(
    \Tr_2
    \biggl(
      \cdots
      \Bigl(
        \Tr_n
        \left(
          \Tr_{n+1}
          \left(
            (\mu^{\otimes n}\Phi(\beta)\otimes\Id_V)
            (\Id_V^{\otimes(n-1)}\otimes R^{\pm1}(\Id_V\otimes\mu))
          \right)
        \right)
      \Bigr)
      \cdots
    \biggr)
  \Biggr)
  \\
  =&
  \Tr_1
  \Biggl(
    \Tr_2
    \biggl(
      \cdots
      \Bigl(
        \Tr_n
        \left(
            (\mu^{\otimes n}\Phi(\beta))
            (\Id_V^{\otimes(n-1)}\otimes\Tr_2(R^{\pm1}(\Id_V\otimes\mu)))
        \right)
      \Bigr)
      \cdots
    \biggr)
  \Biggr)
  \\
  =&
  a^{\pm1}b
  \Tr_1
  \Biggl(
    \Tr_2
    \biggl(
      \cdots
      \Bigl(
        \Tr_n
        \left(
            (\mu^{\otimes n}\Phi(\beta))
        \right)
      \Bigr)
      \cdots
    \biggr)
  \Biggr)
  \\
  =&
  a^{\pm1}b
  \Tr_1
  \Biggl(
    \Tr_2
    \biggl(
      \cdots
      \Bigl(
        \Tr_n
        \left(
            (\Phi(\beta)\mu^{\otimes n})
        \right)
      \Bigr)
      \cdots
    \biggr)
  \Biggr),
\end{split}
\end{equation*}
since $\Tr_2\bigl(R^{\pm1}(\Id_V\otimes\mu\bigr)=a^{\pm1}b\Id_V$ (Definition~\ref{def:YB} (3)) as depicted in Figure~\ref{fig:invariance_stabilization}.
\begin{figure}[h]
  \includegraphics[scale=0.3]{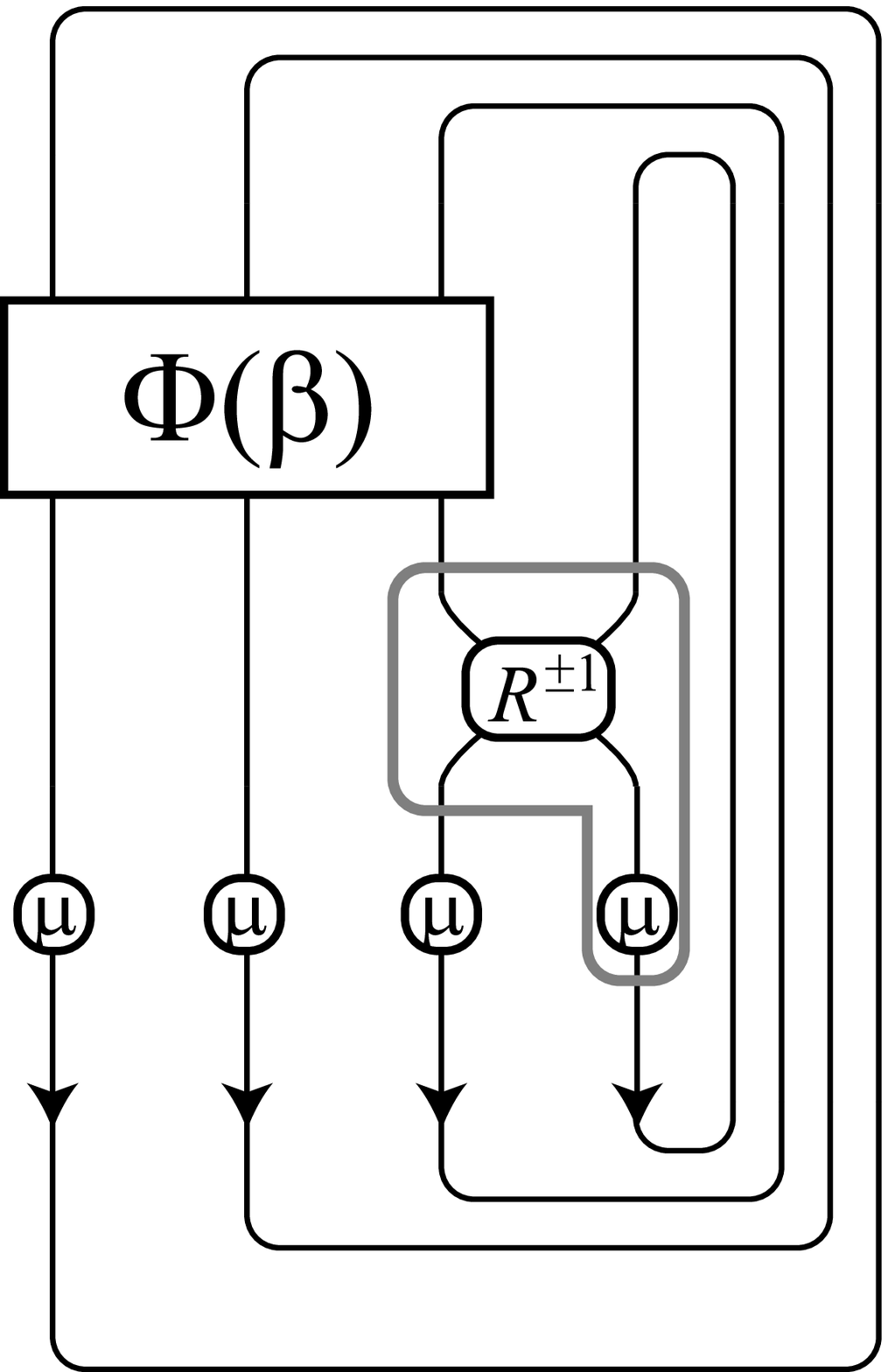}
  \quad\raisebox{24mm}{$=\quad a^{\pm1}b\,$}
  \includegraphics[scale=0.3]{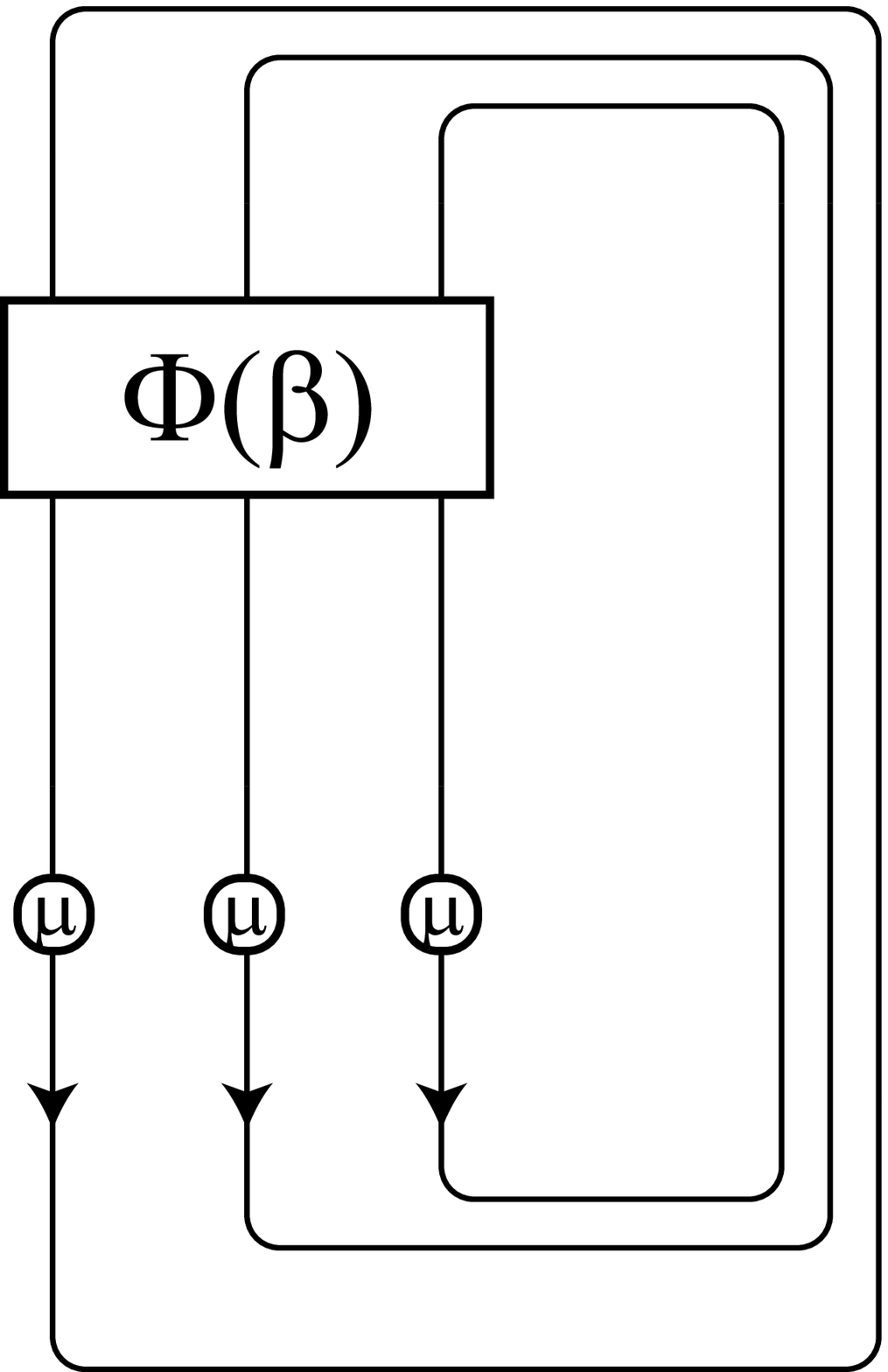}
  \caption{invariance under a stabilization}
  \label{fig:invariance_stabilization}
\end{figure}
Since $w(\beta\sigma_n^{\pm1})=w(\beta)\pm1$, the invariance under a stabilization follows.
\end{proof}
Therefore we can define a link invariant $T_{R,\mu,a,b}(L)$ to be $T_{R,\mu,a,b}(\beta)$ if $L$ is the closure of $\beta$.
\subsection{Quantum $(\mathfrak{g},V)$ invariant}
One of the important ways to construct an enhanced Yang--Baxter operator is to use a quantum group, which is a deformation of a Lie algebra.
\par
Let $\mathfrak{g}$ be a Lie algebra.
Then one can define a quantum group $U_q(\mathfrak{g})$ as a deformation of $\mathfrak{g}$ with $q$ a complex parameter (\cite{Drinfeld:ICM86}, \cite{Jimbo:LETMP1985}).
Given a representation $\rho\colon\mathfrak{g}\to\mathfrak{gl}(V)$ of $\mathfrak{g}$ one can construct an enhanced Yang--Baxter operator.
The corresponding invariant is called the quantum $(\mathfrak{g},V)$ invariant.
For details see \cite{Turaev:INVEM88}.
\par
To define the colored Jones polynomial we need the Lie algebra $\mathfrak{sl}_2(\C)$ and its $N$-dimensional irreducible representation $\rho_N\colon\mathfrak{sl}_2(\C)\to\mathfrak{gl}(V_N)$.
The quantum $(\mathfrak{sl}_2(\C),V_N)$ invariant is called the $N$-dimensional colored Jones polynomial $J_N(L;q)$.
\par
A precise definition is as follows.
\par
Put $V:=\C^N$ and define the $R$-matrix $R\colon V\otimes V\to V\otimes V$ by
\begin{equation*}
  R(e_k\otimes e_l)
  :=
  \sum_{i,j=0}^{N-1}R_{kl}^{ij}e_{i}\otimes e_{j},
\end{equation*}
where
\begin{equation}\label{eq:R}
\begin{split}
  R^{ij}_{kl}
  :=&
  \sum_{m=0}^{\min(N-1-i,j)}
  \delta_{l,i+m}\delta_{k,j-m}
  \frac{\{l\}!\{N-1-k\}!}{\{i\}!\{m\}!\{N-1-j\}!}
  \\
  &
  \times q^{\bigl(i-(N-1)/2\bigr)\bigl(j-(N-1)/2\bigr)-m(i-j)/2-m(m+1)/4},
\end{split}
\end{equation}
with $\{e_0,e_1,\dots,e_{N-1}\}$ is the standard basis of $V$, $\{m\}:=q^{m/2}-q^{-m/2}$ and $\{m\}!:=\{1\}\{2\}\cdots\{m\}$.
Here $q$ is a complex parameter.
A homomorphism $\mu\colon V\to V$ is given by
\begin{equation*}
  \mu(e_j)
  :=
  \sum_{i=0}^{N-1}\mu^i_j e_i
\end{equation*}
with
\begin{equation*}
  \mu^i_j
  :=
  \delta_{i,j}q^{(2i-N+1)/2}.
\end{equation*}
\par
Then it can be shown that $(R,\mu,q^{(N^2-1)/4},1)$ gives an enhanced Yang--Baxter operator.
\begin{definition}[colored Jones polynomial]
For an integer $N\ge1$, put $V:=\C^N$ and define $R$ and $\mu$ as above.
The $N$-dimensional colored Jones polynomial $J_N(L;q)$ for a link $L$ is defined as
\begin{equation*}
  J_N(L;q):=T_{(R,\mu,q^{(N^2-1)/4},1)}(\beta)\times\frac{\{1\}}{\{N\}},
\end{equation*}
where $\beta$ is a braid presenting the link $L$.
\end{definition}
\begin{remark}
Note that $J_N(\text{unknot};q)=1$ since
\begin{equation*}
  \Tr_1(\mu)
  =
  \sum_{i=0}^{N-1}
  q^{(2i-N+1)/2}
  =
  \frac{\{N\}}{\{1\}}.
\end{equation*}
\end{remark}
The two-dimensional colored Jones polynomial $J_2(L;q)$ is (a version) the original Jones polynomial \cite{Jones:BULAM385} as shown below.
\begin{lemma}
Let $L_+$, $L_-$, and $L_0$ be a skein triple, that is, they are the same links except for a small disk as shown in Figure~$\ref{fig:skein_triple}$.
\begin{figure}[h]
  \raisebox{4mm}{$L_+:$}
  \includegraphics[scale=0.3]{crossing_pos_small.eps}\quad,\quad
  \raisebox{4mm}{$L_-:$}
  \includegraphics[scale=0.3]{crossing_neg_small.eps}\quad,\quad
  \raisebox{4mm}{$L_0:$}
  \includegraphics[scale=0.3]{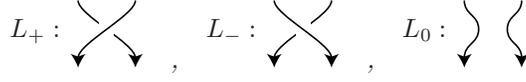}
  \caption{skein triple}
  \label{fig:skein_triple}
\end{figure}
Then we have the following skein relation:
\begin{equation*}
  q J_2(L_+;q)-q^{-1}J_2(L_-;q)=(q^{1/2}-q^{-1/2})J_2(L_0;q).
\end{equation*}
\end{lemma}
\begin{proof}
By the definition, the $R$-matrix is given by
\begin{equation*}
  R
  =
  \begin{pmatrix}
    q^{1/4}&0               &0       &0 \\
    0      &q^{1/4}-q^{-3/4}&q^{-1/4}&0 \\
    0      &q^{-1/4}        &0       &0 \\
    0      &                &0       &q^{1/4}
  \end{pmatrix}
\end{equation*}
with respect to the basis $\{e_0\otimes e_0,e_0\otimes e_1,e_1\otimes e_0,e_1\otimes e_1\}$ of $V\otimes V$, and $\mu$ is given by
\begin{equation*}
  \mu
  =
  \begin{pmatrix}
    q^{-1/2}&0 \\
    0       &q^{1/2}
  \end{pmatrix}
\end{equation*}
with respect to the basis $\{e_0,e_1\}$ of $V$.
\par
Therefore we can easily see that
\begin{equation}\label{eq:skein_R}
  q^{1/4}R-q^{-1/4}R^{-1}
  =
  (q^{1/2}-q^{-1/2})\Id_V\otimes\Id_V.
\end{equation}
\par
Since $L_+$, $L_-$, and $L_0$ can be presented by $n$-braids $\beta\sigma_i\beta'$, $\beta\sigma_i^{-1}\beta'$, and $\beta\beta'$ respectively, we have
\begin{equation*}
\begin{split}
  &
  \left(
    qJ_2(L_+;q)
    -
    q^{-1}J_2(L_-;q)
  \right)
  \times\frac{\{2\}}{\{1\}}
  \\
  =&
  q\times
  q^{-3(w(\beta\beta')+1)/4}
  \Tr_1(\Tr_2(\cdots(\Tr_n(\Phi(\beta\sigma_i\beta')\mu^{\otimes n}))))
  \\
  &\quad
  -
  q^{-1}\times
  q^{-3(w(\beta\beta')-1)/4}
  \Tr_1(\Tr_2(\cdots(\Tr_n(\Phi(\beta\sigma_i^{-1}\beta')\mu^{\otimes n}))))
  \\
  =&
  q^{-3w(\beta\beta')/4}
  \\
  &\times
  \left(
    q^{1/4}
    \Tr_1(\Tr_2(\cdots(\Tr_n(\Phi(\beta\sigma_i\beta')\mu^{\otimes n}))))
    -
    q^{-1/4}
    \Tr_1(\Tr_2(\cdots(\Tr_n(\Phi(\beta\sigma_i^{-1}\beta')\mu^{\otimes n}))))
  \right)
  \\
  =&
  q^{-3w(\beta\beta')/4}
  \\
  &\times
  \left\{
    \Tr_1(\Tr_2(\cdots(\Tr_n(
    \Phi(\beta)
    (\Id_V^{\otimes(i-1)}\otimes q^{1/4}R\otimes\Id_V^{\otimes(n-i-1)})
    \Phi(\beta')
    \mu^{\otimes n}))))
  \right.
  \\
  &\left.
  \qquad
    -
    \Tr_1(\Tr_2(\cdots(\Tr_n(
    \Phi(\beta)
    (\Id_V^{\otimes(i-1)}\otimes q^{-1/4}R^{-1}\otimes\Id_V^{\otimes(n-i-1)})
    \Phi(\beta')\mu^{\otimes n}))))
  \right\}
  \\
  \intertext{(from \eqref{eq:skein_R})}
  =&
  q^{-3w(\beta\beta')/4}
  (q^{1/2}-q^{-1/2})
  \Tr_1(\Tr_2(\cdots(\Tr_n(\Phi(\beta\beta')\mu^{\otimes n}))))
  \\
  =&
  (q^{1/2}-q^{-1/2})J_2(L_0;q)\times\frac{\{2\}}{\{1\}},
\end{split}
\end{equation*}
completing the proof.
\end{proof}
\begin{remark}
The original Jones polynomial $V(L;q)$ satisfies
\begin{equation*}
  q^{-1}V(L_+;q)-qV(L_-;q)
  =
  (q^{1/2}-q^{-1/2})V(L_0;q)
\end{equation*}
\cite[Theorem~12]{Jones:BULAM385}.
So we have $J_2(L;q)=(-1)^{\sharp(L)-1}V(L;q^{-1})$, where $\sharp(L)$ denotes the number of components of $L$.
\end{remark}
\subsection{Example of calculation}\label{subsec:calculation}
Put $\beta:=\sigma_1\sigma_2^{-1}\sigma_1\sigma_2^{-1}$.
Its closure $E$ is a knot called the figure-eight knot (Figure~\ref{fig:braid_presentation}).
We will calculate $J_N(E;q)$.
\par
Instead of calculating $\Tr_1(\Tr_2(\Tr_3(\Phi(\beta)\mu^{\otimes3})))\in\C$, we will calculate $\Tr_2(\Tr_3(\Phi(\beta)(\Id\otimes\mu\otimes\mu)))\in\End(V)$, which is a scalar multiple by Schur's lemma (for a proof see \cite[Lemma~3.9]{Kirby/Melvin:INVEM1991}).
See Figure~\ref{fig:tangle}.
\begin{figure}[h]
  \includegraphics[scale=0.3]{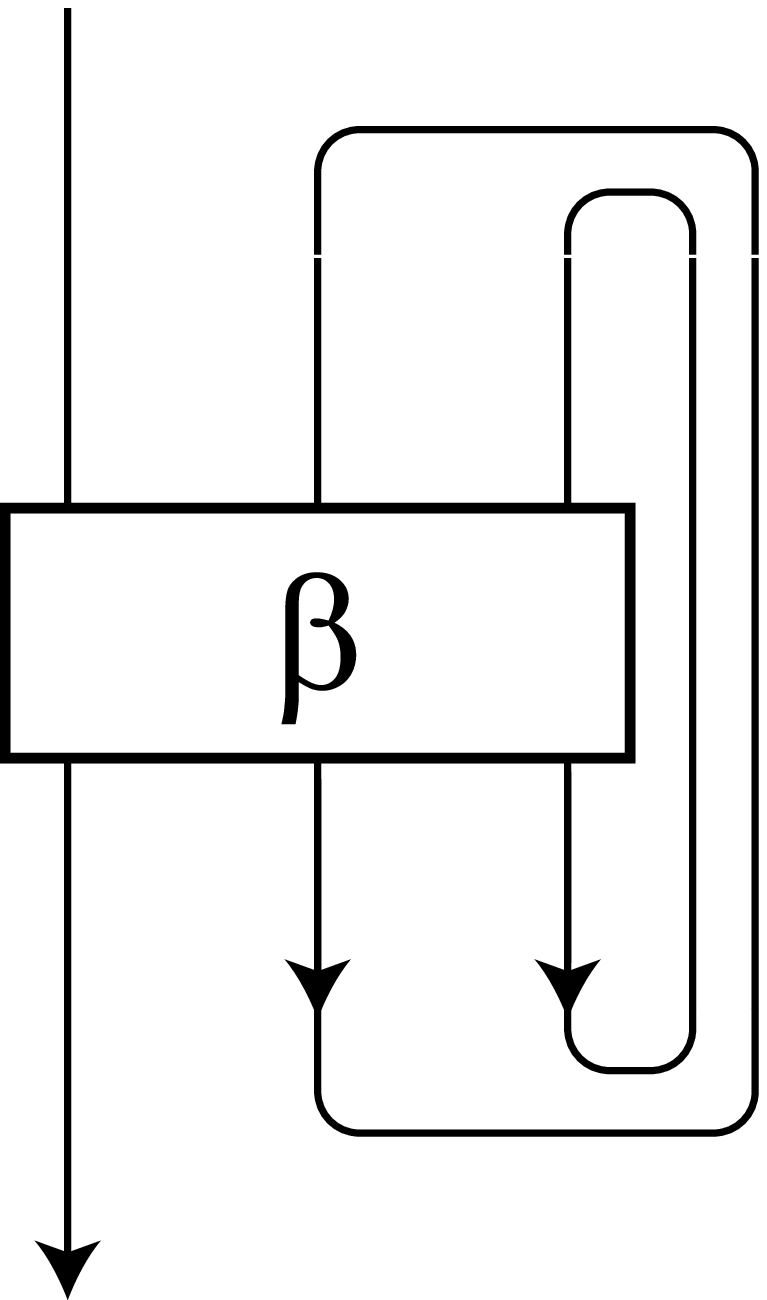}
  \quad\raisebox{19mm}{$\Rightarrow$}\quad
  \includegraphics[scale=0.3]{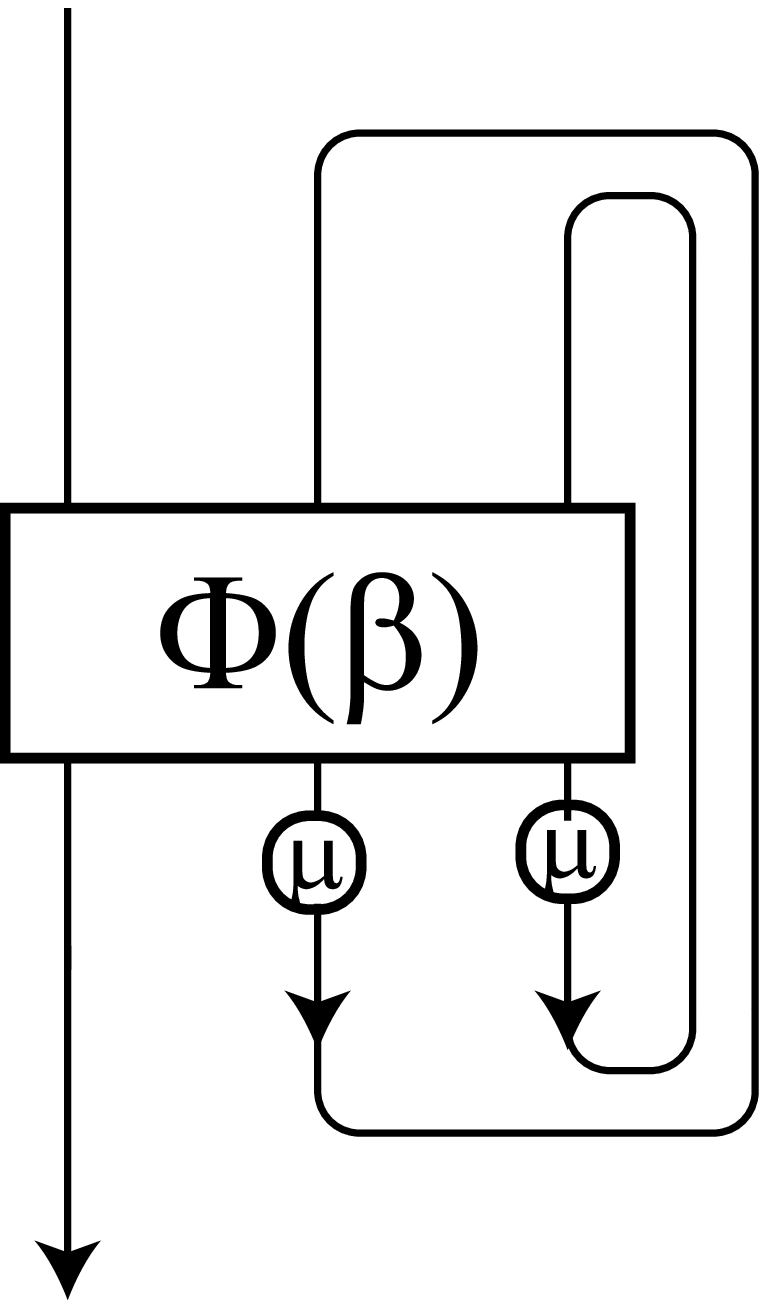}
  \caption{We close all the strings except for the first one.}
  \label{fig:tangle}
\end{figure}
\par
Then $\Tr_1(\Tr_2(\Tr_3(\Phi(\beta)\mu^{\otimes3})))$ coincides with the trace of $\mu$ times the scalar $S$.
Since
\begin{equation*}
\begin{split}
  T_{(R,\mu,q^{(N^2-1)/4},1)}(\beta)
  &=
  q^{-w(\beta)(N^2-1)/4}\Tr_1(\Tr_2(\Tr_3(\Phi(\beta)\mu^{\otimes3})))
  \\
  &=
  q^{-w(\beta)(N^2-1)/4}\Tr_1(S\Id_V)
  \\
  &=
  q^{-w(\beta)(N^2-1)/4}\sum_{i=0}^{N-1}S\, q^{(2i-N+1)/2}
  \\
  &=
  q^{-w(\beta)(N^2-1)/4}\frac{\{N\}}{\{1\}}S,
\end{split}
\end{equation*}
we have $J_N(L;q)=q^{-w(\beta)(N^2-1)/4}S=S$.
\par
We need an explicit formula for the inverse of the $R$-matrix, which is given by
\begin{equation}\label{eq:R_inverse}
\begin{split}
  (R^{-1})^{ij}_{kl}
  =&
  \sum_{m=0}^{\min(N-1-i,j)}
  \delta_{l,i-m}\delta_{k,j+m}
  \frac{\{k\}!\{N-1-l\}!}{\{j\}!\{m\}!\{N-1-i\}!}
  \\
  &\times
  (-1)^{m}q^{-\bigl(i-(N-1)/2\bigr)\bigl(j-(N-1)/2\bigr)-m(i-j)/2+m(m+1)/4}.
\end{split}
\end{equation}
\par
To calculate the scalar $S$, draw a diagram for the braid $\beta$ and close it except for the first string (Figure~\ref{fig:R0}).
\begin{figure}[h]
  \includegraphics[scale=0.3]{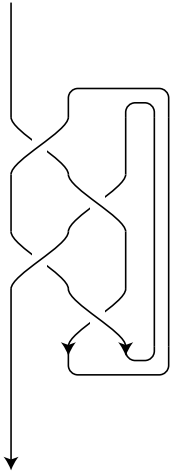}
  \caption{Draw the braid $\sigma_1\sigma_2^{-1}\sigma_1\sigma_2^{-1}$
  and close it except for the left-most one.}
  \label{fig:R0}
\end{figure}
Fix a basis $\{e_0,e_1,\dots,e_{N-1}\}$ of $\C^N$.
\par
Label each arc with a non-negative integer $i$ less than $N$, which corresponds to a basis element $e_i$, where our braid diagram is divided into arcs by crossings so that at each crossing four arcs meet.
Since the homomorphism $\Tr_2(\Tr_1(\Phi(\beta)\otimes2))$ is a scalar multiple, we choose any basis for the first (top-left) arc of Figure~\ref{fig:R0} and calculate the scalar.
For simplicity we choose $e_{N-1}$ and so we label the first arc with $N-1$ (Figure~\ref{fig:R1}).
\begin{figure}[h]
  \includegraphics[scale=0.3]{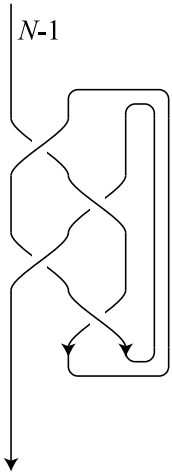}
  \caption{Label the first (top-left) arc with $N-1$.}
  \label{fig:R1}
\end{figure}
\par
Recall that we will associate the $R$-matrix or its inverse with each crossing as follows.
\begin{center}
  \raisebox{-9mm}{\includegraphics[scale=0.3]{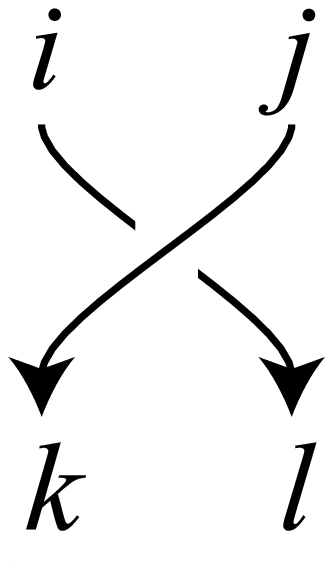}}
  \quad$\Rightarrow$\quad
  $R^{ij}_{kl}$\, ,\quad
  \raisebox{-9mm}{\includegraphics[scale=0.3]{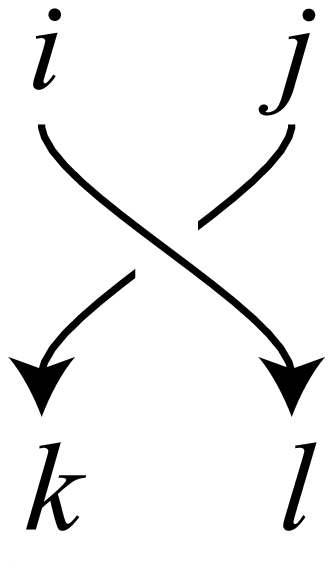}}
  \quad$\Rightarrow$\quad
  $(R^{-1})^{ij}_{kl}$
\end{center}
\par
Therefore we will label the other arcs following the following two rules:
\begin{enumerate}
\item[(i).]
At a positive crossing, the top-left label is less than or equal to the bottom-right label, the top-right label is greater than or equal to the bottom-left label, and their differences coincide (see \eqref{eq:R}).
\begin{center}
  \raisebox{-9mm}{\includegraphics[scale=0.3]{crossing_pos_ijkl_small.eps}}
  $: i+j=k+l$, $l\ge i$, $k\le j$,
\end{center}
\item[(ii).]
At a negative crossing, the top-left label is greater than or equal to the bottm-right label, the top-right label is less than or equal to the bottm-left label, and their differences coincide (see \eqref{eq:R_inverse}).
\begin{center}
  \raisebox{-9mm}{\includegraphics[scale=0.3]{crossing_neg_ijkl_small.eps}}
  $: i+j=k+l$, $l\le i$, $k\ge j$.
\end{center}
\end{enumerate}
\par
From Rule (i), the next arc should be labeled with $N-1$, and the difference at the top crossing is $0$ (Figure~\ref{fig:R2}).
This is why we chose $N-1$ for the label of the first arc.
\begin{figure}[h]
  \includegraphics[scale=0.3]{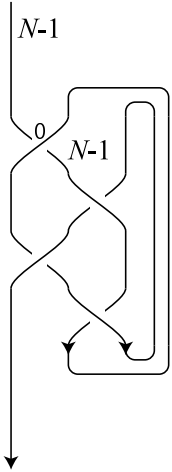}
  \caption{The label of the next arc should be $N-1$, and the difference
  at the top crossing is $0$}
  \label{fig:R2}
\end{figure}
\par
Label the top-right arc with $i$ with $0\le i\le N-1$ (Figure~\ref{fig:R3}).
\begin{figure}[h]
  \includegraphics[scale=0.3]{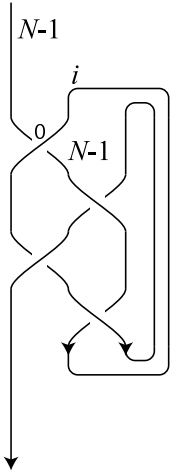}
  \caption{Label the top-right arc with $i$ ($0\le i\le N-1$).}
  \label{fig:R3}
\end{figure}
\par
The label of the left-middle arc should be $i$ since the difference at the top crossing is $0$.
\begin{figure}[h]
  \includegraphics[scale=0.3]{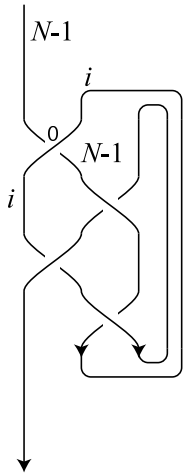}
  \caption{The label of the left-middle arc should be $i$
  since the difference at the top crossing is $0$.}
  \label{fig:R4}
\end{figure}
\par
Label the arcs indicated in Figure~\ref{fig:R6} with $j$ and $k$ with $0\le j\le N-1$ and $0\le k\le N-1$.
Then the difference at the second top crossing is $N-k-1$.
Therefore the arc between the second top crossing and the second bottom crossing should be labeled with $N-k+j-1$ from Rule (ii) (Figure~\ref{fig:R7}).
Since the label should be less than $N$, we have $j-k\le0$.
\begin{figure}[h]
  \includegraphics[scale=0.3]{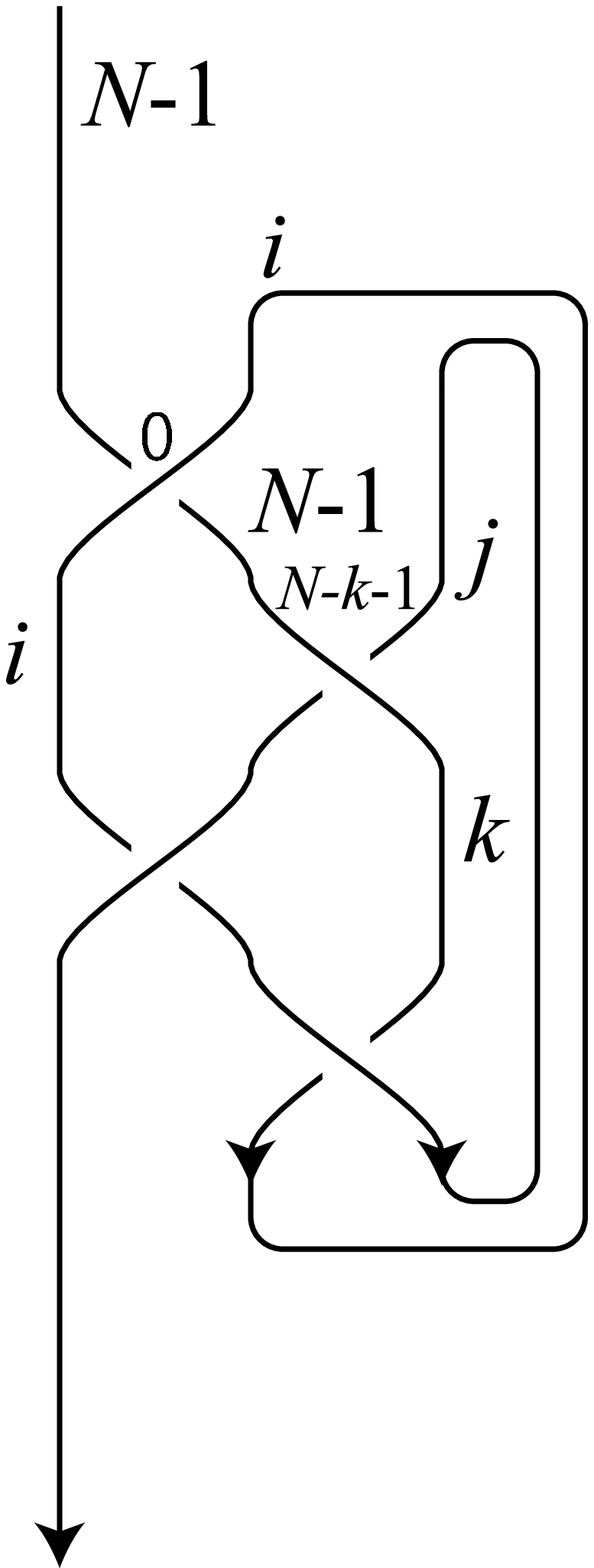}
  \caption{Choose $j$ and $k$.
  Then the difference at the second top crossing is $N-k-1$.}
  \label{fig:R6}
\end{figure}
\begin{figure}[h]
  \includegraphics[scale=0.3]{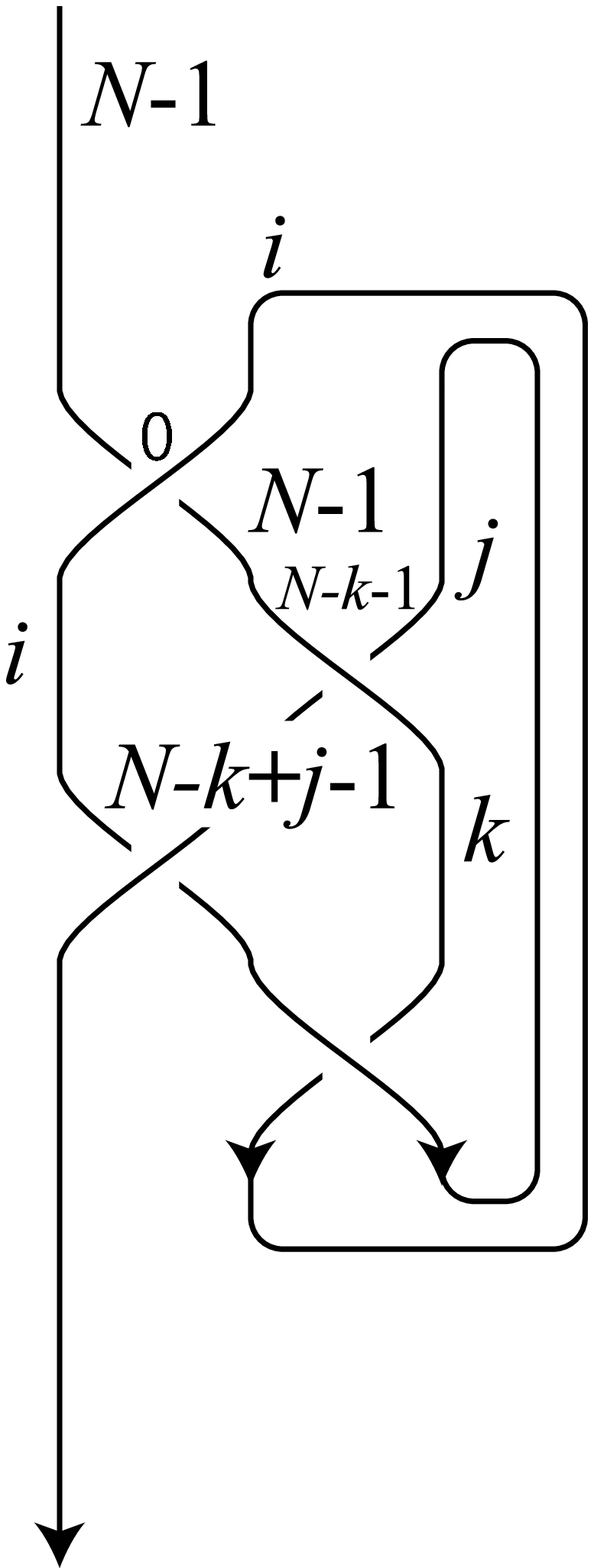}
  \caption{The label of the arc between the second top arc and
  the second bottom arc should be $N-k+j-1$.
  Note that $j-k\le0$.}
  \label{fig:R7}
\end{figure}
\par
Look at the bottom-most crossing and apply Rule (ii).
We see that $i\ge k$ and the difference at the bottom-most crossing is $i-k$.
So the label of the arc between the second bottom crossing and the bottom-most crossing is $i+j-k$ (Figure~\ref{fig:R8}).
We also see that the label of the bottom-most arc is $N-1$ as we expected.
\begin{figure}[h]
  \includegraphics[scale=0.3]{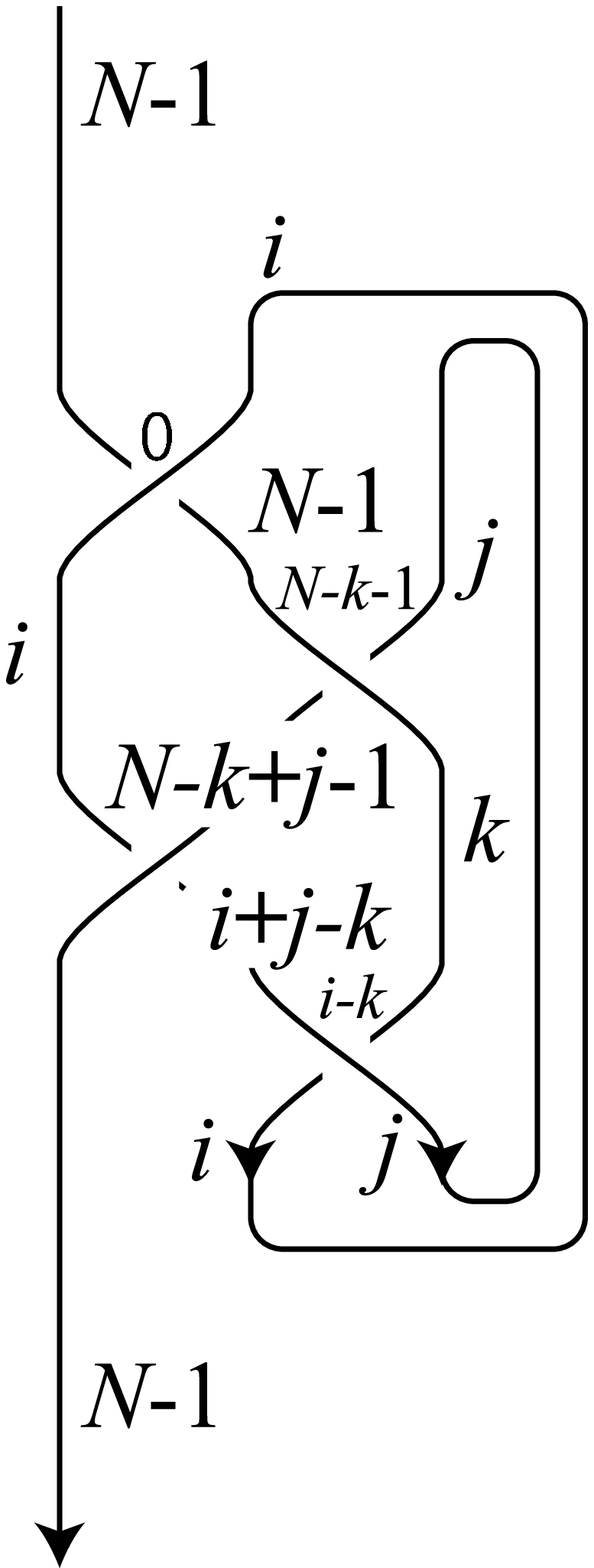}
  \caption{The label of the arc between the second bottom crossing and
  the bottom-most crossing should be $i+j-k$.
  Note that $j-k\ge0$}
  \label{fig:R8}
\end{figure}
\par
However from Rule (i) we have $j-k\ge0$ and so $k=j$.
Therefore we finally have the labeling as indicated in Figure~\ref{fig:R9}.
\begin{figure}[h]
  \includegraphics[scale=0.3]{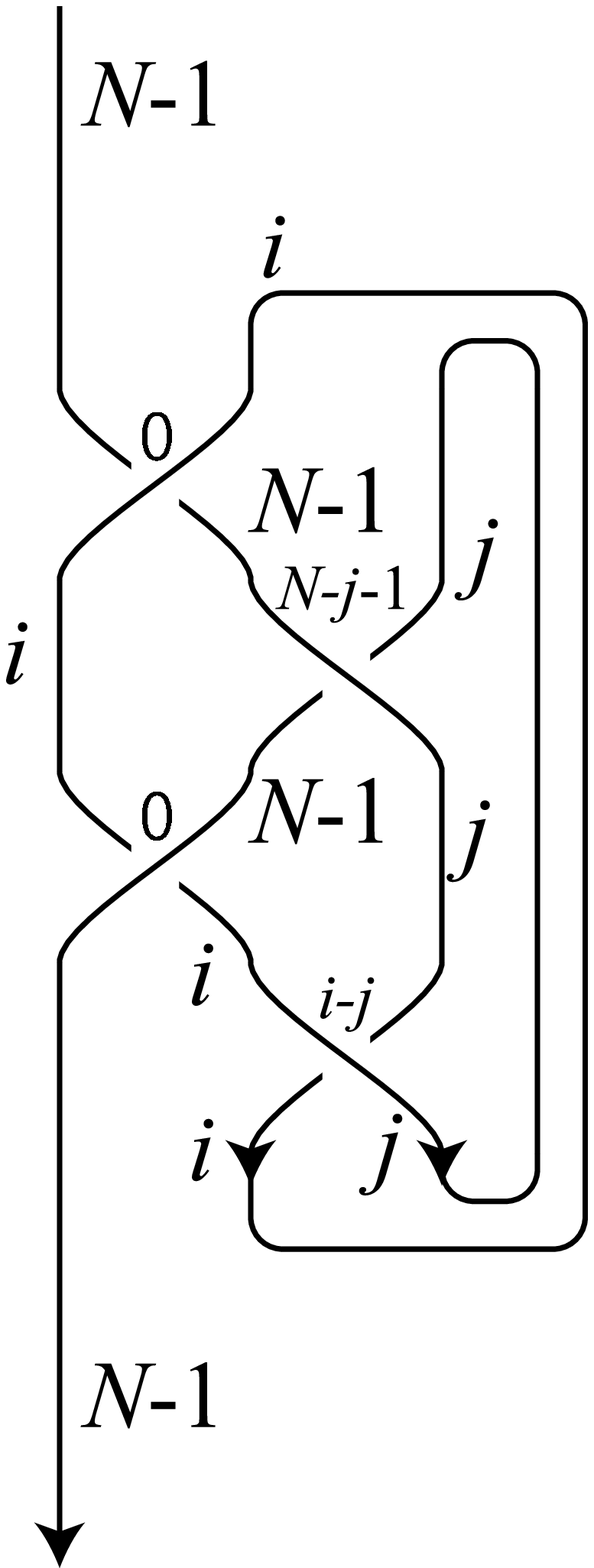}
  \caption{The label $k$ should coincide with $i$.}
  \label{fig:R9}
\end{figure}
\par
Now we can calculate the colored Jones polynomial.
We have
\begin{equation*}
\begin{split}
  J_N(E;q)
  &=
  \sum_{i\ge j}
  R^{N-1,i}_{i,N-1}\,
  (R^{-1})^{N-1,j}_{N-1,j}\,
  R^{i,N-1}_{N-1,i}\,
  (R^  {-1})^{i,j}_{i,j}\,
  \mu^{j}_{j}\,
  \mu^{i}_{i}
  \\
  &=
  \sum_{i\ge j}
  (-1)^{N-1+i}
  \frac{\{N-1\}!\{i\}!\{N-1-j\}!}{(\{j\}!)^2\{i-j\}!\{N-1-i\}!}
  \\
  &\phantom{=\sum_{i\ge j}}\times q^{(-i-i^2-2ij-2j^2+3N+6Ni+2Nj-3N^2)/4}.
\end{split}
\end{equation*}
\par
In this formula we need two summations.
To get a formula involving only one summation we regard the figure-eight knot $E$ as the closure of a tangle as shown in Figure~\ref{fig:fig8_calc}.
\begin{figure}[h]
  \includegraphics[scale=0.3]{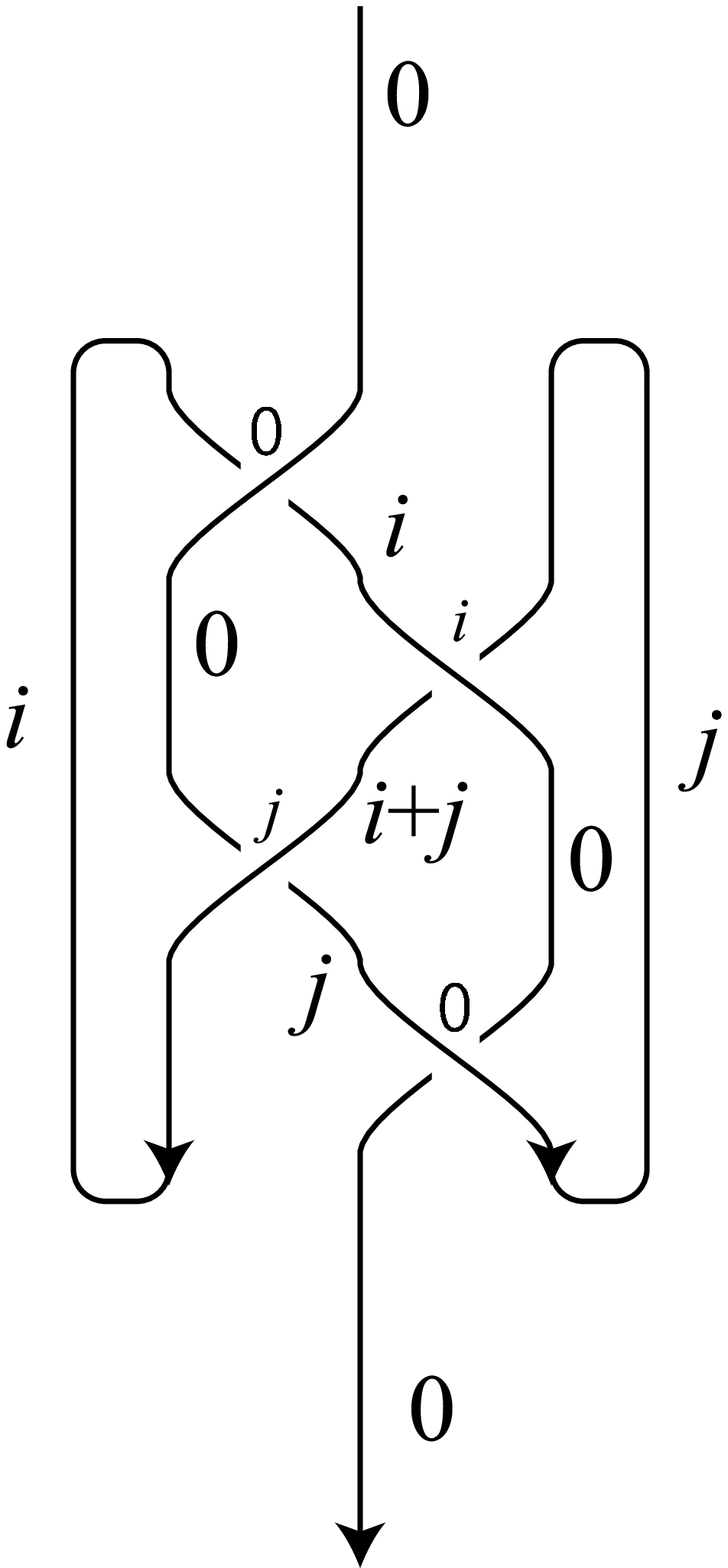}
  \caption{The figure-eight knot can also regarded as the closure of a $(1,1)$-tangle.}
  \label{fig:fig8_calc}
\end{figure}
In this case we need to put $\mu$ at each local minimum where the arc goes from left to right, and $\mu^{-1}$ at each local maximum where the arc goes from right to left.
See \cite[Theorem~3.6]{Kirby/Melvin:INVEM1991} for details.
\par
From Figure~\ref{fig:tangle}, we have
\begin{equation*}
\begin{split}
  &J_N(E;q)
  \\
  =&
  \sum_{\stackrel{0\le i\le N-1,0\le j\le N-1}{0\le i+j\le N-1}}
  R^{i,0}_{0,i}\,
  (R^{-1})^{i,j}_{i+j,0}\,
  R^{0,i+j}_{i,j}\,
  (R^  {-1})^{j,0}_{0,j}\,
  (\mu^{-1})^{i}_{i}\,
  \mu^{j}_{j}
  \\
  =&
  \sum_{\stackrel{0\le i\le N-1,0\le j\le N-1}{0\le i+j\le N-1}}
  (-1)^i
  \frac{\{i+j\}!\{N-1\}!}{\{i\}!\{j\}!\{N-1-i-j\}!}
  q^{-(N-1)i/2+(N-1)j/2-i^2/4+j^2/4-3i/4+3j/4}.
\end{split}
\end{equation*}
Putting $k:=i+j$, we have
\begin{equation*}
  J_N(E;q)
  =
  \sum_{k=0}^{N-1}
  \frac{\{N-1\}!}{\{N-1-k\}!}q^{k^2/4+Nk/2+k/4}
  \left(
    \sum_{i=0}^{k}(-1)^i
    \frac{\{k\}!}{\{i\}!\{k-i\}!}
    q^{-Ni-ik/2-i/2}
  \right).
\end{equation*}
Using the formula (see \cite[Lemma~3.2]{Murakami/Murakami:ACTAM12001})
\begin{equation*}
  \sum_{i=0}^{k}(-1)^iq^{li/2}\frac{\{k\}!}{\{i\}!\{k-i\}!}
  =
  \prod_{g=1}^{k}(1-q^{(l+k+1)/2-g}),
\end{equation*}
we have the following simple formula with only one summand, which is due to Habiro and L{\^e} (I learned this method from L{\^e}).
\begin{equation}\label{eq:fig8_single}
  J_N(E;q)
  =
  \frac{1}{\{N\}}
  \sum_{k=0}^{N-1}\frac{\{N+k\}!}{\{N-1-k\}!}.
\end{equation}
\section{Volume conjecture}\label{sec:Volume_Conjecture}
In this section we state the Volume Conjecture and then prove it for the figure-eight knot.
We also give supporting evidence for the conjecture.
\subsection{Statement of the Volume Conjecture}
In \cite{Kashaev:MODPLA95} Kashaev introduced a link invariant $\langle L\rangle_N\in\C$ for an integer $N$ greater than one and a link $L$ by using the quantum dilogarithm.
Then he observed in \cite{Kashaev:LETMP97} that the limit $2\pi\lim_{N\to\infty}\log|\langle K\rangle_N|/N$ is equal to the hyperbolic volume of the knot complement if a knot $K$ is hyperbolic.
Here a knot in the three-sphere $S^3$ is called hyperbolic if its complement possesses a complete hyperbolic structure with finite volume.
He also conjectured this would also hold for any hyperbolic knot.
\par
In \cite{Murakami/Murakami:ACTAM12001} J.~Murakami and I proved that Kashaev's invariant equals the $N$-dimensional colored Jones polynomial evaluated at the $N$-th root of unity, that is, $\langle L\rangle_N=J_N(L;\exp(2\pi\sqrt{-1}/N))$ and proposed that Kashaev's conjecture would hold for any knot by using the simplicial volume.
\begin{conjecture}[Volume Conjecture \cite{Kashaev:LETMP97,Murakami/Murakami:ACTAM12001}]\label{conj:VC}
The following equality would hold for any knot $K$.
\begin{equation*}
  2\pi\lim_{N\to\infty}\frac{\log|J_N(K;\exp(2\pi\sqrt{-1}/N))|}{N}
  =\Vol(S^3\setminus{K}).
\end{equation*}
\end{conjecture}
To define the simplicial volume (or Gromov norm), we introduce the Jaco--Shalen--Johannson (JSJ) decomposition (or the torus decomposition) of a knot complement.
\begin{definition}[Jaco--Shalen--Johannson decomposition \cite{Jaco/Shalen:MEMAM79,Johannson:1979}]
Let $K$ be a knot.
Then its complement $S^3\setminus{K}$ can be uniquely decomposed into hyperbolic pieces and Seifert fibered pieces by a system of tori:
\begin{equation*}
  S^3\setminus{K}
  =
  \left(\bigsqcup H_i\right)\sqcup\left(\bigsqcup E_j\right)
\end{equation*}
with $H_i$ hyperbolic and $E_j$ Seifert-fibered.
\end{definition}
Then the simplicial volume of the knot complement is defined to be the sum of the hyperbolic volumes of the hyperbolic pieces.
\begin{definition}[Simplicial volume (Gromov norm) \cite{Gromov:INSHE82}]
If a knot complement $S^3\setminus{K}$ is decomposed as above, then its simplicial volume $\Vol(S^3\setminus{K})$ is defined as
\begin{equation*}
  \Vol(S^3\setminus{K})
  :=
  \sum_{H_i:\text{hyperbolic piece}}\text{Hyperbolic Volume of $H_i$}.
\end{equation*}
\end{definition}
\begin{example}
Let us consider the $(2,1)$-cable of the figure-eight knot as shown in figure~\ref{fig:2_1_fig_8}.
\begin{figure}[h]
  \includegraphics[scale=0.3]{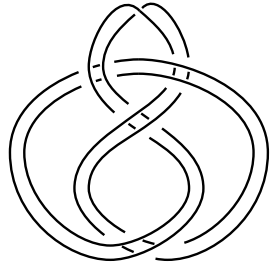}
  \caption{$(2,1)$-cable of the figure-eight knot}
  \label{fig:2_1_fig_8}
\end{figure}
Then its complement can be decomposed by a torus into two pieces (Figure~\ref{fig:2_1_fig_8_JSJ}), one hyperbolic and one Seifert fibered.
\begin{figure}[h]
  \includegraphics[scale=0.3]{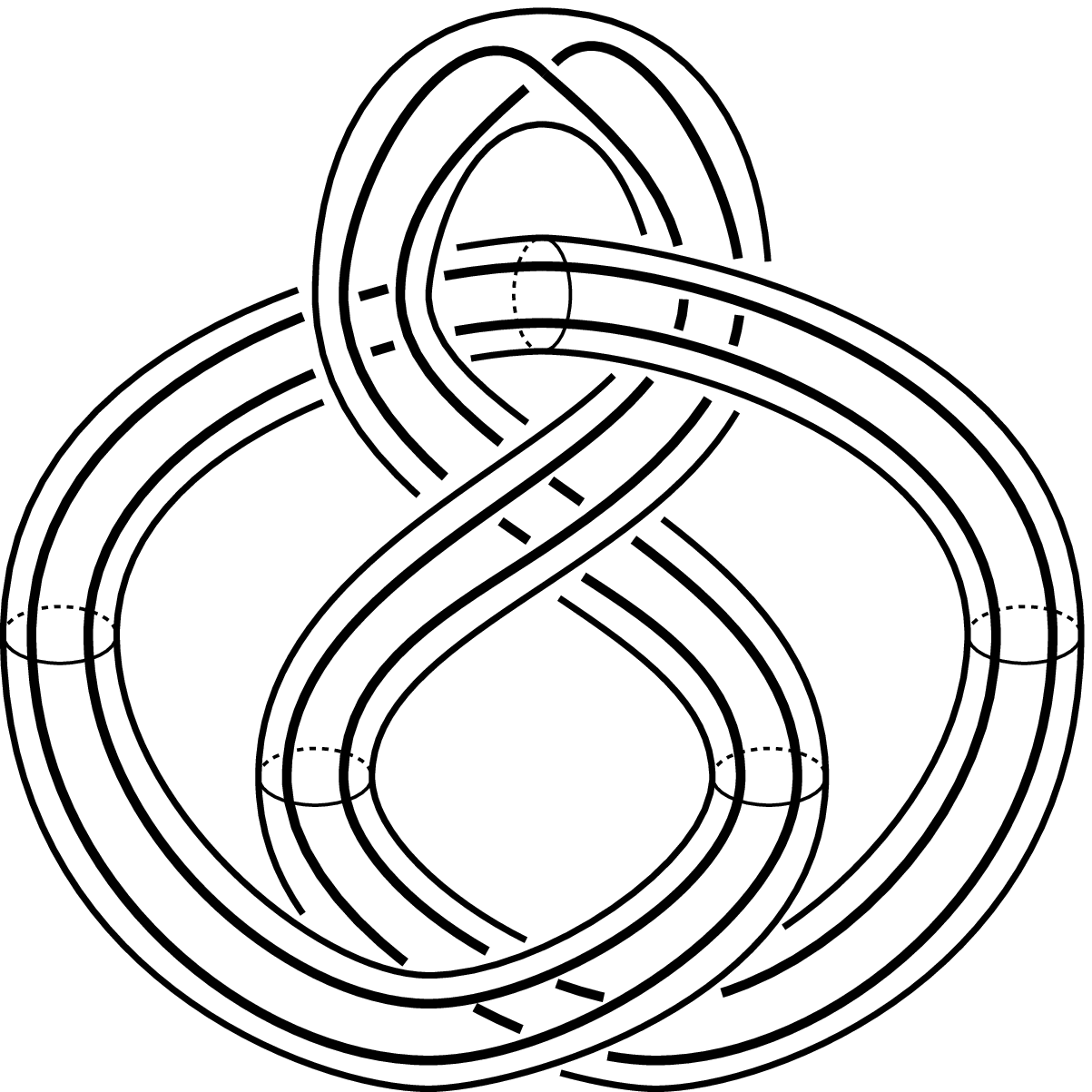}
  \quad\raisebox{15mm}{$=$}\quad
  $\underset{\text{\raisebox{-2mm}{hyperbolic}}}
    {\includegraphics[scale=0.3]{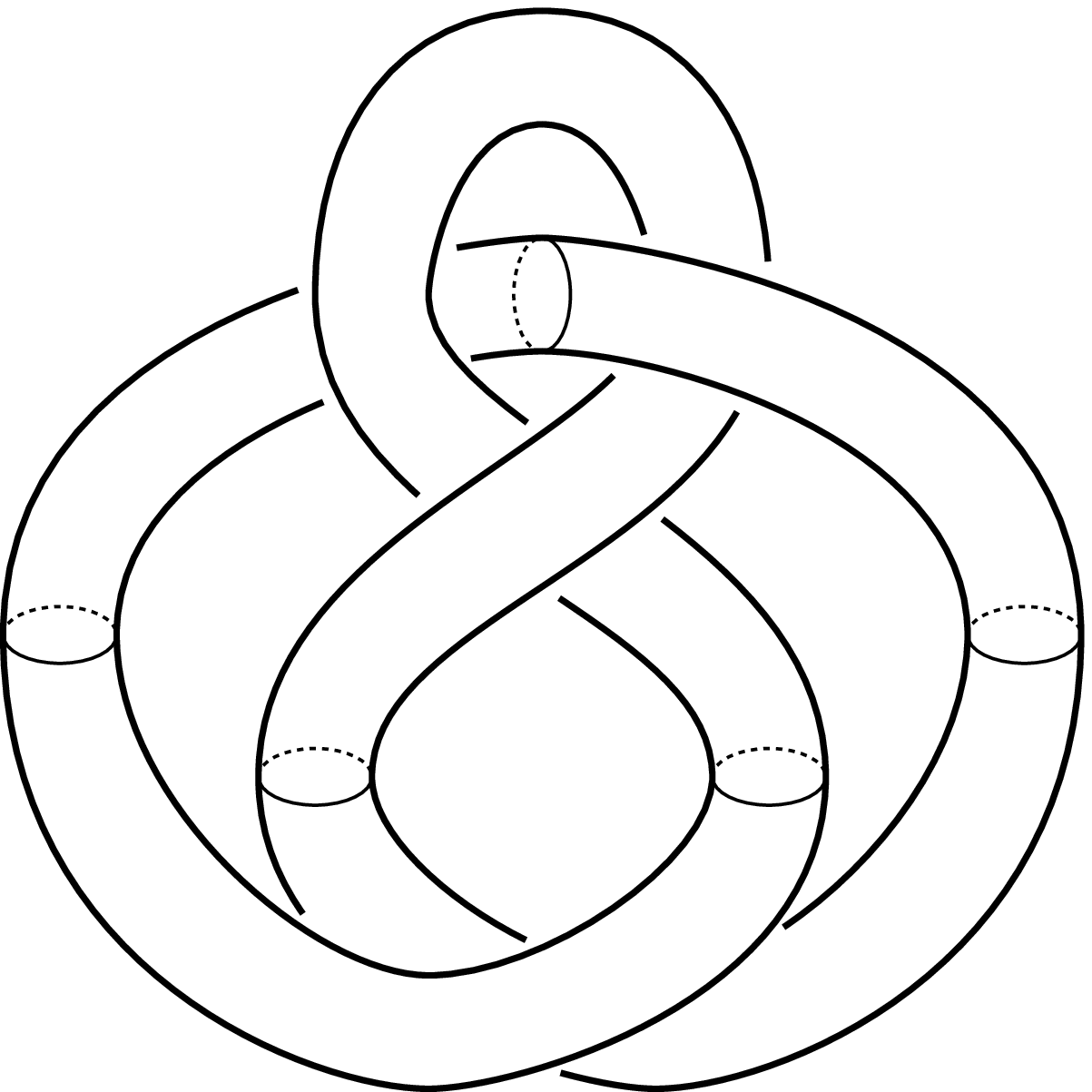}}$
  \quad\raisebox{15mm}{$\sqcup$}\quad
  \raisebox{1mm}{$\underset{\text{\raisebox{-3mm}{Seifert fibered}}}
    {\includegraphics[scale=0.3]{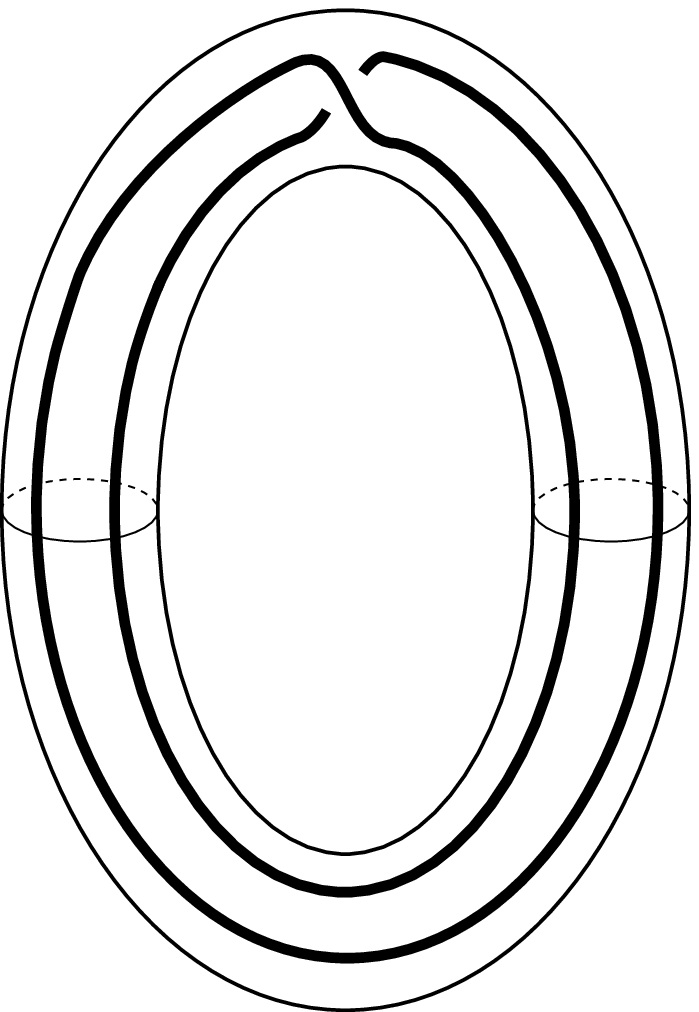}}$}
  \caption{The JSJ decomposition of the $(2,1)$-cable of the figure-eight knot}
  \label{fig:2_1_fig_8_JSJ}
\end{figure}
Therefore we have
\begin{equation*}
  \Vol
  \left(
    \raisebox{-14mm}{\includegraphics[scale=0.3]{complement_2_1_cable_small.eps}}
  \right)
  =
  \Vol
  \left(
    \raisebox{-14mm}{\includegraphics[scale=0.3]{complement_fig8_small.eps}}
  \right).
\end{equation*}
\end{example}
\subsection{Proof of the Volume Conjecture for the figure-eight knot}
\label{subsec:proof_fig8}
We give a proof of the Volume Conjecture for the figure-eight knot due to T.~Ekholm.
\subsubsection{Calculation of the limit}
We use the formula \eqref{eq:fig8_single} of the colored Jones polynomial for the figure-eight knot $E$  due to Habiro and L\^e.
(See \cite{Habiro:SURIK2000,Masbaum:ALGGT12003} for Habiro's method.)
By a simple calculation using it, we have
\begin{equation}\label{eq:Jones_fig8}
  J_N\left(E;q\right)
  =
  \sum_{j=0}^{N-1}
  \prod_{k=1}^{j}
  \left(q^{(N-k)/2}-q^{-(N-k)/2}\right)
  \left(q^{(N+k)/2}-q^{-(N+k)/2}\right).
\end{equation}
Replacing $q$ with $\exp(2\pi\sqrt{-1}/N)$, we have
\begin{equation*}
  J_N\left(E;\exp(2\pi\sqrt{-1}/N)\right)
  =
  \sum_{j=0}^{N-1}\prod_{k=1}^{j}f(N;k)
\end{equation*}
where we put $f(N;k):=4\sin^2(k\pi/N)$.
The graph of $f(N;k)$ is depicted in Figure~\ref{fig:graph_f}.
\begin{figure}[h]
  \includegraphics[scale=0.8]{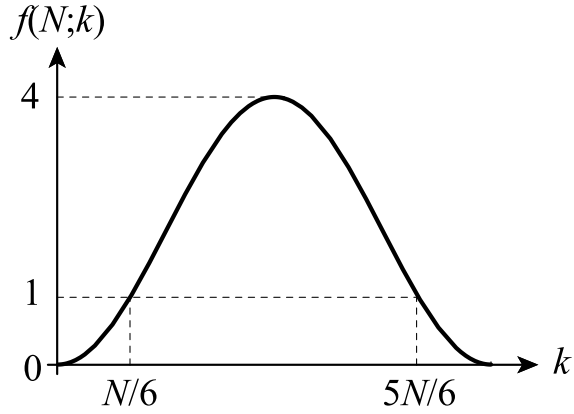}
  \caption{Graph of $f(N;k)$}
  \label{fig:graph_f}
\end{figure}
\par
Put $g(N;j):=\prod_{k=1}^{j}f(N;k)$ so that $J_N\left(E;\exp(2\pi\sqrt{-1}/N)\right)=\sum_{j=0}^{N-1}g(N;j)$.
Then $g(N;j)$ decreases when $0<j<N/6$ and $5N/6<j$, and increases when $N/6<j<5N/6$.
Therefore $g(N;j)$ takes its maximum at $j=5N/6$.
(To be precise we need to take the integer part of $5N/6$.)
See Table~\ref{tbl:f_g}.
\begin{table}[h]
\begin{tabular}{|c|c|c|c|c|c|c|c|}
  \hline
  $j$     &$0$& $\cdots$ &$N/6$& $\cdots$ &$5N/6$& $ \cdots$ &$1$\\
  \hline
    $f(N;k)$&   &   $<1$   & $1$ &   $>1$   & $1$   &  $<1$    &   \\
  \hline
    $g(N;j)$&$1$&$\searrow$&     &$\nearrow$&maximum&$\searrow$&   \\
  \hline
\end{tabular}
\medskip
\caption{table of $f(N;k)$ and $g(N;j)$}
\label{tbl:f_g}
\end{table}
Since there are $N$ positive terms in $J_N\left(E;\exp(2\pi\sqrt{-1}/N)\right)=\sum_{j=0}^{N-1}g(N;j)$ and $g(N;5N/6)$ is the maximum of these terms, we have
\begin{equation*}
  g(N;5N/6)
  \le
  J_N\left(E;\exp(2\pi\sqrt{-1}/N)\right)
  \le
  N\times g(N;5N/6).
\end{equation*}
Noting that each side is positive, we take their logarithms and divide them by $N$.
\begin{equation*}
  \frac{\log{g(N;5N/6)}}{N}
  \le
  \frac{\log J_N\left(E;\exp(2\pi\sqrt{-1}/N)\right)}{N}
  \le
  \frac{\log{N}}{N}+
  \frac{\log{g(N;5N/6)}}{N}.
\end{equation*}
Since $\log_{N\to\infty}(\log{N})/N=0$, we have
\begin{equation*}
  \lim_{N\to\infty}\frac{\log{g(N;5N/6)}}{N}
  \le
  \lim_{N\to\infty}
  \frac{\log J_N\left(E;\exp(2\pi\sqrt{-1}/N)\right)}{N}
  \le
  \lim_{N\to\infty}\frac{\log{g(N;5N/6)}}{N}.
\end{equation*}
Therefore we have
\begin{equation*}
  \lim_{N\to\infty}
  \frac{\log{J_N\left(E;\exp(2\pi\sqrt{-1}/N)\right)}}{N}
  =
  \lim_{N\to\infty}\frac{\log{g(N;5N/6)}}{N}.
\end{equation*}
\par
We can calculate the limit $\lim_{N\to\infty}\bigl(\log{g(N;5N/6)}\bigr)/N$ by integration.
Since $g(N;j)=\prod_{k=1}^{j}f(N;k)$, we have
\begin{equation}\label{eq:Jones_integral}
\begin{split}
  \lim_{N\to\infty}\frac{\log{g(N;5N/6)}}{N}
  &=
  \lim_{N\to\infty}
  \frac{1}{N}
  \sum_{k=1}^{5N/6}\log{f(N;k)}
  \\
  &=
  2
  \lim_{N\to\infty}
  \frac{1}{N}
  \sum_{k=1}^{5N/6}\log\bigl(2\sin(k\pi/N)\bigr)
  \\
  &=
  \frac{2}{\pi}
  \int_{0}^{5\pi/6}
  \log\bigl(2\sin{x}\bigr)\,dx.
\end{split}
\end{equation}
What does this mean?
I will explain a geometric interpretation of this integral.
\subsubsection{Lobachevsky function}
The function defined by the integral in \eqref{eq:Jones_integral} is known as the Lobachevsky function.
More precisely, we define the Lobachevsky function $\Lambda(\theta)$ as
\begin{equation*}
  \Lambda(\theta)
  :=
  -\int_{0}^{\theta}\log|2\sin{x}|\,dx
\end{equation*}
for $\theta\in\R$.
By using this function, we can express the limit of the colored Jones polynomial as
\begin{equation}\label{eq:Jones_Lobachevsky}
  \lim_{N\to\infty}
  \frac{\log{J_N\left(E;\exp(2\pi\sqrt{-1}/N)\right)}}{N}
  =
  -\frac{2}{\pi}\Lambda(5\pi/6)
\end{equation}
We show some properties of the Lobachevsky function (see for example \cite{Milnor:BULAM382}).
\begin{lemma}
The Lobachevsky function satisfies the following two properties.
\begin{enumerate}
\item The Lobachevsky function is an odd function and has period $\pi$.
\item We have $\Lambda(2\theta)=2\Lambda(\theta)+2\Lambda(\theta+\pi/2)$.
In general we have $\Lambda(n\theta)=n\sum_{k=1}^{n-1}\Lambda(\theta+k\pi/n)$.
\end{enumerate}
\end{lemma}
\begin{proof}
The first property is easily shown by the periodicity of the sine function.
\par
To prove the second, we use the double angle formula of the sine function: $\sin(2x)=2\sin{x}\cos{x}$.
We have
\begin{equation*}
  \log|2\sin(2x)|
  =
  \log|2\sin{x}|
  +
  \log|2\cos{x}|
  =
  \log|2\sin x|
  +
  \log|2\sin(x+\pi/2)|,
\end{equation*}
completing the proof.
\end{proof}
From (1) we have
\begin{equation*}
  \Lambda(5\pi/6)
  =
  \Lambda(-\pi/6)
  =
  -\Lambda(\pi/6).
\end{equation*}
From (2) and (1) we also have
\begin{equation*}
  \Lambda(\pi/3)
  =
  2\Lambda(\pi/6)+2\Lambda(2\pi/3)
  =
  2\Lambda(\pi/6)-2\Lambda(\pi/3).
\end{equation*}
Therefore we have
\begin{equation*}
  \Lambda(5\pi/6)=-\frac{3}{2}\Lambda(\pi/3).
\end{equation*}
\par
Returning to the limit of the Jones polynomial \eqref{eq:Jones_Lobachevsky}, we conclude that
\begin{equation*}
  2\pi\lim_{N\to\infty}
  \frac{\log{J_N\left(E;\exp(2\pi\sqrt{-1}/N)\right)}}{N}
  =
  6\Lambda(\pi/3).
\end{equation*}
\par
Next we show how the Lobachevsky function is related to hyperbolic geometry.
\subsubsection{Hyperbolic geometry}
\label{subsubsec:geometry}
It is well known that the complement of the figure-eight knot can be decomposed into two ideal hyperbolic regular tetrahedra.
\begin{theorem}[W.~Thurston \cite{Thurston:GT3M}]
The complement of the figure-eight knot can be obtained by gluing two ideal hyperbolic regular tetrahedra.
\end{theorem}
I will explain what is an ideal hyperbolic regular tetrahedron later.
Topologically, the theorem states that the complement of the figure-eight knot can be obtained by gluing two truncated tetrahedra as in Figure~\ref{fig:fig8_tetra} (see also \cite{Murakami:FUNDM2004}).
In Figure~\ref{fig:fig8_tetra} we identify $A$ with $A'$, $B$ with $B'$, $C$ with $C'$ and $D$ with $D'$.
Note that edges with single arrows and edges with double arrows are also identified respectively.
\begin{figure}[h]
  \includegraphics[scale=0.25]{complement_fig8_small.eps}
  \quad\raisebox{13mm}{$=$}\quad
  \includegraphics[scale=0.4]{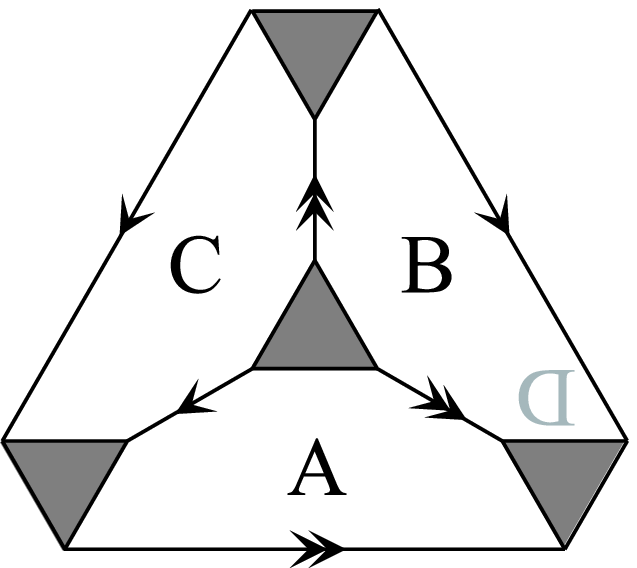}
  \quad\raisebox{13mm}{$\underset{\substack{A=A',B=B',\\C=C',D=D'}}{\cup}$}\quad
  \includegraphics[scale=0.4]{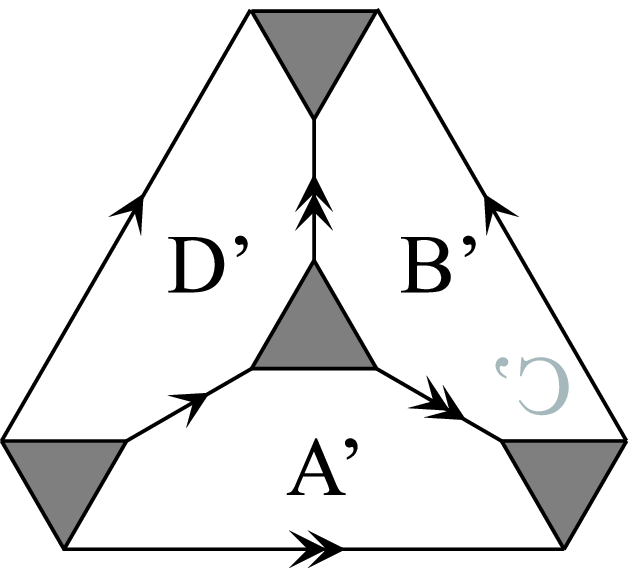}
  \caption{The complement of the figure-eight knot is decomposed into two ideal hyperbolic regular tetrahedra.
  Shadowed triangles make the boundary of the knot complement.}
  \label{fig:fig8_tetra}
\end{figure}
\par
Here I give a short introduction to hyperbolic geometry.
\par
First consider the upper half space $\{(x,y,z)\in\R^3\mid z>0\}$ with hyperbolic metric $ds:=\sqrt{dx^2+dy^2+dz^2}/z$ and denote it by $\H^3$.
It is known that a geodesic line in $\H^3$ is a semicircle or a straight line perpendicular to the $xy$-plane, and that a geodesic plane is a hemisphere or a flat plane perpendicular to the $xy$-plane.
\par
An ideal hyperbolic tetrahedron is a tetrahedron in $\H^3$ with geodesic faces with four vertices at infinity, that is, on the $xy$-plain or at the point at infinity $\infty$.
By isometry we may assume that one vertex is $\infty$ and the other three are on the $xy$-plane.
So its faces consist of three perpendicular planes and a hemisphere as shown in Figure~\ref{fig:tetrahedron}.
\begin{figure}[h]
  \includegraphics[scale=0.7]{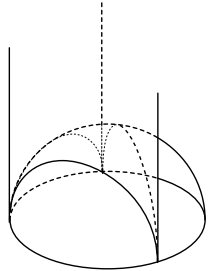}
  \caption{An ideal hyperbolic tetrahedron is the part above the hemisphere surrounded by three perpendicular planes.}
  \label{fig:tetrahedron}
\end{figure}
If we see the tetrahedron from the top, it is a (Euclidean) triangle with angles $\alpha$, $\beta$, and $\gamma$.
\begin{figure}[h]
  \includegraphics[scale=0.7]{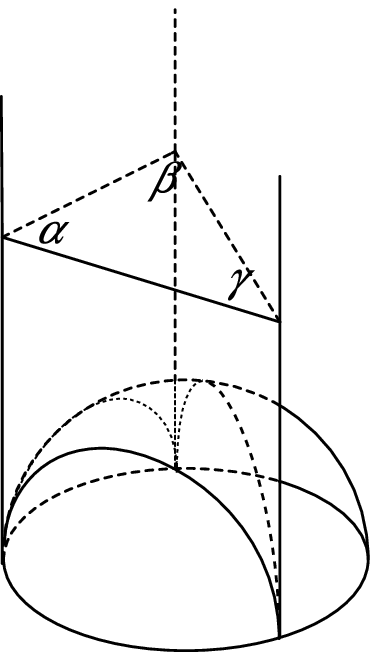}
  \hspace{20mm}
  \raisebox{7mm}{\includegraphics[scale=0.5]{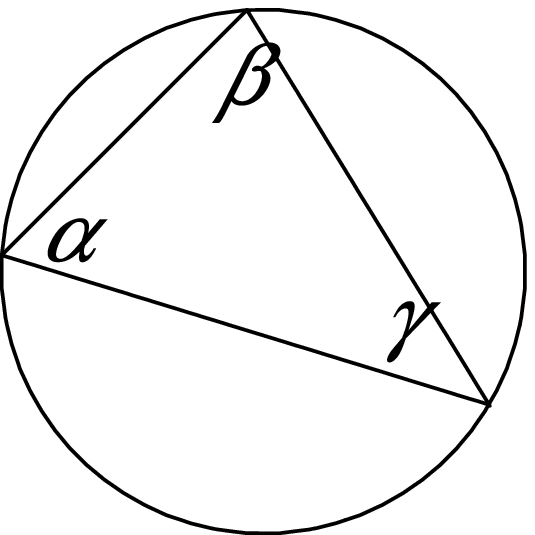}}
  \caption{A top view (right) of the ideal hyperbolic tetrahedron (left)}
  \label{fig:top_view}
\end{figure}
It is known that an ideal hyperbolic tetrahedron is defined (up to isometry) by the similarity class of this triangle.
Therefore we can parametrize an ideal hyperbolic tetrahedron by a triple of positive numbers $(\alpha,\beta,\gamma)$ with $\alpha+\beta+\gamma=\pi$.
We denote it by $\Delta(\alpha,\beta,\gamma)$.
\par
The hyperbolic volume $\Vol(\Delta(\alpha,\beta,\gamma))$ can be expressed by using the Lobachevsky function $\Lambda(\theta)$.
In fact it can be shown that
\begin{equation*}
  \Vol(\Delta(\alpha,\beta,\gamma))
  =
  \Lambda(\alpha)+\Lambda(\beta)+\Lambda(\gamma).
\end{equation*}
For a proof, see for example \cite[Chapter~7]{Thurston:GT3M}.
\par
Now we return to the decomposition of the figure-eight knot complement.
Figure~\ref{fig:fig8_tetra} shows that after identification we have two edges, edge with single arrow and edge with double arrow, each of them is obtained by identifying six edges.
So if the ideal hyperbolic tetrahedra we are using are regular, that is, isometric to $\Delta(\pi/3,\pi/3,\pi/3)$ then the sum of dihedral angles around each edge becomes $2\pi$.
This means that if we use two ideal hyperbolic regular tetrahedra, our gluing is geometric, that is, the complement of the figure-eight knot is isometric to the union of two copies of $\Delta(\pi/3,\pi/3,\pi/3)$.
In particular its volume equals $2\Vol(\Delta(\pi/3,\pi/3,\pi/3))=6\Lambda(\pi/3)$.
\par
Thus we have proved
\begin{equation*}
  2\pi\lim_{N\to\infty}
  \frac{\log{J_N\left(E;\exp(2\pi\sqrt{-1}/N)\right)}}{N}
  =
  6\Lambda(\pi/3)
  =
  \Vol\left(S^3\setminus{E}\right),
\end{equation*}
which is the statement of the Volume Conjecture for the figure-eight knot.
\subsection{Knots and links for which the Volume Conjecture is proved}
As far as I know the Volume Conjecture is proved for
\begin{enumerate}
\item figure-eight knot by Ekholm,
\item $5_2$ knot by Kashaev and Yokota,
\item Whitehead doubles of torus knots by Zheng \cite{Zheng:CHIAM22007},
\item torus knots by Kashaev and Tirkkonen \cite{Kashaev/Tirkkonen:ZAPNS2000},
\item torus links of type $(2,2m)$ by Hikami \cite{Hikami:COMMP2004},
\item knots and links with volume zero by van der Veen \cite{van_der_veen:2008},
\item Borromean rings by Garoufalidis and L{\^e} \cite{Garoufalidis/Le:2005},
\item twisted Whitehead links by Zheng \cite{Zheng:CHIAM22007},
\item Whitehead chains by van der Veen \cite{van_der_Veen:ACTMV2009},
\item a satellite link around the figure-eight knot with pattern the Whitehead link by Yamazaki and Yokota \cite{Yamazaki/Yokota}.
\end{enumerate}
Note that (1) and (2) are for hyperbolic knots, (3) is for a knot whose JSJ decomposition consists of a hyperbolic piece and a Seifert fibered piece, (4)--(6) are for knots and links only with Seifert pieces, (7)--(9) are for hyperbolic links, and (10) is for a link whose JSJ decomposition consists of a hyperbolic piece and a Seifert fibered piece.
\section{Supporting evidence for the Volume Conjecture}
\label{sec:VC_proof}
The Volume Conjecture is proved only for several knots and links but I think it is true possibly with some modification; for example we may need to replace the limit with the limit superior (see \cite[Conjecture~2]{van_der_Veen:ACTMV2009}).
In this section I will explain why I think it is true.
\begin{remark}[Caution!]
Descriptions in this section are {\em not rigorous}.
\end{remark}
\subsection{Approximation of the colored Jones polynomial}
Put $\xi_N:=\exp(2\pi\sqrt{-1}/N)$ and I will give an interpretation of the $R$-matrix used to define the colored Jones polynomial.
From \eqref{eq:R} and \eqref{eq:R_inverse} we have
\begin{equation}\label{eq:R_N}
\begin{split}
  R^{ij}_{kl}\bigr|_{q:=\xi_N}
  &=
  \sum_{m}
  \text{(a power of $\xi_N$)}\times
  \delta_{l,i+m}\delta_{k,j-m}\,
  \frac{\{l\}!\{N-1-k\}!}{\{i\}!\{m\}!\{N-1-j\}!},
  \\
  (R^{-1})^{ij}_{kl}\bigr|_{q:=\xi_q}
  &=
  \sum_{m}
  \text{(a power of $\xi_N$)}\times
  \delta_{l,i-m}\delta_{k,j+m}\,
  \frac{\pm\{k\}!\{N-1-l\}!}{\{j\}!\{m\}!\{N-1-i\}!}.
\end{split}
\end{equation}
Since
\begin{equation*}
  \{k\}!\bigr|_{q:=\xi_N}
  =
  \left(2\sqrt{-1}\right)^{k}
  \sin(\pi/N)
  \sin(2\pi/N)
  \cdots
  \sin(k\pi/N),
\end{equation*}
we have
\begin{equation*}
\begin{split}
  &
  \{k\}!\{N-k-1\}!\bigr|_{q:=\xi_N}
  \\
  =&
  \left(2\sqrt{-1}\right)^{N-1}
  \sin(\pi/N)\sin(2\pi/N)\cdots\sin(k\pi/N)
  \\
  &\times
  \sin\bigl((N-(k+1))\pi/N\bigr)
  \sin\bigl((N-(k+2))\pi/N\bigr)
  \sin\bigl((N-(N-1))\pi/N\bigr)
  \\
  =&
  (-1)^{N-k-1}\left(\sqrt{-1}\right)^{N-1}
  2^{N-1}
  \sin(\pi/N)\sin(2\pi/N)\cdots\sin((N-1)\pi/N)
  \\
  =&
  \varepsilon N
\end{split}
\end{equation*}
for $\varepsilon\in\{1,-1,\sqrt{-1},-\sqrt{-1}\}$.
For the last equality, see for example \cite{Murakami/Murakami:ACTAM12001}.
So if we put
\begin{align*}
  (\xi_N)_{k^+}&:=(1-\xi_N)\cdots(1-\xi^k_N),
  \\
  (\xi_N)_{k^{-}}&:=(1-\xi_N)\cdots(1-\xi^{N-1-k}_N),
\end{align*}
we have
\begin{align*}
  \{k\}!\bigr|_{q:=\xi_N}
  &=
  \text{(a power of $\xi_N$)}\times(\xi_N)_{k^+},
  \\
  \{N-1-k\}!\bigr|_{q:=\xi_N}
  &=
  \text{(a power of $\xi_N$)}\times(\xi_N)_{k^-}
  \\
  \intertext{and}
  (\xi_N)_{k^+}(\xi_N)_{k^-}
  &=
  \text{(a power of $\xi_N$)}\times N.
\end{align*}
\par
Therefore from \eqref{eq:R_N} the $R$-matrix and its inverse can be written as
\begin{align*}
  R^{ij}_{kl}\bigr|_{q:=\xi_N}
  &=
  \sum_{m}
  \delta_{l,i+m}\delta_{k,j-m}\,
  \frac{\text{(a power of $\xi_N$)}\times N^2}
       {(\xi_N)_{m^+}(\xi_N)_{i^+}(\xi_N)_{k^+}
                       (\xi_N)_{j^-}(\xi_N)_{l^-}},
  \\
  (R^{-1})^{ij}_{kl}\bigr|_{q:=\xi_N}
  &=
  \sum_{m}
  \delta_{l,i-m}\delta_{k,j+m}\,
  \frac{\text{(a power of $\xi_N$)}\times N^{2}}
       {(\xi_N)_{m^+}(\xi_N)_{i^-}(\xi_N)_{k^-}
                       (\xi_N)_{j^+}(\xi_N)_{l^+}}.
\end{align*}
\par
If we are given a knot $K$, we express it as a closed braid and calculate the colored Jones polynomial as described in \S\ref{subsec:calculation}.
Then we have
\begin{equation}\label{eq:Jones_R}
\begin{split}
  J_N(K;\xi_N)
  &=
  \sum_{\text{labelings}}
  \left(
    \prod_{\text{crossings}}
    (R^{\pm1})^{ij}_{kl}
  \right)
  \\
  &=
  \sum_{\text{labelings}}
  \left(
    \prod_{\text{crossings}}
    \frac{\text{(a power of $\xi_N$)}\times N^{2}}
         {(\xi_N)_{m^+}(\xi_N)_{i^{\pm}}(\xi_N)_{k^{\pm}}
                         (\xi_N)_{j^{\mp}}(\xi_N)_{l^{\mp}}}
  \right),
\end{split}
\end{equation}
where the summation is over all the labelings with $\{0,1,\dots,N-1\}$ corresponding to the basis $\{e_0,e_1,\dots,e_{N-1}\}$ and for a fixed labeling the product is over all the crossings, each of which corresponds to an entry $R^{ij}_{kl}$ (or $(R^{-1})^{i,j}_{kl}$, respectively), determined by the four labeled arcs around the vertex, of the $R$-matrix (or its inverse, respectively).
\par
We will approximate $(\xi_N)_{k^{+}}$ for large $N$.
By taking the logarithm, we have
\begin{equation*}
\begin{split}
  \log(\xi_N)_{k^{+}}
  &=
  \sum_{j=1}^{k}
  \log(1-\xi^j_N)
  \\
  &=
  \sum_{j=1}^{k}
  \log\bigl(1-\exp(2\pi\sqrt{-1}j/N)\bigr).
\end{split}
\end{equation*}
Putting $x:=j/N$, we may replace the summation with the following integral for large $N$ (this is {\em not rigorous!}).
\begin{equation*}
\begin{split}
  \log(\xi_N)_{k^{+}}
  &\underset{N\to\infty}{\approx}
  N\int_{0}^{k/N}\log\bigl(1-\exp(2\pi\sqrt{-1}x)\bigr)\,dx
  \\
  &=
  \frac{N}{2\pi\sqrt{-1}}
  \int_{1}^{\exp(2\pi\sqrt{-1}k/N)}\frac{\log(1-y)}{y}\,dy,
\end{split}
\end{equation*}
where we put $y:=\exp(2\pi\sqrt{-1}x)$ in the last equality and $\underset{N\to\infty}{\approx}$ means a very rough approximation (which may be not true at all) for large $N$.
\par
This integral is known as the dilogarithm function.
We put
\begin{equation*}
  \Li_2(z)
  :=
  -\int_{0}^{z}\frac{\log(1-y)}{y}\,dy
\end{equation*}
for $z\in\C\setminus{[1,\infty)}$.
This is called dilogarithm since it has the Taylor expansion for $|z|<1$ as
\begin{equation*}
  \sum_{n=1}^{\infty}\frac{z^n}{n^2}.
\end{equation*}
For more details about this function, see for example \cite{Zagier:2007}.
\par
So by using the dilogarithm function, we have the following approximation:
\begin{equation*}
  \log(\xi_N)_{k^{+}}
  \underset{N\to\infty}{\approx}
  \frac{N}{2\pi\sqrt{-1}}
  \left[
    \Li_2(1)-\Li_2(\xi^k_N)
  \right].
\end{equation*}
Since $\Li_2(1)$, which equals $\zeta(2)=\pi^2/6$ ($\zeta(z)$ is the Riemann zeta function), can be ignored for large $N$, we have
\begin{equation*}
  (\xi_N)_{k^{+}}
  \underset{N\to\infty}{\approx}
  \exp
  \left[
    -\frac{N}{2\pi\sqrt{-1}}\Li_2(\xi^{k}_N)
  \right].
\end{equation*}
Similarly we have
\begin{equation*}
  (\xi_N)_{k^{-}}
  \underset{N\to\infty}{\approx}
  \exp
  \left[
    -\frac{N}{2\pi\sqrt{-1}}\Li_2(\xi^{-k}_N)
  \right].
\end{equation*}
\par
Therefore from \eqref{eq:Jones_R} the colored Jones polynomial can be (roughly) approximated as follows.
\begin{multline}\label{eq:Jones_approx}
  J_N(K;\xi_N)
  \underset{N\to\infty}{\approx}
  \sum_{\text{labelings}}
  \exp
  \left[\vphantom{\sum_{\text{crossings}}}
    \frac{N}{2\pi\sqrt{-1}}
  \right.
  \\
  \left.
  \times
  \left(
    \sum_{\text{crossings}}
    \left\{
      \Li_2(\xi^m_N)
      +
      \Li_2(\xi^{\pm i}_N)
      +
      \Li_2(\xi^{\mp j}_N)
      +
      \Li_2(\xi^{\pm k}_N)
      +
      \Li_2(\xi^{\mp l}_N)
      +
      \text{$\log$ terms}
    \right\}
  \right)
  \right],
\end{multline}
where the $\log$ terms come from powers of $\xi_N$ and $N$.
\par
Since the term
\begin{equation}\label{eq:dilog}
  \sum_{\text{crossings}}
  \left\{
    \Li_2(\xi^m_N)
    +
    \Li_2(\xi^{\pm i}_N)
    +
    \Li_2(\xi^{\mp j}_N)
    +
    \Li_2(\xi^{\pm k}_N)
    +
    \Li_2(\xi^{\mp l}_N)
    +
    \text{$\log$ terms}
  \right\}
\end{equation}
can be regarded as a function of $\xi_N^{i_1},\xi_N^{i_2},\dots,\xi_N^{i_c}$ with $i_1,i_2,\dots,i_c$ labelings of arcs, we can write it as $V(\xi_N^{i_1},\xi_N^{i_1}\dots,\xi_N^{i_c})$.
Note that $m$ can be expressed as a difference of two such labelings.
So we have
\begin{equation}\label{eq:Jones_sum}
  J_N(K;\xi_N)
  \underset{N\to\infty}{\approx}
  \sum_{i_1,i_2,\dots,i_c}
  \exp
  \left[
    \frac{N}{2\pi\sqrt{-1}}
    V(\xi_N^{i_1},\xi_N^{i_2},\dots,\xi_N^{i_c})
  \right].
\end{equation}
\par
We want to apply a method used in the proof for the figure-eight knot in \S\ref{subsec:proof_fig8}.
Recall that the point of the proof is to find the maximum term of the summand because the maximum dominates the asymptotic behavior.
We will seek for the ``maximum'' term of the summation in \eqref{eq:Jones_sum}.
\par
To do that we first approximate this with the following integral:
\begin{equation}\label{eq:Jones_integral2}
  J_N(K;\xi_N)
  \underset{N\to\infty}{\approx}
  \int_{J_1}\int_{J_2}\cdots\int_{J_c}
  \exp
  \left[
    \frac{N}{2\pi\sqrt{-1}}
    V(z_1,z_2\dots,z_c)
  \right]\,dz_1dz_2\cdots dz_c,
\end{equation}
where $z_d$ corresponds to $\xi_N^{i_d}$ and $J_1,J_2,\dots,J_c$ are certain contours.
(The argument here is not rigorous.
In particular I do not know how to choose the contours.)
\par
Then we will find the maximum of the absolute value of the integrand.
To be more precise we apply the steepest descent method.
For a precise statement, see for example \cite[Theorem~7.2.9]{Marsden/Hoffman:Complex_Analysis}).
\begin{theorem}[steepest descent method]\label{thm:steepest_descent_method}
Under certain conditions for the functions $f$, $g$, and a contour $C$, we have
\begin{equation*}
\begin{split}
  \int_{C}g(z)\exp(N\,h(z))\,dz
  &\underset{N\to\infty}{\sim}
  \frac{\exp(N\,h(z_0))\sqrt{2\pi}g(z_0)}{\sqrt{N}\sqrt{-h''(z_0)}}
  \\
  &\underset{N\to\infty}{\approx}
  \exp(N\,h(z_0)),
\end{split}
\end{equation*}
where $h'(z_0)=0$ and $\Re(h(z))$ takes its positive maximum at $z_0$.
\par
Note that the symbol $\underset{N\to\infty}{\sim}$ means that the ratio of both sides converges to $1$ when $N\to\infty$ and that we ignore the constant term and $\sqrt{N}$ in the rough approximation $\underset{N\to\infty}{\approx}$ because $\exp(N\,h(z_0))$ grows exponentially when $N\to\infty$.
$($Recall that we want to know the limit of $\log|J_N(K,\xi_N)|/N$ and so polynomial terms will not matter.$)$
\end{theorem}
\begin{remark}
In general, to apply the steepest descent method, we need to change the contour $C$ so that it passes through $z_0$.
\end{remark}
Now we apply (a multidimensional version of) this method to \eqref{eq:Jones_integral2} and we will find the the maximum of $\{\Im V(z_1,\dots,z_c)\mid{(z_1,z_2,\dots,z_c)\in J_1\times J_2\dots\times J_c}\}$.
Let $(x_1,x_2,\dots,x_c)$ be such a point.
Then we have
\begin{equation}\label{eq:steepest_descent}
  J_N(K;\xi_N)
  \underset{N\to\infty}{\approx}
  \exp
  \left[
    \frac{N}{2\pi\sqrt{-1}}
    V(x_1,x_2,\dots,x_c)
  \right]
\end{equation}
and so we finally have
\begin{equation}\label{eq:limit}
  2\pi\sqrt{-1}
  \lim_{N\to\infty}
  \frac{\log{J_N(K;\xi_N)}}{N}
  =
  V(x_1,x_2,\dots,x_c).
\end{equation}
Note that the point $(x_1,x_2,\dots,x_c)$ is a solution to the following equation:
\begin{equation}\label{eq:differential}
  \frac{\partial\,V}{\partial\,z_d}(z_1,z_2,\dots,z_c)=0
\end{equation}
for $d=1,2,\dots,c$.
\par
Remember that our argument here is far from rigor!
Especially I am cheating on the following points:
\begin{itemize}
\item Replacing a summation with an integral in \eqref{eq:Jones_integral2}.
Here we do not know how to choose the multidimensional contour.
\item Applying the steepest descent method in \eqref{eq:steepest_descent}.
In general, we have many solutions to the system of equations \eqref{eq:differential} but we do not know which solution gives the maximum.
Moreover we may need to change the contour so that it passes through the solution that gives the maximum but we do not know whether this is possible or not.
\end{itemize}
\subsection{Geometric interpretation of the limit}
In this subsection I will give a geometric interpretation of the limit \eqref{eq:limit}.
\par
Since $V(z_1,z_2,\dots,z_c)$ is the sum of the terms as in \eqref{eq:dilog}, we first describe a geometric meaning of $\Li_2(\zeta_N^{\pm i})$.
\par
Recall that an ideal hyperbolic tetrahedron can be put in $\H^3$.
Regarding the $xy$-plane as the complex plain, we can assume that the three of the four (ideal) vertices are at $0$, $1$ and $z\in\C$ ($\Im{z}>0$), respectively (Figure~\ref{fig:tetrahedron_parameter}).
\begin{figure}[h]
  \includegraphics[scale=0.7]{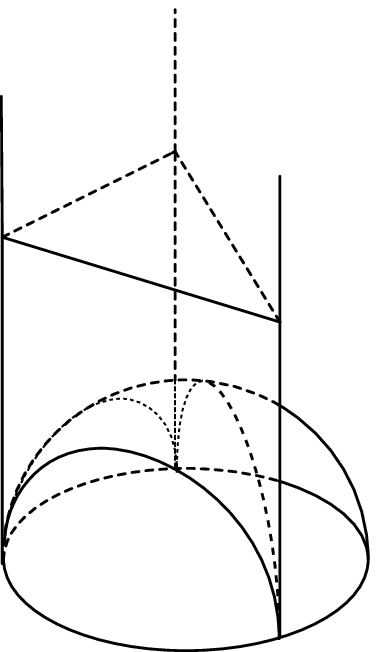}
  \quad\raisebox{13mm}{$\Rightarrow$}\quad
  \raisebox{3mm}{\includegraphics[scale=0.3]{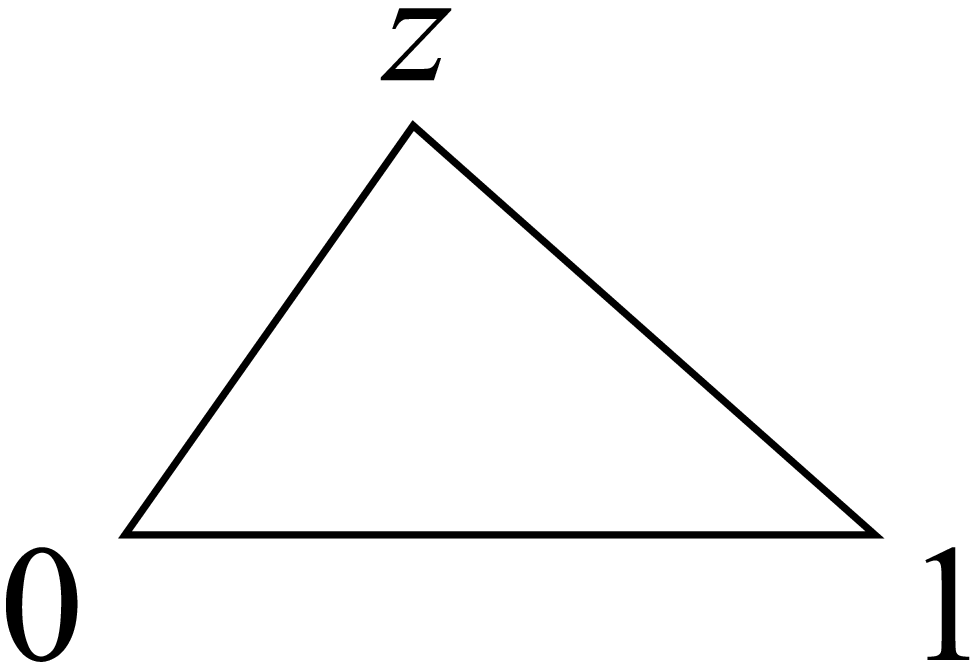}}
  \caption{Parametrization of ideal hyperbolic tetrahedra}
  \label{fig:tetrahedron_parameter}
\end{figure}
Thus the set $\{z\in\C\mid\Im{z}>0\}$ gives a parametrization of ideal hyperbolic tetrahedra.
We denote by $\Delta(z)$ the hyperbolic tetrahedron parametrized by $z$.
\par
The volume of $\Delta(z)$ is given as follows (see for example \cite[p.~324]{Neumann/Zagier:TOPOL85}):
\begin{equation}\label{eq:tetra_vol}
  \Vol(\Delta(z))
  =
  \Im\Li_2(z)+\log|z|\arg(1-z).
\end{equation}
\par
So we expect $V(x_1,x_2,\dots,x_c)$ gives the sum of the volumes of certain tetrahedra related to the knot.
\par
In fact we can express the volume of the knot complement in terms of $V(z_1,z_2,\dots,z_c)$ \cite{D.Thurston:Grenoble,Yokota:GTM02}.
\par
We follow \cite{D.Thurston:Grenoble} to describe this.
\par
We decompose the knot complement into topological, truncated tetrahedra.
To do this we put an octahedron at each positive crossing as in Figure~\ref{fig:crossing_octa}, where $i$, $j$, $k$ and $l$ are labeling of the four arcs around the vertex.
\begin{figure}[h]
  \includegraphics[scale=0.3]{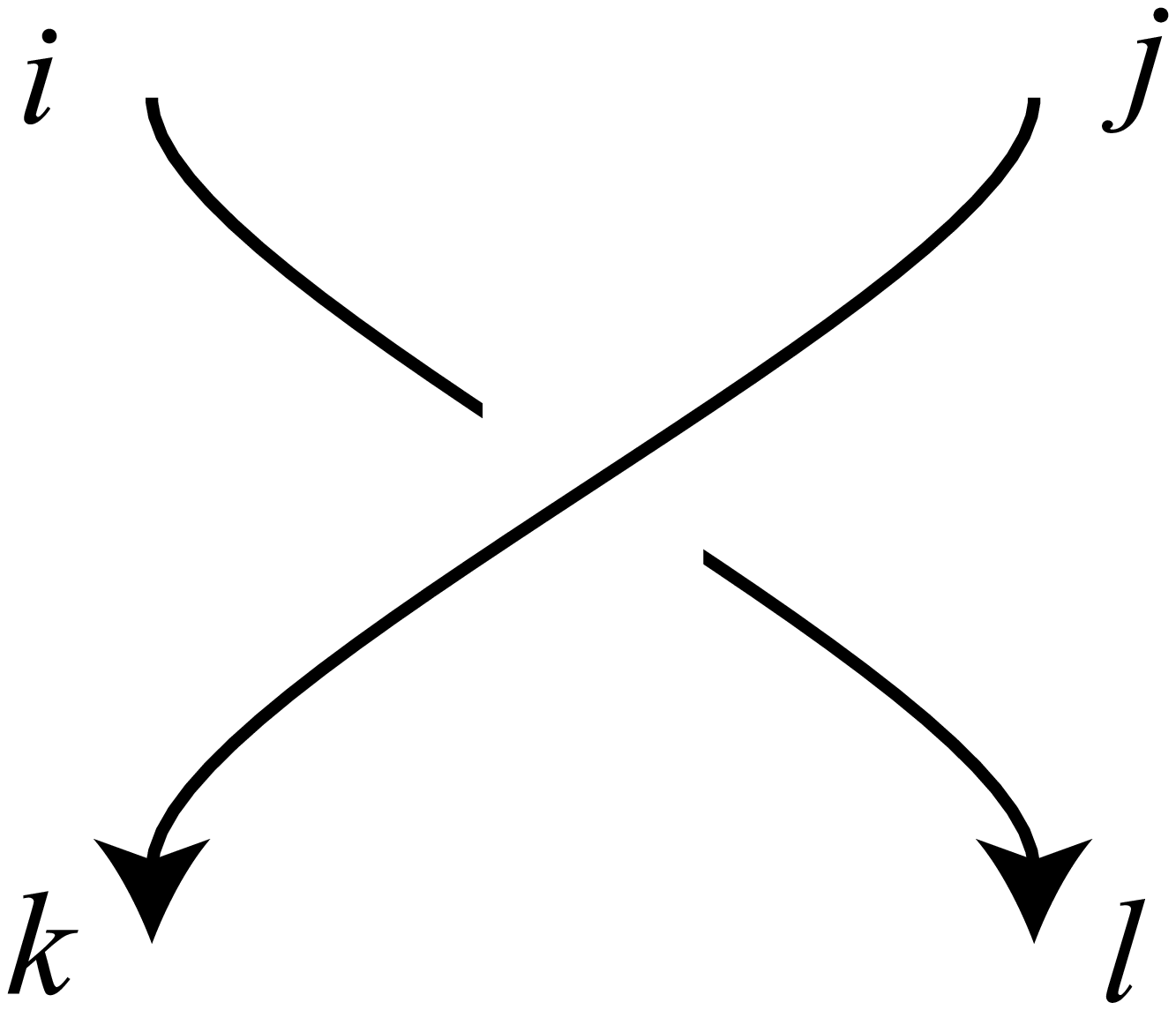}
  \quad\raisebox{17mm}{$\Rightarrow$}\quad
  \includegraphics[scale=0.3]{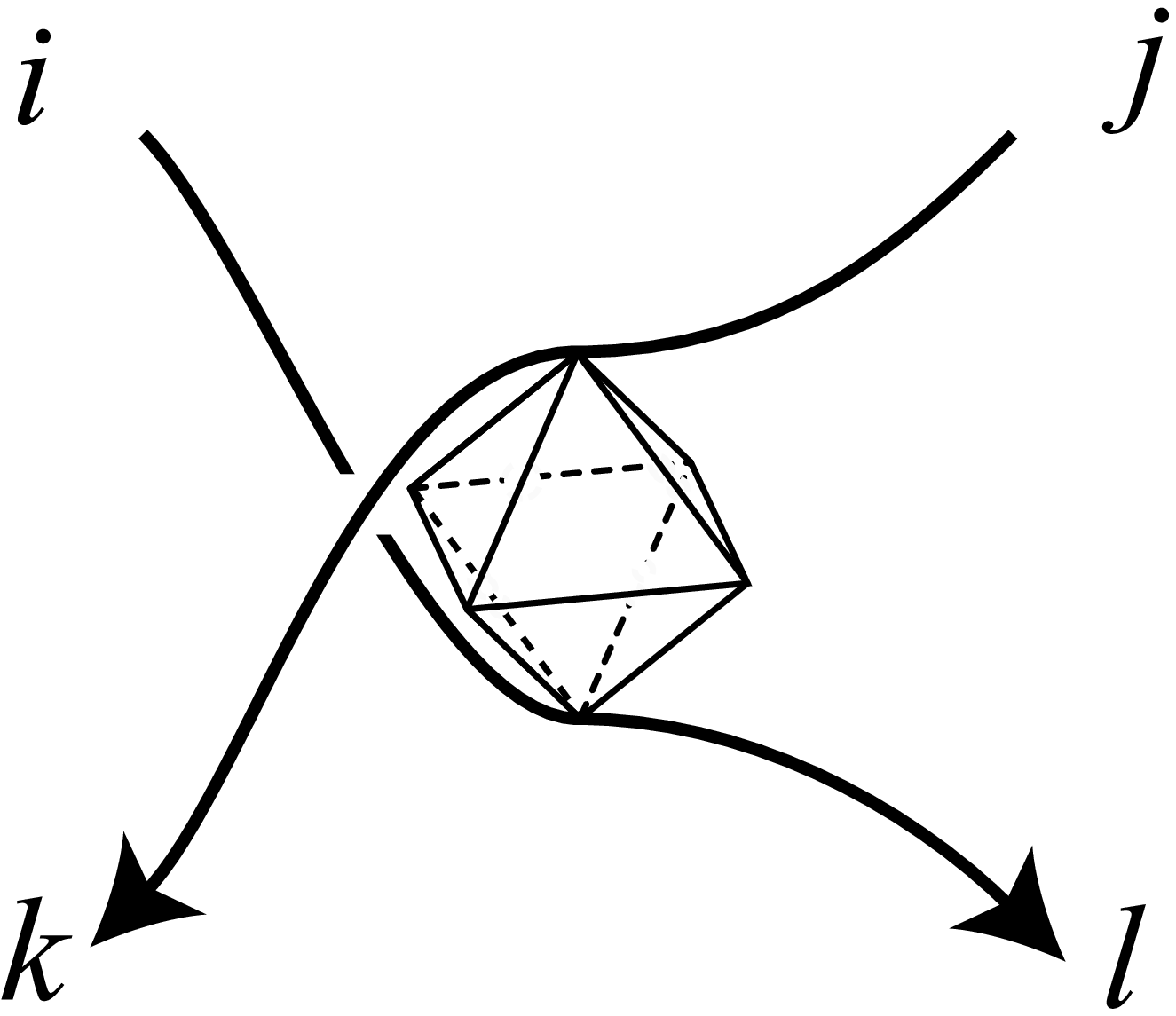}
  \caption{An octahedron put at a crossing}
  \label{fig:crossing_octa}
\end{figure}
Then decompose the octahedron into five tetrahedra as in Figure~\ref{fig:octa_tetra}, where the four of them are decorated with $\xi_N^i$, $\xi_N^{-j}$, $\xi_N^k$ and $\xi_N^{-l}$, respectively and the one in the center is decorated with $\xi_N^m$, where $m:=l-i=j-k$.
Here each truncated tetrahedron is just a topological one with some decoration.
\begin{figure}[h]
  \includegraphics[scale=0.3]{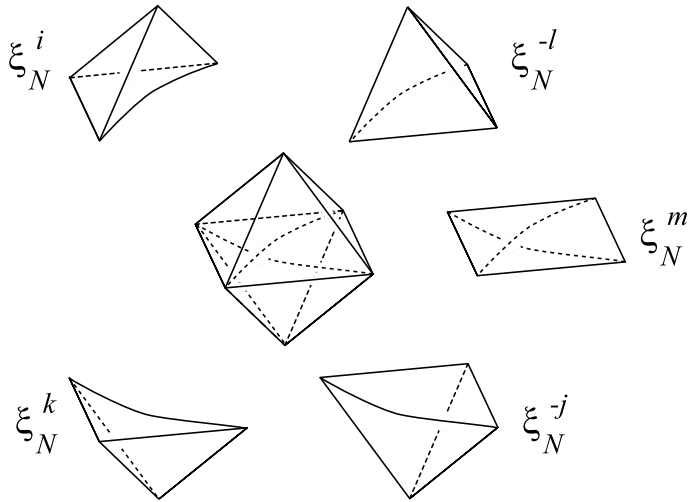}
  \caption{decomposition of the octahedron into five tetrahedra}
  \label{fig:octa_tetra}
\end{figure}
\par
Now only two of the vertices are attached to the knot.
We pull the two of the remaining four vertices to the top ($+\infty$) and the other two to the bottom ($-\infty$) as shown in Figure~\ref{fig:octa_tetra_infinity}.
\begin{figure}[h]
  \includegraphics[scale=0.3]{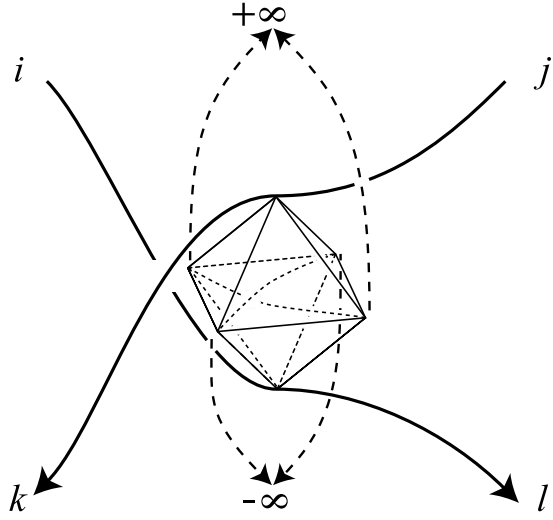}
  \caption{Pull the vertices to $+\infty$ and $-\infty$.}
  \label{fig:octa_tetra_infinity}
\end{figure}
We attach five tetrahedra to every crossing (if the crossing is negative, we change them appropriately) in this way.
At each arc two faces meet, and we paste them together.
Thus we have a decomposition of $S^3\setminus(K\cup\{+\infty,-\infty\})$.
By deforming this decomposition a little we get a decomposition of $S^3\setminus{K}$ by (topological) truncated tetrahedra, decorated with complex numbers $\xi^{\pm i_k}_N$ (k=1,2,\dots,c).
\par
Next we want to regard each tetrahedron as an ideal hyperbolic one.
\par
Recall that when we approximate the summation in \eqref{eq:Jones_sum} by the integral in \eqref{eq:Jones_integral2} we replace $\xi_N^{i_k}$ with a complex variable $z_k$.
Following this we replace the decoration $\xi_N^{i_k}$ for a tetrahedron with a complex number $z_k$.
Then regard the tetrahedron decorated with $z_k$ as an ideal hyperbolic tetrahedron parametrized by $z_k$.
\par
So far this is just formal parametrizations.
We need to choose appropriate values for parameters so that the tetrahedra fit together to provide a complete hyperbolic structure to $S^3\setminus{K}$.
To do this we choose $z_1,z_2,\dots,z_c$ so that:
\begin{itemize}
\item
Around each edge several tetrahedra meet.
To make the knot complement hyperbolic, the sum of these dihedral angles should be $2\pi$,
\item
Even if the knot complement is hyperbolic, the structure may not be complete.
To make it complete, the parameters should be chosen as follows.
\par
Since we truncate the vertices of the tetrahedra, four small triangles appear at the places where the vertices were (see Figure~\ref{fig:tetrahedron_parameter} for the triangle associated with the vertex at infinity).
After pasting these triangles make a torus which can be regarded as the boundary of the regular neighborhood of the knot $K$.
Each triangle has a similarity structure provided by the parameter $z_k$.
We need to make this boundary torus Euclidean.
\end{itemize}
See \cite[Chapter~4]{Thurston:GT3M},\cite[\S~2]{Neumann/Zagier:TOPOL85} for more details.
\par
Surprisingly these conditions are the same as the system of equations \eqref{eq:differential} that we used in the steepest descent method!
Therefore we can expect that a solution $(x_1,x_2,\dots,x_c)$ to \eqref{eq:differential} gives the complete hyperbolic structure.
\par
Then, what does $V(x_1,x_2,\dots,x_c)=2\pi\sqrt{-1}\lim_{N\to\infty}\log\bigl(J_N(K,\xi_N)\bigr)/N$ mean?
\par
Recall the formula \eqref{eq:tetra_vol} and that $V(x_1,x_2,\dots,x_c)$ is a sum of dilogarithm functions and logarithm functions.
Using these facts we can prove
\begin{equation*}
  \Im{V(x_1,x_2,\dots,x_c)}
  =
  \Vol(S^3\setminus{K}),
\end{equation*}
that is,
\begin{equation*}
  \Im
  \left(
    2\pi\sqrt{-1}\lim_{N\to\infty}\frac{\log\bigl(J_N(K,\xi_N)\bigr)}{N}
  \right)
  =
  \Vol(S^3\setminus{K}).
\end{equation*}
So we have proved
\begin{equation*}
  2\pi\lim_{N\to\infty}\frac{\log|J_N(K,\xi_N)|}{N}
  =
  \Vol(S^3\setminus{K}),
\end{equation*}
which is the Volume Conjecture (Conjecture~\ref{conj:VC}).
\section{Generalizations of the Volume Conjecture}
\label{sec:generalization}
In this section we consider generalizations of the Volume Conjecture.
\subsection{Complexification}
In \cite{Thurston:BULAM31982} W.~Thurston pointed out that the Chern--Simons invariant \cite{Chern/Simons:ANNMA21974} can be regarded as an imaginary part of the volume.
Neumann and Zagier gave a precise conjecture \cite[Conjecture, p.~309]{Neumann/Zagier:TOPOL85} which was proved to be true by Yoshida \cite{Yoshida:INVEM85}.
For combinatorial approaches to the Chern--Simons invariant, see \cite{Neumann1992} and \cite{Zickert:DUKMJ2009}.
\par
So it would be natural to drop the absolute value sign of the left hand side of the Volume Conjecture and add the Chern--Simons invariant to the right hand side.
\begin{conjecture}[Complexification of the Volume Conjecture
\cite{Murakami/Murakami/Okamoto/Takata/Yokota:EXPMA02}]
\label{conj:CVC}
If a knot $K$ is hyperbolic, that is, its complement possesses a complete hyperbolic structure, then
\begin{multline*}
  2\pi\lim_{N\to\infty}\frac{\log J_N(K;\exp(2\pi\sqrt{-1}/N)) }{N}
  \equiv
  \Vol(S^3\setminus{K})+\sqrt{-1}\CS(S^3\setminus{K})
  \\
  \pmod{\pi^2\sqrt{-1}\Z},
\end{multline*}
where $\CS$ is the Chern--Simons invariant defined for a three-manifold with torus boundary by Meyerhoff \cite{Meyerhoff:density}.
\end{conjecture}
\begin{remark}
We may regard the left hand side as a definition of the Chern--Simons invariant for non-hyperbolic knots provided that the limit of Conjecture~\ref{conj:CVC} exists.
\end{remark}
\subsection{Deformation of the parameter}
In the Volume Conjecture (Conjecture~\ref{conj:VC}) and its complexification (Conjecture~\ref{conj:CVC}), the (possible) limit corresponds to the complete hyperbolic structure of $S^3\setminus{K}$ for a hyperbolic knot $K$.
As described in \cite[Chapter~4]{Thurston:GT3M} the complete structure can be deformed to incomplete ones.
\par
How can we perform this deformation in the colored Jones polynomial?
If we deform the parameter $2\pi\sqrt{-1}$ in the Volume Conjecture, is the corresponding limits related to incomplete hyperbolic structures?
\par
Let us consider the limit
\begin{equation*}
  \lim_{N\to\infty}\frac{\log J_N\Bigl(K;\exp\bigl((u+2\pi\sqrt{-1})/N\bigr)\Bigr) }{N}.
\end{equation*}
Note that when $u=0$, this limit is considered in the (complexified) Volume Conjecture.
\par
\subsubsection{Figure-eight knot}
Before stating a conjecture for general knots, I will explain what happens in the case of the figure-eight knot.
\begin{theorem}[\cite{Murakami/Yokota:JREIA2007}]\label{thm:MY}
Let $E$ be the figure-eight knot.
There exists a neighborhood $\mathcal{O}\subset\C$ of $0$ such that if $u\in(\mathcal{O}\setminus\pi\sqrt{-1}\Q)\cup\{0\}$, then the following limit exists:
\begin{equation*}
  \lim_{N\to\infty}
  \frac{\log J_N(E;\exp\bigl((u+2\pi\sqrt{-1})/N\bigr))}{N}.
\end{equation*}
Moreover if we put
\begin{align*}
  H(u)
  &:=
  (u+2\pi\sqrt{-1})\times(\text{the limit above}),
  \\
  \intertext{and}
  v
  &:=
  2\dfrac{d\,H(u)}{d\,u}-2\pi\sqrt{-1}
\end{align*}
then we have
\begin{multline*}
  \Vol(E_{u})+\sqrt{-1}\CS(E_{u})
  \\
  \equiv
  -\sqrt{-1}H(u)
  -\pi{u}
  +u\,v\sqrt{-1}/4
  -\pi\kappa(\gamma_u)/2
  \pmod{\pi^2\sqrt{-1}\Z}.
\end{multline*}
\par
Here $E_{u}$ is the closed hyperbolic three-manifold associated with the following representation of $\pi_1\left(S^3\setminus{E}\right)\to SL(2;\C)$:
\begin{equation}\label{eq:meridian_longitude}
\begin{split}
  \mu
  &\mapsto
  \begin{pmatrix}\exp(u/2)&\ast\\0&\exp(-u/2)\end{pmatrix},
  \\[5mm]
  \lambda
  &\mapsto
  \begin{pmatrix}\exp(v/2)&\ast\\0&\exp(-v/2)\end{pmatrix}.
\end{split}
\end{equation}
Here $\mu$ is the meridian of $E$ $($a loop in $S^3\setminus{E}$ that goes around the knot, which generates $H_1(S^3\setminus{E})\cong\Z$$)$, $\lambda$ is the longitude $($a loop in $S^3\setminus{E}$ that goes along the knot such that it is homologous to $0$ in $H_1(S^3\setminus{E})$$)$, and $\gamma_u$ is the loop attached to $S^3\setminus{E}$ when we complete the hyperbolic structure defined by $u$.
We also put $\kappa(\gamma_{u}):=\length(\gamma_u)+\sqrt{-1}\torsion(\gamma_u)$, where $\length(\gamma_u)$ is the length of the attached loop $\gamma_u$, and $\torsion$ is its torsion, which is defined modulo $2\pi$ as the rotation angle when one travels along $\gamma_u$.
See \cite{Neumann/Zagier:TOPOL85} for details $($see also \cite{Murakami:FUNDM2004}$)$.
\end{theorem}
We will give a sketch of the proof in the following two subsections.
\subsubsection{Calculation of the limit}
First we calculate the limit.
Note that here I give just a sketch of the calculation but it can be done rigorously.
For details see \cite{Murakami/Yokota:JREIA2007}.
\par
From \eqref{eq:Jones_fig8} we have
\begin{equation*}
  J_N(E;q)
  =
  \sum_{j=0}^{N-1}
  q^{jN}
  \prod_{k=1}^{j}
  \left(1-q^{-N-k}\right)
  \left(1-q^{-N+k}\right).
\end{equation*}
Put $q:=\exp(\theta/N)$ for $\theta$ near $2\pi\sqrt{-1}$.
If $\theta$ is not a rational multiple of $\pi\sqrt{-1}$ (but it can be $2\pi\sqrt{-1}$), we have
\begin{equation*}
\begin{split}
  &\log
  \left(
    \prod_{k=1}^{j}
    \left(1-q^{-N\pm k}\right)
  \right)
  \\
 =&
  \sum_{k=1}^{j}
  \log
  \big(1-\exp(\pm k\theta/N-\theta)\bigr)
  \\
  \underset{N\to\infty}{\approx}&
  N
  \int_{0}^{j/N}\log(1-\exp(\pm\theta s-\theta))\,ds
  \\
  =&
  \frac{\pm N}{\theta}
  \int_{\exp(-\theta)}^{\exp(\pm j\theta/N-\theta)}
  \frac{\log(1-t)}{t}dt
  \\
  =&
  \frac{\pm N}{\theta}
  \left(
    \Li_2\bigl(\exp(-\theta)\bigr)
    -
    \Li_2\bigl(\exp(\pm j\theta/N-\theta)\bigr)
  \right).
\end{split}
\end{equation*}
So we have
\begin{equation*}
\begin{split}
  &J_N\bigl(E;\exp(\theta/N)\bigr)
  \\
  \underset{N\to\infty}{\approx}&
  \sum_{j=0}^{N-1}
  \exp(j\theta)
  \exp
  \left[
    \frac{N}{\theta}
    \left(
      \Li_2\bigl(\exp(-j\theta/N-\theta)\bigr)
      -
      \Li_2\bigl(\exp(j\theta/N-\theta)\bigr)
    \right)
  \right]
  \\
  =&
  \sum_{j=0}^{N-1}
  \exp
  \left[
    \frac{N}{\theta}
    H\bigl(\exp(j\theta/N),\exp(\theta)\bigr)
  \right]
  \\
  \underset{N\to\infty}{\approx}&
  \int_{C}
  \exp
  \left[
    \frac{N}{\theta}
    H\bigl(x,\exp(\theta)\bigr)
  \right]
  \,dx
\end{split}
\end{equation*}
for a suitable contour $C$.
Here we put
\begin{equation}\label{eq:H}
  H(\zeta,\eta)
  :=
  \Li_2(1/(\zeta\eta))-\Li_2(\zeta/\eta)+\log{\zeta}\log{\eta}.
\end{equation}
\par
To apply the steepest descent method (Theorem~\ref{thm:steepest_descent_method}), we find the maximum of $\Re\left(H\bigl(x,\exp(\theta)\bigr)/\theta\right)$ over $x$.
To do that we will find a solution $y$ to the equation $d\,H\bigl(x,\exp(\theta)\bigr)/d\,x=0$, which is
\begin{equation*}
  \frac{\log\left[\exp(\theta)+\exp(-\theta)-x-x^{-1}\right]}{x}
  =0.
\end{equation*}
We can show that appropriately chosen $y$ gives the maximum and from the steepest descent method we have
\begin{equation*}
  J_N\bigl(E;\exp(\theta/N)\bigr)
  \underset{N\to\infty}{\approx}
  \exp
  \left[
    \frac{N}{\theta}
    H\bigl(y,\exp(\theta)\bigr)
  \right],
\end{equation*}
that is,
\begin{equation}\label{eq:GVC_limit}
  \theta
  \lim_{N\to\infty}
  \frac{J_N\bigl(E;\exp(\theta/N)\bigr)}{N}
  =
  H\bigl(y,\exp(\theta)\bigr),
\end{equation}
where $y$ satisfies
\begin{equation*}
  y+y^{-1}
  =
  \exp(\theta)+\exp(-\theta)-1.
\end{equation*}
\subsubsection{Calculation of the volume}
Next we will relate the limit to the volume of a three-manifold obtained by $S^3\setminus{E}$.
\par
As described in \S~\ref{subsubsec:geometry}, $S^3\setminus{E}$ is obtained by gluing two ideal hyperbolic tetrahedra as in Figure~\ref{fig:fig8_tetra}.
Here we assume that they are parametrized by complex numbers $z$ and $w$.
When $z=w=(1+\sqrt{-3})/2$, $S^3\setminus{E}$ has a complete hyperbolic structure as described in \S~\ref{subsubsec:geometry}.
We assume that the left tetrahedron (with faces labeled with $A$, $B$, $C$ and $D$) and the right tetrahedron (with faces labeled with $A'$, $B'$, $C'$ and $D'$) in Figure~\ref{fig:fig8_tetra} are $\Delta(z)$ and $\Delta(w)$ respectively.
\par
The boundary torus, which is obtained from the shadowed triangles in Figure~\ref{fig:fig8_tetra}, looks like Figure~\ref{fig:torus}.
Here the leftmost triangle is the one in the center of the picture of $\Delta(z)$ and the second leftmost one is the one in the center of the picture of $\Delta(w)$.
Let $\alpha$, $\beta$ and $\gamma$ be the dihedral angle between $B$ and $C$, $A$ and $B$, and $C$ and $A$ respectively.
Let $\alpha'$, $\beta'$ and $\gamma'$ be the dihedral angle between $B'$ and $D'$, $A'$ and $B'$, and $D'$ and $A'$ respectively.
\begin{figure}[h]
  \includegraphics[scale=0.4]{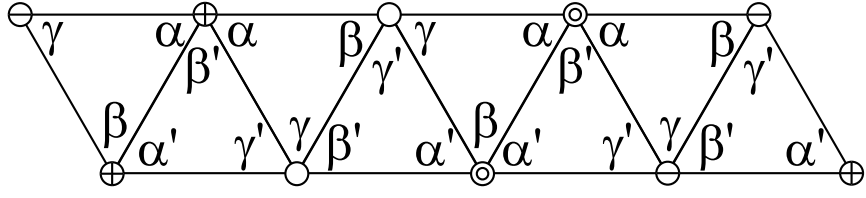}
  \caption{Identifying the sides as indicated by the circles,
  we get a triangulation of the boundary torus.
  Here the single circles denote the arrow head of the single arrow,
  the double circles denote the arrow head of the double arrow,
  the circles with $-$ denote the arrow tail of the single arrow,
  and the circles with $+$ denote the arrow tail of the double arrow
  in Figure~\ref{fig:fig8_tetra}.
  Note that we view this torus from outside of $S^3\setminus{E}$.}
  \label{fig:torus}
\end{figure}
As described in \S~\ref{subsubsec:geometry}, the sum of the dihedral angles around each edge should be $2\pi$.
So from Figure~\ref{fig:torus}, we have
\begin{align*}
  \beta+2\gamma+\beta'+2\gamma'&=2\pi\quad\text{from the single circles},
  \\
  2\alpha+\beta+2\alpha'+\beta'&=2\pi\quad\text{from the double circles},
  \\
  \beta+2\gamma+\beta'+2\gamma'&=2\pi\quad\text{from the circles with $-$},
  \\
  2\alpha+\beta+2\alpha'+\beta'&=2\pi\quad\text{from the circles with $+$},
\end{align*}
which is equivalent to a single equation
\begin{equation*}
  2\alpha+\beta+2\alpha'+\beta'=2\pi
\end{equation*}
since $\alpha+\beta+\gamma=\alpha'+\beta'+\gamma'=\pi$.
\par
Assume that $\alpha=\arg{z}$ and $\alpha'=\arg{w}$.
Since $\beta=\arg(1-1/z)$ and $\beta'=\arg(1-1/w)$ (see Figures~\ref{fig:top_view} and \ref{fig:tetrahedron_parameter}) this turns out to be
\begin{equation*}
  2\arg{z}+\arg(1-1/z)+2\arg{w}+\arg(1-1/w)=2\pi.
\end{equation*}
So we have
\begin{equation}\label{eq:hyperbolicity}
  zw(z-1)(w-1)=1.
\end{equation}
\begin{remark}
This is just a condition that $S^3\setminus{E}$ is hyperbolic.
To make the metric complete we need to add the condition that the upper side and the lower side of the parallelogram in Figure~\ref{fig:torus} are parallel.
\end{remark}
Now we introduce parameters $x$ and $y$ as
\begin{equation}\label{eq:xyzw}
\begin{split}
  x&:=w(1-z),
  \\
  y&:=-zw.
\end{split}
\end{equation}
Note that the following equality holds from \eqref{eq:hyperbolicity}:
\begin{equation}\label{eq:xy}
  y+y^{-1}=x+x^{-1}-1.
\end{equation}
\par
Since $\Delta(z)$ and $\Delta(w)$ can also be parametrized as $\Delta(1-1/z)$ and $\Delta(1-1/w)$, we have
\begin{multline*}
  \Vol(\Delta(z))+\Vol(\Delta(w))
  \\
  =
  \Im\Li_2(w')+\Im\Li_2(z')+\log|w'|\arg(1-w')+\log|z'|\arg(1-z'),
\end{multline*}
from \eqref{eq:tetra_vol}, where we put $z':=1-1/z$ and $w':=1-1/w$.
Using the equation
\begin{equation*}
  \Li_2(z')+\Li_2({z'}^{-1})
  =
  -\frac{\pi^2}{6}-\frac{1}{2}\bigl(\log(-z')\bigr)^2
\end{equation*}
(see for example \cite[\S~1.2]{Zagier:2007}), we have
\begin{equation}\label{eq:Vol_H1}
\begin{split}
  &\Vol(\Delta(z))+\Vol(\Delta(w))
  \\
  =&
  \Im\Li_2(w')-\Im\Li_2({z'}^{-1})
  +\log|w'|\arg(1-w')-\log|{z'}^{-1}|\arg(1-{z'}^{-1})
  \\
  &\text{(since $w'=1/(yx)$ and ${z'}^{-1}=y/x$)}
  \\
  =&
  \Im\Li_2(1/(yx))-\Im\Li_2(y/x)
  +\log|1/(yx)|\arg(1-1/(yx))-\log|y/x|\arg(1-y/x)
  \\
  &\text{(from \eqref{eq:xy})}
  \\
  =&
  \Im\Li_2(1/(yx))-\Im\Li_2(y/x)
  +\log|y|\arg{x}+\log|x|\arg\frac{1-y/x}{1-1/(yx)}
  \\
  &\text{(from \eqref{eq:xyzw} and \eqref{eq:hyperbolicity})}
  \\
  =&
  \Im\Li_2(1/(yx))-\Im\Li_2(y/x)
  +\log|y|\arg{x}+\log|x|\arg\frac{y}{z(z-1)}
  \\
  &\text{(from \eqref{eq:H})}
  \\
  =&
  \Im H(y,x)-\log|x|\arg(z(z-1)).
\end{split}
\end{equation}
\par
Putting
\begin{align*}
  u&:=\log{x}=\log(w(1-z)),
  \\
  v&:=2\log(z(1-z)),
  \\
  H(u)&:=H(y,x),
\end{align*}
we have
\begin{equation}\label{eq:Vol_H2}
  \Vol(\Delta(z))+\Vol(\Delta(w))
  =
  \Im H(u)-\pi\Re{u}-\frac{1}{2}\Re{u}\Im{v}.
\end{equation}
Moreover from \eqref{eq:GVC_limit}, we have
\begin{equation*}
  (u+2\pi\sqrt{-1})
  \lim_{N\to\infty}
  \frac{\log{J_N\bigl(E;\exp((u+2\pi\sqrt{-1})/N)\bigr)}}{N}
  =
  H(u)
\end{equation*}
if we put $\theta=u+2\pi\sqrt{-1}$ since $x=\exp(u)=\exp(\theta)$.
Note that $z$, $w$, $x$, $y$ and $v$ are functions of $u$.
Note also that $v$ is given as
\begin{equation*}
  v
  =
  2\frac{d\,H(u)}{d\,u}-2\pi\sqrt{-1}
\end{equation*}
since $\dfrac{d\,H(u)}{d\,u}=\log(z(z-1))$.
\begin{remark}
We need to be more careful about the arguments of variables.
For details see \cite{Murakami/Yokota:JREIA2007}.
\end{remark}
\par
I will give geometrical interpretation of $u$ and $v$ to relate the term $\Re{u}\Im{v}$ in \eqref{eq:Vol_H2} to the length of $\gamma_u$.
\par
We first calculate $H_1(S^3\setminus{E})=H_1(\Delta(z)\cup\Delta(w))$.
Since the interiors of three-simplices do not matter to the first homology, we can calculate it from the boundary torus, the edges of $\Delta(z)$ and $\Delta(w)$, and the faces $A=A'$, $B=B'$, $C=C'$, and $D=D'$ (see Figures~\ref{fig:fig8_tetra} and \ref{fig:torus}).
From Figures~\ref{fig:fig8_tetra}, \ref{fig:torus} and \ref{fig:torus_edge} one reads
\begin{align*}
  \partial A=\partial A'
  &=
   \raisebox{-1mm}{\rotatebox{45}{$\rightarrow$}}-e_{7}
  +\raisebox{-1mm}{\rotatebox{45}{$\twoheadrightarrow$}}+e_{10}
  -\raisebox{-1mm}{\rotatebox{45}{$\twoheadrightarrow$}}-e_{5},
  \\
  \partial B=\partial B'
  &=
  -\raisebox{-1mm}{\rotatebox{45}{$\rightarrow$}}-e_{11}
  -\raisebox{-1mm}{\rotatebox{45}{$\twoheadrightarrow$}}+e_{6}
  +\raisebox{-1mm}{\rotatebox{45}{$\twoheadrightarrow$}}-e_{9},
  \\
  \partial C=\partial C'
  &=
   \raisebox{-1mm}{\rotatebox{45}{$\twoheadrightarrow$}}+e_{4}
  +\raisebox{-1mm}{\rotatebox{45}{$\rightarrow$}}+e_{8}
  -\raisebox{-1mm}{\rotatebox{45}{$\rightarrow$}}+e_{1},
  \\
  \partial D=\partial D'
  &=
   \raisebox{-1mm}{\rotatebox{45}{$\rightarrow$}}+e_{3}
  -\raisebox{-1mm}{\rotatebox{45}{$\twoheadrightarrow$}}+e_{2}
  -\raisebox{-1mm}{\rotatebox{45}{$\rightarrow$}}+e_{12},
\end{align*}
where \raisebox{-1mm}{\rotatebox{45}{$\rightarrow$}} and \raisebox{-1mm}{\rotatebox{45}{$\twoheadrightarrow$}} mean the single arrowed edge and the double arrowed edge in Figure~\ref{fig:fig8_tetra} respectively, and the $e_i$ are the edges of the boundary torus as indicated in Figure~\ref{fig:torus_edge}.
Since
\begin{align*}
  e_{7}&=e_{6}+e_{2},
  \\
  e_{10}&=e_{6},
  \\
  e_{5}&=e_{1}+e_{6},
  \\
  e_{11}&=e_{4}+e_{12}=e_{4}+e_{6},
  \\
  e_{9}&=e_{3}+e_{10}=e_{3}+e_{6},
  \\
  e_{8}&=e_{6},
  \\
  e_{12}&=e_{6}
\end{align*}
in the first homology group, we have
\begin{align*}
  \raisebox{-1mm}{\rotatebox{45}{$\rightarrow$}}-e_1-e_2-e_6&=0,
  \\
  -\raisebox{-1mm}{\rotatebox{45}{$\rightarrow$}}-e_3-e_4-e_6&=0,
  \\
  \raisebox{-1mm}{\rotatebox{45}{$\twoheadrightarrow$}}+e_1+e_4+e_6&=0,
  \\
  -\raisebox{-1mm}{\rotatebox{45}{$\twoheadrightarrow$}}+e_2+e_3+e_6&=0.
\end{align*}
So if we put $\mu:=e_6$, $\lambda:=e_1+e_2+e_3+e_4+2\mu$, then we see that the first homology group of the boundary torus is generated by $\mu$ and $\lambda$, that $H_1(S^3\setminus{E})\cong\Z$ is generated by $\mu$, and that $\lambda=0$ in $H_1(S^3\setminus{E})$.
Therefore $\mu$ is the meridian and $\lambda$ is the longitude.
\begin{figure}[h]
  \includegraphics[scale=0.4]{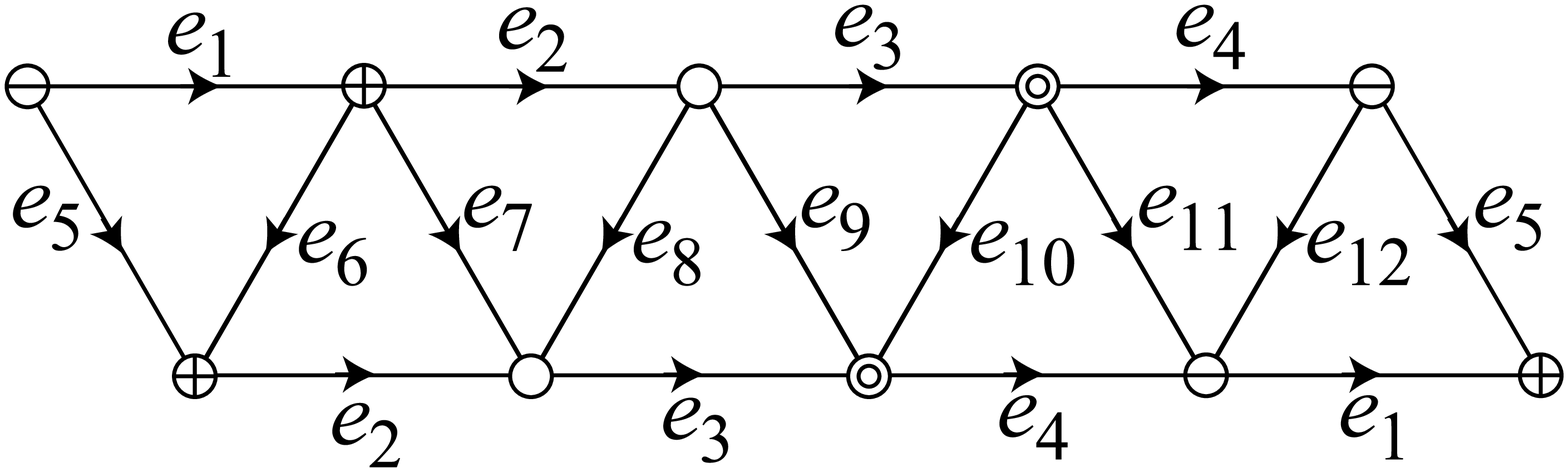}
  \caption{The cycle $e_6$ is the meridian and the cycle
  $e_1+e_6+e_2+e_8+e_3+e_4$ is the longitude.}
  \label{fig:torus_edge}
\end{figure}
\par
Now let us consider the universal cover $\widetilde{S^3\setminus{E}}$ of $S^3\setminus{E}=\Delta(z)\cup\Delta(w)$ which is $\H^3$.
We can construct it by developing $\Delta(z)$ and $\Delta(w)$ isometrically in $\H^3$.
Then each loop in $S^3\setminus{E}$ is regarded as a covering translation of $\widetilde{S^3\setminus{E}}$ and it defines an isometric translation of $\H^3$.
This defines a representation (holonomy representation) of $\pi_1\bigl(S^3\setminus{E})$ at $PSL(2;\C)$.
Taking a lift to $SL(2;\C)$, we can define a representation $\rho\colon\pi_1\bigl(S^3\setminus{E})\to SL(2;\C)$.
\par
We consider how $\rho(\mu)$ and $\rho(\lambda)$ act on $\partial\H^3=S^2=\C\cup\{\infty\}$.
The image of the meridian $\rho(\mu)$ sends the top side to the bottom side.
So it is the composition of a $-\alpha$-rotation around the circle with $+$ in the top (between $e_1$ and $e_2$) and a $\gamma'$-rotation around the single circle in the bottom (between $e_2$ and $e_3$), which means a multiplication by $1/z\times1/(1-w))=w(1-z)$ plus a translation from \eqref{eq:hyperbolicity}.
Similarly $\rho(\lambda)$ acts as a multiplication by $z^2(1-z)^2$ plus a translation.
\par
Therefore $u=\log(w(1-z))$ and $v=2\log(z(1-z))$ can be regarded as the logarithms of the actions by the meridian $\mu$ and the longitude $\lambda$, respectively.
\par
Since the meridian and the longitude commute in $\pi_1(S^3\setminus{E})$, their images can be simultaneously triangularizable.
Recalling that $\mu$ and $\lambda$ define multiplications by $\exp(u)$ and $\exp(v)$ plus translations on $\partial\H^3$, we may assume
\begin{align*}
  \rho(\mu)
  &=
  \begin{pmatrix}
    \exp(u/2)&\ast \\
    0        &\exp(-u/2)
  \end{pmatrix},
  \\
  \rho(\lambda)
  &=
  \begin{pmatrix}
    \exp(v/2)&\ast \\
    0        &\exp(-v/2)
  \end{pmatrix},
\end{align*}
which is \eqref{eq:meridian_longitude}.
This is a geometric interpretation of $u$ and $v$.
\par
Since $u$ determines $z$ and $w$, it defines a hyperbolic structure of $S^3\setminus{E}$ as the union of $\Delta(z)$ and $\Delta(w)$.
When $u\ne0$ this hyperbolic structure is incomplete.
We can complete this incomplete structure  by attaching either a point or a circle.
\par
Since $v$ is not a real multiple of $u$ when $u$ is small, there exists a pair $(p,q)\in\R^2$ such that $pu+qv=2\pi\sqrt{-1}$.
The pair $(p,q)$ is called the generalized Dehn surgery coefficient \cite{Thurston:GT3M}.
If $p$ and $q$ are coprime integers, then the completion is given by attaching a circle $\gamma_{u}$ and the result is a closed hyperbolic three-manifold which we denote by $E_u$.
(For other cases the completion is given by adding either a point or a circle.
In the former case the regular neighborhood of the attached point is a cone over a torus, and in the latter case the regular neighborhood of the attached circle is topologically a solid torus but geometrically the angle around the core is not $2\pi$.)
\par
If $p$ and $q$ are coprime integers, the completion is nothing but the $(p,q)$-Dehn surgery along the knot, that is, we attach a solid torus $D$ to $S^3\setminus\Int(N(E))$ so that the meridian of $D$ coincides with the loop on the boundary of the regular neighborhood $N(E)$ of $E\subset S^3$ presenting $p\mu+q\lambda\in H_1(S^3\setminus\Int(N(E)))$, where $\Int$ denotes the interior (Figure~\ref{fig:Dehn_surgery}).
\begin{figure}[h]
  \includegraphics[scale=0.3]{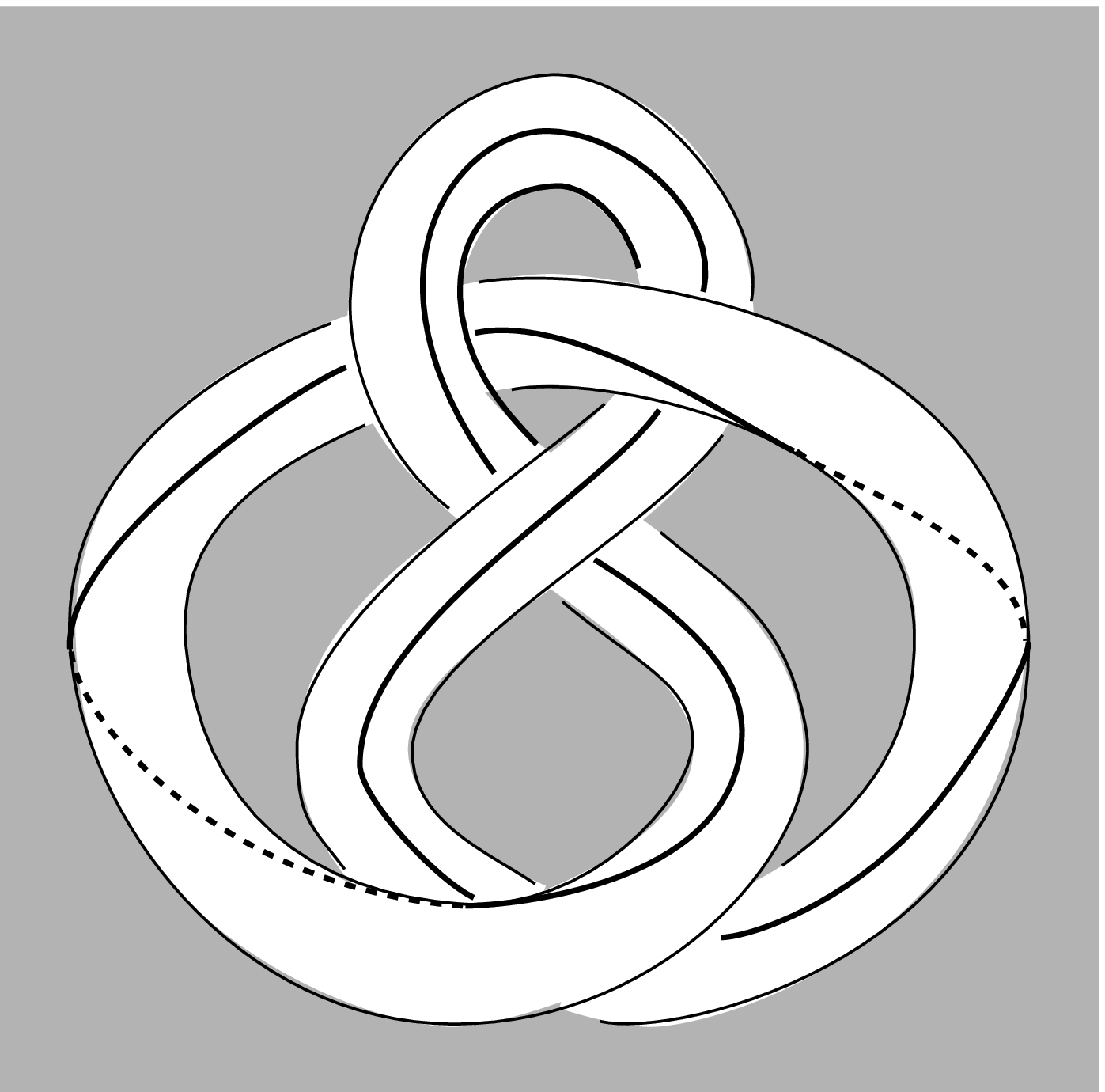}
  \raisebox{20mm}{\quad$\cup$\quad}
  \raisebox{4mm}{\includegraphics[scale=0.3]{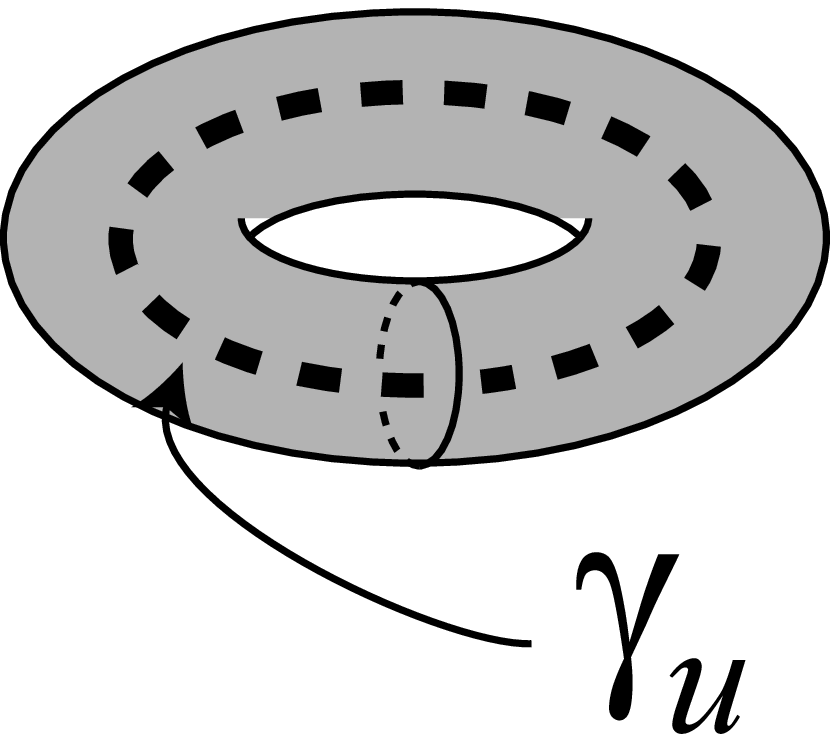}}
  \caption{$(2,1)$-Dehn surgery along the figure-eight knot}
  \label{fig:Dehn_surgery}
\end{figure}
Then the circle $\gamma_u$ can be regarded as the core of $D$.
\par
To complete the proof of Theorem~\ref{thm:MY}, we want to describe the length of the attached circle $\gamma_u$ in terms of $u$ and $v$.
We will show
\begin{equation}\label{eq:length}
  \length{\gamma}_{u}
  =
  -\frac{1}{2\pi}\Im\left(u\overline{v}\right).
\end{equation}
\par
When $u$ is small and non-zero, we can assume that $\exp(u)\ne1$ and $\exp(v)\ne1$.
So we can also assume that $\rho(\mu)$ and $\rho(\lambda)$ are both diagonal.
This means that the image of $\mu$ is a multiplication by $\exp(u)$ and that the image of $\lambda$ is a multiplication by $\exp(v)$ (with no translations).
Note that now $\widetilde{S^3\setminus{E}}$ is identified with $\H^3$ minus the $z$-axis, and the completion is given by adding the $z$-axis.
\par
Since $p$ and $q$ are coprime, we can choose integers $r$ and $s$ so that $ps-qr=1$.
We push $\gamma_u\in D$ to the boundary of the solid torus $\partial D$ and denote the resulting circle by $\tilde{\gamma}_u$.
Then we see that $[\tilde{\gamma}_u]=r\mu+s\lambda\in H_1(\partial(S^3\setminus\Int(N(E)));\Z)$ since the meridian of $D$ is identified with $p\mu+q\lambda$, and the images of the meridian and $\tilde{\gamma}_u$ make a basis of $H_1(\partial(S^3\setminus\Int(N(E)));\Z)$.
\begin{remark}
Even if we use another pair $(r',s')$ such that $ps'-qr'=1$, we get the same manifold.
This is because changing $(r,s)$ corresponds to changing of $\tilde{\gamma}_u\in\partial D$.
Observe that ambiguity of the choice of $\tilde{\gamma}_u$ is given by a twist of $D$ and that it does not matter to the resulting manifold.
\end{remark}
\par
Therefore $\rho(\gamma_u)$ corresponds to a multiplication by $\exp(\pm(ru+sv))$.
This means that if we identify the completion of $\widetilde{S^3\setminus{E}}$ with $\H^3$, a fundamental domain of the lift of $\gamma_u$ is identified with the segment $[1,\exp\bigl(\pm\Re(ru+sv)\bigr)]$ in the $z$-axis.
Since the metric is given by $\sqrt{dx^2+dy^2+dz^2}/z$, the length of $\gamma_u$ is given by
\begin{equation*}
\begin{split}
  \length(\gamma_u)
  &=
  \int_{1}^{\exp\bigl(\pm\Re(ru+sv)\bigr)}\frac{dz}{z}
  \\
  &=
  \pm\Re(ru+sv)
  \\
  &=
  \pm\left(\frac{ps-1}{q}\Re(u)+s\Re(v)\right)
  \\
  &=
  \pm\left(\frac{s}{q}(p\Re(u)+q\Re(v))-\frac{1}{q}\Re(u)\right)
  \\
  &=
  \mp\frac{\Re(u)}{q}
  \\
  &=
  \frac{\mp1}{2\pi}\left(\Re{u}\Im{v}-\Im{u}\Re{v}\right)
  \\
  &=
  \frac{\pm1}{2\pi}\Im(u\overline{v}),
\end{split}
\end{equation*}
where the fourth and the sixth equalities follow from
\begin{equation*}
  \begin{pmatrix}
    \Re(u)&\Re(v) \\
    \Im(u)&\Im(v)
  \end{pmatrix}
  \begin{pmatrix}
    p\\q
  \end{pmatrix}
  =
  \begin{pmatrix}
    0\\2\pi
  \end{pmatrix}.
\end{equation*}
Since $v=u\times\dfrac{|v|^2}{u\overline{v}}$ and the orientation of $(u,v)$ should be positive on $\C$, we see that $\Im(u\overline{v})$ is negative (see \cite{Neumann/Zagier:TOPOL85} for details) and so \eqref{eq:length} follows.
\par
Therefore from \eqref{eq:Vol_H2} we finally have
\begin{equation*}
  \Vol(E_u)
  =
  \Re
  \left(
    -\sqrt{-1}H(u)
    -\pi u
    +\frac{1}{4}uv\sqrt{-1}
    -\frac{\pi}{2}\kappa(\gamma_u)
  \right).
\end{equation*}
\par
The Chern--Simons invariant is obtained by Yoshida's formula \cite{Yoshida:INVEM85}.
See \cite{Murakami/Yokota:JREIA2007} for details.
\subsubsection{General knots}
Here I propose a generalization of the Volume Conjecture for general knots.
\begin{conjecture}[\cite{Murakami:ADVAM22007}]\label{conj:PVC}
For any knot $K$, there exists an open set $U\in\C$ such that if $u\in U$, then the following limit exists:
\begin{equation*}
  \lim_{N\to\infty}
  \frac{\log J_N(K;\exp\bigl((u+2\pi\sqrt{-1})/N\bigr))}{N}.
\end{equation*}
Moreover if we put
\begin{align*}
  H(K;u)&:=(u+2\pi\sqrt{-1})\times(\text{the limit above}),
  \\
  \intertext{and}
  v&:=2\dfrac{d\,H(K;u)}{d\,u}-2\pi\sqrt{-1},
\end{align*}
then we have
\begin{equation*}
  \Vol(K;u)
  =
  \Im{H(K;u)}
  -\pi\Re{u}
  -\frac{1}{2}\Re{u}\Im{v}.
\end{equation*}
Here $\Vol(K;u)$ is the volume function corresponding to the representation of $\pi_1(S^3\setminus{K})$ to $SL(2;\C)$ as in Theorem~$\ref{thm:MY}$.
\end{conjecture}
\begin{remark}
In the case of a hyperbolic knot, we can also propose a similar conjecture with the imaginary part as in the case of the figure-eight knot.
For a general knot, a relation to the Chern--Simons invariant is also expected by using a combinatorial description of the Chern--Simons invariant by Zickert \cite{Zickert:DUKMJ2009}.
\end{remark}
\begin{remark}
Conjecture~\ref{conj:PVC} is known to be true for the figure-eight knot \cite{Murakami/Yokota:JREIA2007} and for torus knots \cite{Murakami:ADVAM22007}.
See also \cite{Murakami:ACTMV2008} and \cite{Hikami/Murakami:Bonn} for a possible relation to the Chern--Simons invariant.
\end{remark}
Finally note that Garoufalidis and L{\^e} proved the following result, which should be compared with Conjecture~\ref{conj:PVC}.
(See also \cite{Murakami:JPJGT2007} for the case of the figure-eight knot.)
\begin{theorem}[S.~Garoufalidis and T.~L{\^e} \cite{Garoufalidis/Le:aMMR}]
For any $K$, there exists $\varepsilon>0$ such that if $|\theta|<\varepsilon$, we have
\begin{equation*}
  \lim_{N\to\infty}
  J_N(K;\exp(\theta/N))
  =
  \frac{1}{\Delta(K;\exp{\theta})},
\end{equation*}
where $\Delta(K;t)$ is the Alexander polynomial of $K$.
\end{theorem}
\bibliography{mrabbrev,hitoshi}
\bibliographystyle{amsplain}
\end{document}